\documentclass[12pt,reqno]{amsart}
\usepackage{amssymb}
\usepackage{mathtools}
\usepackage{txfonts}
\usepackage{eucal}
\usepackage{extarrows}

\usepackage[all]{xy}

\usepackage{tikz}

\usepackage{graphicx}
\usepackage{float}

\usepackage{xspace}

\usepackage{amsmath}

\usepackage[a4paper,body={16.3cm,22.8cm},centering]{geometry}
\usepackage[colorlinks,final,backref=page,hyperindex,pdftex]{hyperref}

\newcommand{\nc}{\newcommand}
\newcommand{\delete}[1]{}

	\nc{\mlabel}[1]{\label{#1}}  
	\nc{\mcite}[1]{\cite{#1}}  
	\nc{\mref}[1]{\ref{#1}}  
	\nc{\mbibitem}[1]{\bibitem{#1}} 

\delete{
	\nc{\mlabel}[1]{\label{#1}  
		{\hfill \hspace{1cm}{\small\tt{{\ }\hfill(#1)}}}}
	\nc{\mcite}[1]{\cite{#1}{\small{\tt{{\ }(#1)}}}}  
	\nc{\mref}[1]{\ref{#1}{{\tt{{\ }(#1)}}}}  
	\nc{\mbibitem}[1]{\bibitem[\bf #1]{#1}} 
}
\newtheorem{theorem}{Theorem}[section]
\newtheorem{prop}[theorem]{Proposition}
\newtheorem{lemma}[theorem]{Lemma}
\newtheorem{coro}[theorem]{Corollary}

\theoremstyle{definition}
\newtheorem{defn}[theorem]{Definition}
\newtheorem{remark}[theorem]{Remark}
\newtheorem{exam}[theorem]{Example}
\newtheorem{prop-def}{Proposition-Definition}[section]

\newcommand\cal[1]{\mathcal{#1}}

\newcommand\alphlist{a,b,c,d,e,f,g,h,i,j,k,l,m,n,o,p,q,r,s,t,u,v,w,x,y,z}
\newcommand\Alphlist{A,B,C,D,E,F,G,H,I,J,K,L,M,N,O,P,Q,R,S,T,U,V,W,X,Y,Z}
\newcommand\getcmds[3]{\expandafter\newcommand\csname #2#1\endcsname{#3{#1}}}
\makeatletter
\@for\x:=\alphlist\do{\expandafter\getcmds\expandafter{\x}{frak}{\mathfrak}}
\@for\x:=\Alphlist\do{\expandafter\getcmds\expandafter{\x}{frak}{\mathfrak}}
\makeatother

\nc{\bfk}{{\bf k}}
\font\cyr=wncyr10

\newfont{\scyr}{wncyr10 scaled 550}
\nc{\sha}{\mbox{\cyr X}}
\nc{\ssha}{\mbox{\bf \scyr X}}

\nc{\id}{\mathrm{id}}
\nc{\Id}{\mathrm{Id}}
\nc{\lbar}[1]{\overline{#1}}
\nc{\ot}{\otimes}
\nc{\dep}{\mathrm{dep}}

\nc{\tred}[1]{\textcolor{red}{#1}} \nc{\tgreen}[1]{\textcolor{green}{#1}}
\nc{\tblue}[1]{\textcolor{blue}{#1}} \nc{\tpurple}[1]{\textcolor{purple}{#1}}

\nc{\li}[1]{\tpurple{\underline{Li:}#1 }}
\nc{\liadd}[1]{\tpurple{#1}}
\nc{\xing}[1]{\tblue{\underline{Xing:}#1 }}
\nc{\dominique}[1]{\tblue{\underline{Dominique: }#1 }}
\nc{\yuan}[1]{\tred{\underline{Yuan:}#1 }}
\nc{\markus}[1]{\tred{\underline{Markus:} #1}}
\nc\hu[1]{\tgreen{\underline{Huhu:}#1}}

\usetikzlibrary{calc}
\newlength\xch
\newsavebox\dbox
\sbox\dbox{\tikz{\fill (0,0) circle (0.05cm);}}
\newif\ifqdd
\newif\ifzdd
\renewcommand{\atopwithdelims}[2]{%
	\genfrac{(}{)}{0pt}{}{#1}{#2}}

\nc{\dnx}{\Delta_n A} \nc{\dx}{\Delta A} \nc{\dgp}{{\rm deg_{P}}}
\nc{\dgt}{{\rm deg_{T}}} \nc{\dg}{{\rm deg}} \nc{\ida}{ID($A$)} \nc{\tu}{\tilde{u}} \nc{\tv}{\tilde{v}}
\nc{\nr}{\calr_n} \nc{\nz}{\calz_n} \nc{\fun}{\cala_{n,d}}
\nc{\fbase}{\calb} \nc{\LF}{\mathrm{RF}} \nc{\FFA}{\mathrm{LF}} \nc{\irr}{\mathrm{Irr}}
\nc{\result}{\bfk\mathrm{Irr}(S_n)}  \nc{\I}{I_{\mathrm{ID},n}^0}
\nc{\nrs}{\calr_n^\star} \nc{\ii}{\mathrm{I}} \nc{\iii}{\mathrm{II}}
\nc{\intl}{{\rm int}}\nc{\ws}[1]{{#1}}\nc{\deleted}[1]{\delete{#1}}\nc{\plas}{placements\xspace}

\nc{\bim}[1]{#1}  \nc{\shaop}{\sha_{\Omega}^{+}}  \nc{\shao}{\sha_{\Omega}}
\nc{\bbim}[2]{#1 #2} \nc{\bbbim}[2]{#1,\, #2} \nc{\RBF}{{\rm RBF}}
\nc{\frb}{F_{\RB}} \nc{\shaf}{\ssha_{\tiny{\Omega}}} \nc{\sham}{\diamond_{\tiny{\Omega}}}
\nc{\lf}{\lfloor} \nc{\rf}{\rfloor} \nc{\shan}{\ssha_{\lambda}}
\nc{\rlex}{{\rm {lex}}} \nc{\bb}{\Box} \nc{\ra}{\rightarrow}
\nc{\e}{{\rm {e}}}
\nc{\DDF}{\mathrm{DD}(X,\,\Omega)}\nc{\DTF}{\mathrm{DT}(X,\,\Omega)} \nc{\DT}{\mathrm{DT}'(\Omega,\,V)}
\nc{\bra}{\mathrm{bra}} \nc{\bre}{\mathrm{bre}}
\nc{\dec}{\mathrm{dec}} \nc{\diamondw}{\diamond_{w}}
\nc{\type}{\mathrm{type}}

\nc\caF[1]{\cal{F}_{#1}(X,\,\Omega)}
\nc\calt{\cal{T}(X,\,\Omega)} \nc\caltn{\cal{T}_n(X,\,\Omega)}
\nc\caltbin{\cal{T}_b(X,\,\Omega)}
\nc\calta{\cal{T}_0(X,\,\Omega)}
\nc\caltb{\cal{T}_1(X,\,\Omega)}
\nc\caltc{\cal{T}_2(X,\,\Omega)}
\nc\caltd{\cal{T}_3(X,\,\Omega)}
\nc\caltm{\cal{T}_m(X,\,\Omega)}
\nc\calf{\cal{F}(X,\,\Omega)}
\nc\fram{\frak{M}(\Omega,\, X)}
\nc\shaw{\sha^{NC}_w(\Omega,\, X)}
\nc\dw{\diamond_w} \nc\dl{\diamond_\ell}
\nc\shal{\sha^{NC}_\ell(X,\, \Omega)} \nc\shav{\sha^{NC}_w(\Omega,\, V)} \nc\shat{\sha^{NC,1}_w(\Omega,\, T^{+}(V))}
\nc{\cfo}{\cal{F}(X,\,\Omega)}

\nc{\lar}{\varinjlim}
\nc\XO{(X,\,\Omega)}
\def\cxo#1#2;{\cal{#1}#2\XO}
\def\cxob#1#2;{\cal{#1}#2_b\XO}
\nc\lrf[2]{B_{#2}^+(#1)}
\nc{\fd}{\mathrm{\text{typed angularly decorated planar rooted trees}}}
\nc{\rb}{\mathrm{RBFWs}} \nc{\dfw}{\mathrm{DFW{(X)}}} \nc{\tfw}{\mathrm{TFW{(X)}}}
\nc{\tfv}{\mathrm{TFW{(V)}}} \nc{\rbf}{\mathrm{RBF}}
\nc{\db}{\mathrm{db}}
\nc{\st}{\mathrm{st}}

\def\Ve#1,#2,#3;{\vee_{#1,\,(#2,\,#3)}}
\def\bigv#1;#2;#3;{\bigvee\nolimits_{#1}^{#2;\,#3}}

\nc{\Irr}{\mathrm{Irr}} \nc{\lc}{\lfloor} \nc{\rc}{\rfloor}
\nc{\rswx}{\frak{M}( \Omega_R\sqcup \Omega_S, X)}
\nc{\rswxs}{\frak{M}^\star( \Omega_R\sqcup \Omega_S, X)}
\nc{\Dl}{\leq_{_{{\rm Dl}}}} \nc{\Dll}{<_{_{{\rm Dl}}}} \nc{\bbs}{\mathbb{S}}
\nc{\orbsa}{$\Omega$-Rota-Baxter system\xspace}
\nc{\orbsas}{$\Omega$-Rota-Baxter systems\xspace}
\nc{\mrbs}{matching Rota-Baxter system\xspace}
\nc{\mrbss}{matching Rota-Baxter systems\xspace}

\nc\prbsla[4]{{R}_{#1}\left(#3\right)R_{#2}\left(#4\right)}
\nc\prbsra[4]{{R}_{#1\rightarrow#2}\left(R_{#1\rhd#2}\left(#3\right)#4\right)+{R}_{#1\leftarrow#2}\left(#3S_{#1\lhd#2}\left(#4\right)\right)}
\nc\prbslb[4]{{S}_{#1}\left(#3\right)S_{#2}\left(#4\right)}
\nc\prbsrb[4]{{S}_{#1\rightarrow#2}\left(R_{#1\rhd#2}\left(#3\right)#4\right)+{S}_{#1\leftarrow#2}\left(#3S_{#1\lhd#2}\left(#4\right)\right)}

\nc\rbsla[4]{\lc #3 \rc ^{R}_{#1} \lc #4 \rc ^R_{#2}}
\nc\rbslb[4]{\lc #3\rc ^{S}_{#1}  \lc #4 \rc ^S_{#2}}
\nc\rbsray[4]{\lc \lc #3 \rc^R_{#1\rhd#2} #4 \rc ^{R}_{#1\rightarrow#2}}
\nc\rbsraz[4]{\lc #3 \lc #4\rc^S_{#1\lhd#2}\rc ^{R}_{#1\leftarrow#2}}
\nc\rbsrby[4]{\lc \lc #3\rc ^R_{#1\rhd#2}#4\rc ^{S}_{#1\rightarrow#2}}
\nc\rbsrbz[4]{\lc #3 \lc #4 \rc ^S_{#1\lhd#2}\rc ^{S}_{#1\leftarrow#2}}
\nc\rbsrac[4]{\lc #3 \lc #4\rc^R_{#1\lhd#2}\rc ^{R}_{#1\leftarrow#2}}

\nc\rbslq[4]{\lc #3 \rc ^{Q}_{#1} \lc #4 \rc ^Q_{#2}}
\nc\rbslt[4]{\lc #3\rc ^{T}_{#1}  \lc #4 \rc ^T_{#2}}
\nc\rbsrqy[4]{\lc \lc #3 \rc^R_{#1\rhd#2} #4 \rc ^{Q}_{#1\rightarrow#2}}
\nc\rbsrqz[4]{\lc #3 \lc #4\rc^S_{#1\lhd#2}\rc ^{Q}_{#1\leftarrow#2}}
\nc\rbsrty[4]{\lc \lc #3\rc ^R_{#1\rhd#2}#4\rc ^{T}_{#1\rightarrow#2}}
\nc\rbsrtz[4]{\lc #3 \lc #4 \rc ^S_{#1\lhd#2}\rc ^{T}_{#1\leftarrow#2}}

\nc{\obr}[1]{\lc #1 \rc_\omega^R} \nc{\obs}[1]{\lc #1 \rc_\omega^S} \nc{\obq}[1]{\lc #1 \rc_\omega^*}
\nc{\obqa}[1]{\lc #1 \rc_\alpha^*} \nc{\obqb}[1]{\lc #1 \rc_\beta^*}
\nc{\obra}[1]{\lc #1 \rc_\alpha^R} \nc{\obrb}[1]{\lc #1 \rc_\beta^R}
\nc{\obsa}[1]{\lc #1 \rc_\alpha^S} \nc{\obsb}[1]{\lc #1 \rc_\beta^S}

\begin{document}

\title[Cohomology Theory of Rota-Baxter family BiHom-$\Omega$-associative algebras]{Cohomology Theory of Rota-Baxter family BiHom-$\Omega$-associative algebras}

\author{Jiaqi Liu
}
\address{School of Mathematics and Statistics, Henan University, Henan, Kaifeng 475004, P.\,R. China}
\email{liujiaqi@henu.edu.cn
}

\author{Chao Song}
\address{School of Mathematical Sciences, East China Normal University, Shanghai 200241, China}
\email{52265500011@stu.ecnu.edu.cn}

\author{Yuanyuan Zhang$^{*}$
}
\footnotetext{* Corresponding author.}
\address{School of Mathematics and Statistics, Henan University, Henan, Kaifeng 475004, P.\,R. China}
\email{zhangyy17@henu.edu.cn
}
%
%

\date{\today}

\begin{abstract}
	In this paper, we first introduce the concept of Rota-Baxter family BiHom-$\Omega$-associative algebras of weight $\lambda$, then we define the cochain complex of BiHom-$\Omega$-associative algebras and verify it via Maurer-Cartan method. Next, we further introduce and study the cohomology theory of Rota-Baxter family BiHom-$\Omega$-associative algebras of weight $\lambda$ and show that this cohomology controls the corresponding deformations. Finally, we study abelian extensions of Rota-Baxter family BiHom-$\Omega$-associative algebras in terms of the second cohomology group.
\end{abstract}

\makeatletter
\@namedef{subjclassname@2020}{\textup{2020} Mathematics Subject Classification}
\makeatother
\subjclass[2020]{
	16W99, 
	16S80, 
	17B38  
}

\keywords{Rota-Baxter family BiHom-$\Omega$-associative algebra, cohomology, Maurer-Cartan element, deformation, abelian extension}

\maketitle

\tableofcontents

\setcounter{section}{0}

\allowdisplaybreaks

\section{Introduction}
The concept of Rota-Baxter algebras was proposed in 1960 by G. Baxter~\cite{baxter} in the probability study about the Spitzer's identity in fluctuation theory. Since then, this concept has appeared in a wide range of areas in mathematics and mathematical physics, such as number theory~\cite{number}, Hopf algebras~\cite{hopf-1,hopf-2} and quantum field theory~\cite{quantum}. The concept of algebras with multiple linear operators was ﬁrst introduced by Kurosch in~\cite{multi-operator}. After that, Guo~\cite{rbf} proposed the concept of Rota-Baxter family algebras, which is a generalization of Rota-Baxter algebras. Then, more and more scholars began to study the family algebra framework, which promoted the development of Rota-Baxter family algebra to a certain extent. In~\cite{BHOalg}, we have given the concept of BiHom-$\Omega$-associative algebras, which is the BiHom-$\Omega$ version of associative algebras. In this paper, we present the concept of Rota-Baxter family BiHom-$\Omega$-associative algebras, which makes the Rota-Baxter family compatible with the BiHom-$\Omega$-associative algebraic structure.

\smallskip

For the classical associative algebras, the cohomology theory has been studied in~\cite{hoch}. Gerstenhaber in~\cite{assoalgdeforma} showed that Hochschild cohomology of associative algebras controls the corresponding formal deformations, and he found that the Hochschild cohomology has a rich structure, which is called the Gerstenhaber algebra~\cite{Gersten}. The Rota-Baxter algebra is an associative algebra equipped with a linear operator satisfying one specific relation, it is natural to consider the cohomology theory of Rota-Baxter algebras when studying the structure of Rota-Baxter algebras, which has been solved by Wang and Zhou in~\cite{RBA}. In recent years, the cohomology theory and deformation theory of a series of algebraic structures related to Rota-Baxter operators have been studied one by one. For example, Das has studied the cohomology of relative Rota-Baxter algebra~\cite{cohomo-relaRBO}, twisted Rota-Baxter operator~\cite{coho-twisted-RBO}, Rota-Baxter family~\cite{cohomo-RBf} and matching relative Rota-Baxter algebra~\cite{coho-matching}. In addition, Zhang ~\cite{RBFassoconformal} studied the cohomology theory of Rota-Baxter family $\Omega$-associative conformal algebras. The deformations and cohomology theory of $\Omega$-Rota-Baxter algebras have been studied by Song in~\cite{L} via constructing the twisted $L_{\infty}[1]$ algebras. Of course, the cohomology theory of BiHom-class algebraic structures has also been studied by many scholars, such as BiHom-associative algebras~\cite{hochsch-BiHom-asso}, BiHom-left-symmetric algebras~\cite{coho-pre-Lie}, and so on.

\smallskip

In order to better study the cohomology of Rota-Baxter family BiHom-$\Omega$-associative algebras, we first describe the cohomology of BiHom-$\Omega$-associative algebras. Similar to~\cite{hochsch-BiHom-asso}, given a vector space $A$, we first construct a non-symmetric operad structure~\cite{operad1,operad}, then we give a graded Lie algebra structure~(Proposition~\ref{graded-Liealg}) from this structure, whose Maurer-Cartan elements are in one-to-one correspondence with the BiHom-$\Omega$-associative algebraic structures on $A$ (Proposition~\ref{prop-MC}). By constructing a new BiHom-$\Omega$-associative algebraic structure with a Rota-Baxter family, we get the cochain complex of Rota-Baxter family on BiHom-$\Omega$-associative algebras, and further, we obtain the cochain complex of Rota-Baxter family BiHom-$\Omega$-associative algebras.

\smallskip
The paper is organized as follows. In Section~\ref{sec2}, we mainly propose the concept of Rota-Baxter family BiHom-$\Omega$-associative algebras and introduce some of its related properties. In Section~\ref{sec3}, we first define the cohomology theory of BiHom-$\Omega$-associative algebras in two ways. One is to define coboundary operator directly, and the other is to characterize cohomology by constructing a graded Lie algebra whose Maurer-Cartan elements correspond to the BiHom-$\Omega$-associative algebraic structures. Then we characterize the cohomology theory of Rota-Baxter family BiHom-$\Omega$-associative algebras by studying the cohomology of BiHom-$\Omega$-associative algebras. In Section~\ref{sec4}, we study the deformations of BiHom-$\Omega$-associative algebras and Rota-Baxter family BiHom-$\Omega$-associative algebras, respectively. We interpret them via the lower degree cohomology groups. In Section~\ref{sec5}, we study the abelian extensions of Rota-Baxter family BiHom-$\Omega$-associative algebras and show that they are classified by the second cohomology.

\smallskip
\textbf{Notation.} Throughout this paper, we fix a commutative unitary ring \bfk, which will be the base
ring of all algebras as well as linear maps. By an algebra we mean a
unitary associative noncommutative algebra, unless the contrary is specified. Denote by $\Omega $ a semigroup, unless otherwise specified. For the composition of two maps $ p $ and $ q $, we will write either $ p\circ q $ or simply $ pq $ without causing confusion.

\bigskip

\section{Rota-Baxter family BiHom-$\Omega$-associative algebras}\label{sec2}
\setcounter{equation}{0}
In this section, we first recall the concept of BiHom-$\Omega$-associative algebras and study some related properties. Then we introduce the definition of Rota-Baxter family BiHom-$\Omega$-associative algebras. In the end, we obtain an important result (Proposition~\ref{defrhd}), which prepares for the study of cohomology theory in Section~\ref{sec-coho-RBf}.

\subsection{BiHom-$\Omega$-associative algebras}
In this subsection, we first give the definition of bimodules over the BiHom-$\Omega$-associative algebras. Then we introduce the concept of the semi-direct product BiHom-$\Omega$-associative algebras and give a corresponding example. Finally, we introduce the definition and property of bimodule algebras under the BiHom-$\Omega$-associative version. Now, let's recall the definition of BiHom-$\Omega$-associative algebras, as a generalization of BiHom-associative algebras~\cite{gmmp}.

\begin{defn}\label{def-BHO}\cite{BHOalg}
	A \textbf{BiHom-$\Omega$-associative algebra} is a
	4-tuple $( A, \cdot_{\alpha,\,\beta},p^A_{\omega},q^A_{\omega} )_{\alpha,\,\beta,\,\omega \in \Omega} $ consisting of a vector space $A$, two commuting families of linear maps $( p^A_{\omega})_{\omega\in\Omega},( q^A_{\omega})_{\omega\in \Omega} :A\rightarrow A$ and  a family of bilinear maps $(\cdot_{\alpha,\,\beta})_{\alpha,\,\beta \in \Omega}:A\otimes A\rightarrow A$ satisfying
	\begin{gather}
		p^A_{\alpha\,\beta} (x \cdot_{\alpha,\,\beta }y) =p^A_{\alpha} (x)\cdot_{\alpha,\,\beta}p^A_{\beta} (y) \text{ and }q^A_{\alpha\,\beta} (x\cdot_{\alpha,\,\beta}y)=q^A_{\alpha} (x)\cdot_{\alpha,\,\beta}q^A_{\beta} (y)
		,\quad \text{(multiplicativity)}  \label{eqalfabeta} \\
		p^A_{\alpha}(x)\cdot_{\alpha,\,\beta\gamma}(y\cdot_{\beta,\,\gamma}z)=(x\cdot_{\alpha,\,\beta}y)\cdot_{\alpha\beta,\,\gamma}q^A_{\gamma}(z),\quad \text{(BiHom-$\Omega$-associativity)}
		\label{eqasso}
	\end{gather}
	for all $x,y, z\in A,\,\alpha,\,\beta,\,\gamma \in \Omega.$	
	The maps $(p^A_{\omega})_{\omega\in \Omega} $ and $(q^A_{\omega})_{\omega\in \Omega} $ (in this order) are called the structure maps of $A$.

	Let $(A, \cdot_{\alpha,\,\beta}, p^A_{\omega},q^A_{\omega})_{\alpha,\,\beta,\,\omega \in \Omega}$ and $(A', \cdot'_{\alpha,\,\beta}, p^{A'}_{\omega},q^{A'}_{\omega})_{\alpha,\,\beta,\,\omega \in \Omega}$ be two BiHom-$\Omega$-associative algebras. A family of linear maps $(f_{\alpha})_{\alpha \in \Omega}:A\rightarrow A'$ is called a $\mathbf{BiHom}$-$\mathbf{\Omega}$-$\mathbf{associative \, algebra}$ $\mathbf{ homomorphism}$ if 
	\[p^{A'}_{\alpha}\circ f_{\alpha}=f_{\alpha}\circ p^{A}_{\alpha} ,\quad q^{A'}_{\alpha}\circ f_{\alpha}=f_{\alpha}\circ q^{A}_{\alpha},\]
	\begin{align}
		f_{\alpha\,\beta}(x\cdot_{\alpha,\,\beta}y)=f_{\alpha}(x)\cdot_{\alpha,\,\beta}'f_{\beta}(y), \label{RBfBHOalghomo}
	\end{align}
	for all $ x,y\in A,\;\alpha,\,\beta\in \Omega.$
\end{defn}

\begin{defn} \label{bimodule}
	Let $(A, \cdot_{\alpha,\,\beta}, p^A_{\omega}, q^A_{\omega})_{\alpha,\,\beta,\,\omega \in \Omega} $ be a BiHom-$\Omega$-associative algebra, $M$
	be a vector space and $( p_{\omega}^{M})_{\omega\in\Omega}, (q_{\omega}^{M})_{\omega\in\Omega}:M \rightarrow M$ be two commuting families of linear maps.
	\begin{enumerate}
		\item A $\bf{left \, module}$ over $A$ on $M$ consists of $(M, p_{\omega}^{M}, q_{\omega}^{M})_{\omega\in\Omega}$ together with a family of bilinear maps $( \rhd_{\alpha,\,\beta})_{\alpha,\,\beta \in \Omega}:A\otimes M\rightarrow M $ such that
		\begin{eqnarray}
			&&p_{\alpha\,\beta}^{M}(x\rhd_{\alpha,\,\beta} m)=p^A_{\alpha}(x)\rhd_{\alpha,\,\beta} p_{\beta}^{M}(m),\label{lmod-1}\\
			&&q_{\alpha\,\beta}^{M}(x\rhd_{\alpha,\,\beta} m)=q^A_{\alpha}(x)\rhd_{\alpha,\,\beta} q_{\beta}^{M}(m),\label{lmod-2}\\
			&&p^A_{\alpha}(x)\rhd_{\alpha,\,\beta \,\gamma} (x^{\prime }\rhd_{\beta,\,\gamma} m)=(x\cdot_{\alpha,\,\beta}x^{\prime })\rhd_{\alpha\beta,\,\gamma} q_{\gamma}^{M}(m),
			\label{lmod}
		\end{eqnarray}
		for all $ x, x'\in A,\, m\in M,\,\alpha,\,\beta,\,\gamma\in\Omega.$
		
		\item A $\bf{right \, module}$ over $A$ on $M$ consists of $(M, p_{\omega}^{M}, q_{\omega}^{M})_{\omega\in \Omega}$ together with a family of bilinear maps $( \lhd_{\alpha,\,\beta})_{\alpha,\,\beta \in \Omega}:M\otimes A\rightarrow M $ such that
		\begin{eqnarray}
			&&p_{\alpha\,\beta}^{M}(m\lhd_{\alpha,\,\beta} x)=p_{\alpha}^{M}(m)\lhd_{\alpha,\,\beta} p^A_{\beta}(x),\label{rmod-1}\\
			&&q_{\alpha\,\beta}^{M}(m\lhd_{\alpha,\,\beta} x)=q_{\alpha}^{M}(m)\lhd_{\alpha,\,\beta} q^A_{\beta}(x),\label{rmod-2}\\
			&&p_{\alpha}^{M}(m)\lhd_{\alpha,\,\beta\,\gamma} (x\cdot_{\beta,\,\gamma}x')=(m\lhd_{\alpha,\,\beta} x)\lhd_{\alpha\,\beta,\,\gamma} q^A_{\gamma}(x'),
			\label{rmod}
		\end{eqnarray}
		for all $ x, x'\in A,\, m\in M,\,\alpha,\,\beta,\,\gamma\in\Omega.$
		
		\item Let $(M, \rhd_{\alpha,\,\beta}, p_{\omega}^{M}, q_{\omega}^{M})_{\alpha,\,\beta,\,\omega\in \Omega}$ be
		a left module over $A$ and $(M,\lhd_{\alpha,\,\beta}, p_{\omega}^{M}, q_{\omega}^{M})_{\alpha,\,\beta,\,\omega\in \Omega}$ be a right module over $A$. We call $(M,\rhd_{\alpha,\,\beta},\lhd_{\alpha,\,\beta}, p_{\omega}^{M}, q_{\omega}^{M})_{\alpha,\,\beta,\,\omega\in \Omega}$ a $\bf{bimodule}$ over $A$ if
		\begin{eqnarray}
			&&p^A_{\alpha}(x)\rhd_{\alpha,\,\beta\,\gamma} (m\lhd_{\beta,\,\gamma} x')=(x\rhd_{\alpha,\,\beta} m)\lhd_{\alpha\,\beta,\,\gamma} q^A_{\gamma}(x'), \label{bimod}
		\end{eqnarray}
		for all $x, x'\in A,\, m\in M,\,\alpha,\,\beta,\,\gamma\in\Omega.$
	\end{enumerate}
	In particular, we call $ ( A,\cdot_{\alpha,\,\beta},p^A_{\omega},q^A_{\omega} )_{\alpha,\,\beta,\,\omega \in \Omega} $ the \textbf{regular bimodule} over $A$.
\end{defn}

Let $(A, \cdot_{\alpha,\,\beta}, p_{\omega}^{A}, q_{\omega}^{A})_{\alpha,\,\beta,\,\omega\in \Omega}$ be a BiHom-$\Omega$-associative algebra and let $M$ be a vector space with two commuting families of linear maps $(p_{\omega}^{M})_{\omega\in\Omega}, (q_{\omega}^{M})_{\omega\in\Omega}
:M\rightarrow M$. There are two families of bilinear maps
\[(\rhd_{\alpha,\,\beta})_{\alpha,\,\beta\in\Omega}:A\otimes M\rightarrow M, \; x\otimes m\mapsto x\rhd_{\alpha,\,\beta} m,\]
\[(\lhd_{\alpha,\,\beta})_{\alpha,\,\beta\in\Omega}:M\otimes A\rightarrow M, \; m\otimes x\mapsto m\lhd_{\alpha,\,\beta} x.\]
We define the multiplication and structure maps on direct sum space $A\oplus M$ by
\begin{align}
	(x, m)\circ_{\alpha,\,\beta}(x', m')&:=( x \cdot_{\alpha,\,\beta} x' , x\rhd_{\alpha,\,\beta} m'+m\lhd _{\alpha,\,\beta}x'), \label{semi-product-1}\\
	p_{\alpha} (x, m)&:=\big(p_{\alpha}^{A}(x), p_{\alpha}^{M}(m)\big),\label{semi-product-2}\\
	q _{\alpha}(x, m)&:=\big(q_{\alpha}^{A}(x), q_{\alpha}^{M}(m)\big),\label{semi-product-3}
\end{align}
for all $(x,m),\, (x',m')\in A\oplus M, \,\alpha,\,\beta \in \Omega.$ Then $A\ltimes M :=( A\oplus M, \circ_{\alpha,\,\beta}, p_{\omega}, q_{\omega})_{\alpha,\,\beta,\,\omega\in \Omega}$ is a
BiHom-$\Omega$-associative algebra if and only if $(M, \rhd_{\alpha,\,\beta}, \lhd_{\alpha,\,\beta}, p_{\omega}^{M}, q_{\omega}^{M})_{\alpha,\,\beta,\,\omega \in \Omega}$ is a bimodule over BiHom-$\Omega$-associative algebra $(A, \cdot_{\alpha,\,\beta}, p_{\omega}^{A}, q_{\omega}^{A})_{\alpha,\,\beta,\,\omega \in \Omega}$. Moreover, $A\ltimes M$ is called the \textbf{semi-direct product BiHom-$\Omega$-associative algebra} of $A$ with $M$.

 In~\cite[Example~2.5]{BHOalg}, we already introduced that $(A=\mathbf{k}\lbrace e_{1},e_{2}\rbrace, \cdot_{\alpha,\,\beta},p^A_{\omega},q^A_{\omega})_{\alpha,\,\beta,\,\omega\in \Omega}$ is a BiHom-$\Omega$-associative algebra and the operations on $A$ are defined by
   \[(k_{1}e_{1}+k_{2}e_{2})\cdot_{\alpha,\,\beta}(k_{3}e_{1}+k_{4}e_{2}):=k_{1}(k_{3}+k_{4})c(\alpha,\,\beta)e_{1}+k_{2}(k_{3}+k_{4})c(\alpha,\,\beta)e_{2},\]
   \[p^A_{\alpha}(k_{1}e_{1}+k_{2}e_{2}):=k_{1}(\alpha\rightthreetimes 1_{k})e_{1}+k_{2}(\alpha\rightthreetimes 1_{k})e_{2},\]
   \[q^A_{\alpha}(k_{1}e_{1}+k_{2}e_{2}):=(k_{1}+k_{2})(1_{k}\leftthreetimes\alpha)e_{1},
   \,\text{for all } k_{1}e_{1}+k_{2}e_{2},\,k_{3}e_{1}+k_{4}e_{2}\in A,\,\alpha,\,\beta\in \Omega,\]
   where the maps $ c: \Omega \times \Omega\rightarrow \mathbf{k},\;\rightthreetimes : \Omega \times \mathbf{k}\rightarrow \mathbf{k}\;\text{and}\;\leftthreetimes: \mathbf{k}\times \Omega \rightarrow \mathbf{k}$ satisfy
   \[\alpha\,\beta\rightthreetimes 1_{k}=(\alpha\rightthreetimes 1_{k})(\beta\rightthreetimes1_{k}),\quad 1_{k}\leftthreetimes\alpha\,\beta=(1_{k}\leftthreetimes\alpha)(1_{k}\leftthreetimes\beta),\]
   \[c(\alpha,\,\beta)(1_{k}\leftthreetimes \gamma)c(\alpha\,\beta,\,\gamma)=c(\alpha,\,\beta\,\gamma)(\alpha\rightthreetimes1_{k})c(\beta,\,\gamma),\]
and $1_{k}$ is the unit of $\mathbf{k}$. Based on this example, we give the example of semi-direct product BiHom-$\Omega$-associative algebras as follows.
\begin{exam}
Let $M=\mathbf{k}\lbrace e_{3}\rbrace$ be a vector space. If we define
	\[\rhd_{\alpha,\,\beta}: A\times M\rightarrow M,\quad (k_{1}e_{1}+k_{2}e_{2})\rhd_{\alpha,\,\beta}k_{3}e_{3}:=k_{3}(k_{1}+k_{2})c(\alpha,\,\beta)e_{3},\]
	\[\lhd_{\alpha,\,\beta}: M\times A\rightarrow M,\quad k_{3}e_{3}\lhd_{\alpha,\,\beta}(k_{1}e_{1}+k_{2}e_{2}):=k_{3}(k_{1}+k_{2})c(\alpha,\,\beta)e_{3},\]
	\[p^M_{\alpha}(k_{3}e_{3}):=k_{3}(\alpha\rightthreetimes 1_{k})e_{3},\quad q^M_{\alpha}(k_{3}e_{3}):=k_{3}(1_{k}\leftthreetimes\alpha)e_{3},\]
	for all $k_{1}e_{1}+k_{2}e_{2},\,k_{3}e_{1}+k_{4}e_{2}\in A,\,k_{3}e_{3}\in M,\,\alpha,\,\beta\in \Omega.$ Then $(M=\mathbf{k}\lbrace e_{3}\rbrace,\rhd_{\alpha,\,\beta},\lhd_{\alpha,\,\beta},p^M_{\omega},q^M_{\omega})_{\alpha,\,\beta,\,\omega\in \Omega}$ is a bimodule over the BiHom-$\Omega$-associative algebra $(A=\mathbf{k}\lbrace e_{1},e_{2}\rbrace,\cdot_{\alpha,\,\beta},p^A_{\omega},q^A_{\omega})_{\alpha,\,\beta,\,\omega\in\Omega}$. Moreover, $A\ltimes M$ is a semi-direct product BiHom-$\Omega$-associative algebra of $A$ with bimodule $M$, where operations $(\circ_{\alpha,\,\beta})_{\alpha,\,\beta\in\Omega}$, $(p_{\omega})_{\omega\in\Omega}$, $(q_{\omega})_{\omega\in\Omega}$ are defined by Eqs.~(\ref{semi-product-1})-(\ref{semi-product-3}).
\end{exam}

Inspired by~\cite{BiHom-NS,opv}, we introduce the concept of bimodule algebras over BiHom-$\Omega$-associative algebras. Given a family of bilinear maps $(\bullet_{\alpha,\,\beta})_{\alpha,\,\beta\in\Omega}
:M\otimes M\rightarrow M$, we have the following definition.
\begin{defn}\label{bimalg}
	The 6-tuple $(M, \bullet_{\alpha,\,\beta}, \rhd_{\alpha,\,\beta}, \lhd_{\alpha,\,\beta}, p_{\omega}^{M}, q_{\omega}^{M})_{\alpha,\,\beta,\,\omega\in\Omega}$ is called a $\mathbf{bimodule \, algebra}$ over the BiHom-$\Omega$-associative algebra $ (A, \cdot_{\alpha,\,\beta}, p_{\omega}^{A}, q_{\omega}^{A})_{\alpha,\,\beta,\,\omega \in \Omega} $ if BiHom-$\Omega$-associative algebra $(A\oplus M, \ast_{\alpha,\,\beta}, p_{\omega} , q_{\omega} )_{\alpha,\,\beta,\,\omega \in \Omega}$ satisfies
	\begin{align*}
		&\quad p_{\alpha}(x, m)=\big(p_{\alpha}^{A}(x), p_{\alpha}^{M}(m)\big),\quad q_{\alpha}(x, m)=\big(q_{\alpha}^{A}(x), q_{\alpha}^{M}(m)\big),\\
&(x,m)\ast_{\alpha,\,\beta}(x',m')=(x\cdot_{\alpha,\,\beta}x', x \rhd_{\alpha,\,\beta} m'+m\lhd_{\alpha,\,\beta} x'+m\bullet_{\alpha,\,\beta} m'),
\end{align*}
	for all $(x,m),\, (x',m')\in A\oplus M,\,\alpha,\,\beta \in \Omega$.
\end{defn}

The following statement shows that a bimodule algebra defined by Definition~\ref{bimalg} is a generalization of~\cite[Definition 2.3]{pacific} and~\cite[Proposition~2.6]{BiHom-NS}.

\begin{prop}\label{bimalgbim}
	The 6-tuple	$ ( M, \bullet_{\alpha,\,\beta}, \rhd_{\alpha,\,\beta}, \lhd_{\alpha,\,\beta}, p_{\omega}^{M}, q_{\omega}^{M} )_{\alpha,\,\beta,\,\omega\in \Omega} $ is a bimodule algebra over BiHom-$\Omega$-associative algebra $ (A, \cdot_{\alpha,\,\beta}, p_{\omega}^{A}, q_{\omega}^{A})_{\alpha,\,\beta,\,\omega \in \Omega} $ if and only if
	$ ( M, \rhd_{\alpha,\,\beta}, \lhd_{\alpha,\,\beta}, p_{\omega}^{M}, q_{\omega}^{M})_{\alpha,\,\beta,\,\omega \in \Omega} $ is a bimodule over $A$ and $(M, \bullet_{\alpha,\,\beta}, p_{\omega}^{M}, q_{\omega}^{M})_{\alpha,\,\beta,\,\omega\in \Omega}$ is a BiHom-$\Omega$-associative algebra satisfying
	\begin{eqnarray}
		&&p_{\alpha}^{A}(x)\rhd_{\alpha,\,\beta\,\gamma} (m\bullet_{\beta,\,\gamma} m')=(x\rhd_{ \alpha,\,\beta} m)\bullet_{\alpha\,\beta,\,\gamma} q_{\gamma}^{M}(m'), \label{extra1} \\
		&&p_{\alpha}^{M}(m)\bullet_{\alpha,\,\beta\,\gamma} (m'\lhd_{\beta,\,\gamma} x)=(m\bullet_{\alpha,\,\beta} m')\lhd_{\alpha\,\beta,\,\gamma} q_{\gamma}^{A}(x), \label{extra2} \\
		&&p_{\alpha}^{M}(m)\bullet_{\alpha,\,\beta\,\gamma} (x\rhd_{\beta,\,\gamma} m')=(m\lhd_{\alpha,\,\beta} x)\bullet_{\alpha\beta,\,\gamma} q_{\gamma}^{M}(m'), \label{extra3}
	\end{eqnarray}
	for all $x\in A,\, m, m'\in M,\,\alpha,\,\beta,\,\gamma \in \Omega$.
\end{prop}
\begin{proof}
	According to Definition~\ref{bimalg}, we only need to verify that $(A\oplus M,\ast_{\alpha,\,\beta},p_{\omega},q_{\omega})_{\alpha,\,\beta,\,\omega\in\Omega}$ is a BiHom-$\Omega$-associative algebra if and only if Eqs.~(\ref{lmod-1})-(\ref{bimod}), (\ref{extra1})-(\ref{extra3}) hold and $(M, \bullet_{\alpha,\,\beta}, p_{\omega}^{M}, q_{\omega}^{M})_{\alpha,\,\beta,\,\omega\in \Omega}$ satisfies Eqs.~(\ref{eqalfabeta})-(\ref{eqasso}). For any $ (x,m),\,(x',m'),\,(x'',m'')\in A\oplus M$  and $\alpha,\,\beta,\,\gamma\in \Omega,$ the BiHom-$\Omega$-associativity for $A\oplus M$ is equivalent to
	\begin{align*}
		\Big( &p_{\alpha}^{A}( x)\cdot_{\alpha,\,\beta\gamma}( x'\cdot_{\beta,\,\gamma}x''),p_{\alpha}^{A}( x)\rhd_{\alpha,\,\beta\,\gamma}( x'\rhd_{\beta,\,\gamma}m'')+p_{\alpha}^{A}( x)\rhd_{\alpha,\,\beta\,\gamma}( m'\lhd_{\beta,\,\gamma}x'')\\
		&+p_{\alpha}^{A}( x)\rhd_{\alpha,\,\beta\,\gamma}( m'\bullet_{\beta,\,\gamma} m'')+p_{\alpha}^{M}( m)\lhd_{\alpha,\,\beta\,\gamma}( x'\cdot_{\beta,\,\gamma}x'')+p_{\alpha}^{M}( m)\bullet_{\alpha,\,\beta\gamma}( x'\rhd_{\beta,\,\gamma}m'')\\
		&+p_{\alpha}^{M}( m)\bullet_{\alpha,\,\beta\gamma}( m'\lhd_{\beta,\,\gamma}x'') +p_{\alpha}^{M}( m)\bullet_{\alpha,\,\beta\gamma}( m'\bullet_{\beta,\,\gamma}m''\Big)\\
		=&\Big( ( x\cdot_{\alpha,\,\beta}x')\cdot_{\alpha\beta,\,\gamma}q_{\gamma}^{A}( x''),( x\cdot_{\alpha,\,\beta}x')\rhd_{\alpha\beta,\,\gamma}q_{\gamma}^{M}( m'')+( x\rhd_{\alpha,\,\beta}m')\lhd_{\alpha\,\beta,\,\gamma}q_{\gamma}^{A}( x'')\\
		&+( m\lhd_{\alpha,\,\beta}x')\lhd_{\alpha\,\beta,\,\gamma}q_{\gamma}^{A}( x'') + ( m\bullet_{\alpha,\,\beta}m')\lhd_{\alpha\,\beta,\,\gamma}q_{\gamma}^{A}( x'')+( x\rhd_{\alpha,\,\beta}m')\bullet_{\alpha\beta,\,\gamma}q_{\gamma}^{M}( m'')\\
		&+ (m\lhd_{\alpha,\,\beta}x')\bullet_{\alpha\,\beta,\,\gamma}q_{\gamma}^{M}(m'')+( m\bullet_{\alpha,\,\beta}m')\bullet_{\alpha\beta,\,\gamma}q_{\gamma}^{M}( m'')\Big).
	\end{align*} 	
We obtain that Eqs.~(\ref{lmod}), (\ref{rmod})-(\ref{bimod}), (\ref{extra1})-(\ref{extra3}) hold and $(M, \bullet_{\alpha,\,\beta}, p_{\omega}^{M}, q_{\omega}^{M})_{\alpha,\,\beta,\,\omega\in \Omega}$ satisfies Eq.~(\ref{eqasso}) by taking $ m=m'=x''=0 $, $ x=m'=m''=0 $, $m=x'=m''=0$, $m=x'=x''=0$, $x=x'=m''=0$, $x=m'=x''=0$ and $x=x'=x''=0$, respectively. Similarly, we get that the multiplicativity of $A\oplus M$ is equivalent to Eqs.~(\ref{lmod-1})-(\ref{lmod-2}), (\ref{rmod-1})-(\ref{rmod-2}) hold and $(M, \bullet_{\alpha,\,\beta}, p_{\omega}^{M}, q_{\omega}^{M})_{\alpha,\,\beta,\,\omega\in \Omega}$ satisfies Eq.~(\ref{eqalfabeta}). This completes the proof.
\end{proof}

\subsection{Rota-Baxter family BiHom-$\Omega$-associative algebra of weight $\lambda$}
 In this subsection, we first give the concept of Rota-Baxter family BiHom-$\Omega$-associative algebras of weight $\lambda$. Then, we introduce the definition of Rota-Baxter family BiHom-$\Omega$-bimodules. Finally, we construct a new bimodule structure from a Rota-Baxter family BiHom-$\Omega$-bimodule.

\begin{defn}\label{def-RBfBHOalg}
Let $\lambda$ be a given element in {\bf k}. A 5-tuple $(A, \cdot_{\alpha,\,\beta},R_{\omega},p^A_{\omega},q^A_{\omega}) _{\alpha,\,\beta,\,\omega\in \Omega} $ is called a \textbf{Rota-Baxter family BiHom-$\Omega$-associative algebra of weight $\lambda$} if $(A,\cdot_{\alpha,\,\beta},p^A_{\omega},q^A_{\omega})_{\alpha,\,\beta,\,\omega\in \Omega} $ forms a BiHom-$\Omega$-associative algebra and the family of linear maps $ ( R_{\omega})_{\omega\in \Omega}:A\rightarrow A $ satisfy
	\begin{align}\label{pR=Rp}
		p^A_{\alpha}\circ R_{\alpha}=R_{\alpha}\circ p^A_{\alpha}, \;\;\;\; q^A_{\alpha}\circ R_{\alpha}=R_{\alpha}\circ q^A_{\alpha},
	\end{align}
	\begin{align}\label{RBfbihom}
		R_{\alpha}(x)\cdot_{\alpha,\,\beta}R_{\beta}(y)=R_{\alpha\beta}(R_{\alpha}(x)\cdot_{\alpha,\,\beta}y)+R_{\alpha\beta}(x\cdot_{\alpha,\,\beta}R_{\beta}(y))+\lambda R_{\alpha\beta}(x\cdot_{\alpha,\,\beta}y),
	\end{align}
	for all $x,y\in A,\,\alpha,\,\beta\in \Omega$. Then the family of maps $(R_{\omega})_{\omega\in \Omega} $ is called a Rota-Baxter family of weight $\lambda$ on BiHom-$\Omega$-associative algebra $(A,\cdot_{\alpha,\,\beta},p^A_{\omega},q^A_{\omega})_{\alpha,\,\beta,\,\omega\in \Omega}$ .
\end{defn}

\begin{defn}
	Let $ (A, \cdot_{\alpha,\,\beta},R_{\omega},p^A_{\omega},q^A_{\omega})_{\alpha,\,\beta,\,\omega\in \Omega} $ and $ (A', \cdot'_{\alpha,\,\beta},R'_{\omega},p^{A'}_{\omega},q^{A'}_{\omega})_{\alpha,\,\beta,\,\omega\in \Omega} $ be two Rota-Baxter family BiHom-$\Omega$-associative algebras of weight $\lambda$. A family of linear maps $( f_{\alpha}) _{\alpha\in \Omega} $ is called a $\mathbf{Rota}$-$\mathbf{Baxter \, family \, BiHom}$-$\mathbf{\Omega}$-$\mathbf{associative \, algebra \, homomorphism \, of \, weight}$ $\mathbf{\lambda}$ if $(f_{\alpha})_{\alpha\in\Omega}:A\rightarrow A'$ is a homomorphism of  BiHom-$\Omega$-associative algebras of weight $\lambda$ and satisfies 
	\[ f_{\alpha}\circ R_{\alpha}=R_{\alpha}'\circ f_{\alpha},\quad \text{for all }\alpha\in\Omega. \]
\end{defn}

\begin{remark}
	\begin{enumerate}
		\item If the semigroup $\Omega$ is taken to be the trivial monoid with one single
element, then a Rota-Baxter family on the BiHom-$\Omega$-associative algebra reduces to a Rota-Baxter operator on a BiHom-associative algebra induced by Liu, Makhlouf, Menini and Panaitc in \cite[Definition~1.1]{lmmp1}.
		\item In Definition~\ref{def-RBfBHOalg}, if $ p^A_{\alpha}=q^A_{\alpha} $, for all $\alpha\in \Omega$, then we can obtain the notion of a Rota-Baxter family Hom-$\Omega$-associative algebra of weight $\lambda$.
		Moreover, if $ p^A_{\alpha}=q^A_{\alpha}=\mathrm{id}_{\mathrm{A}} $, for all $\alpha\in \Omega$, we get the Rota-Baxter family $\Omega$-associative algebra of weight $\lambda$, which has been introduced in~\cite[Definition 2.5]{L}.	
	\end{enumerate}
\end{remark}
Next, we characterize the Yau twisting procedure for Rota-Baxter family BiHom-$\Omega$-associative algebras.

\begin{prop}\label{RBfftoRBfbihomf}
	Let $A$ be a vector space and let $(p^A_{\omega})_{\omega\in\Omega}, (q^A_{\omega})_{\omega\in \Omega}:A\rightarrow A$ be two commuting families of invertible linear maps which commute with a family of linear maps $( R_{\omega})_{\omega\in \Omega}: A\rightarrow A$. If we define the operation on $A$ by
	\[ x\ast_{\alpha , \,\beta}y
	:=p^A_{\alpha}(x)\cdot_{\alpha,\,\beta}q^A_{\beta}(y),\]
	for all $ x,y\in A,\,\alpha,\,\beta\in \Omega. $	
	Then $ (A,\cdot_{\alpha,\,\beta},R_{\omega})_{\alpha,\,\beta,\,\omega\in \Omega} $ is a Rota-Baxter family $\Omega$-associative algebra if and only if $ (A,\ast_{\alpha,\, \beta},R_{\omega},p^A_{\omega},q^A_{\omega})_{\alpha,\,\beta,\,\omega\in \Omega} $ is a Rota-Baxter family BiHom-$\Omega$-associative algebra.
\end{prop}

\begin{proof}
According to~\cite[Definition 2.5]{L} and Definition~\ref{def-RBfBHOalg}, we only need to prove that Eq.~(\ref{RBfbihom}) holds for the operation $(\ast_{\alpha,\, \beta})_{\alpha,\,\beta\in\Omega}$. For any $ x,y \in A,\, \alpha,\,\beta\in \Omega $, we have 
	\begin{eqnarray*}
		R_{\alpha}(x)\ast_{\alpha,\, \beta}R_{\beta}(y)&=&p^A_{\alpha}R_{\alpha}(x)\cdot_{\alpha,\,\beta }q^A_{\beta}R_{\beta}(y)=R_{\alpha}p^A_{\alpha}(x)\cdot_{\alpha,\,\beta}R_{\beta}q^A_{\beta}(y)\\
		&=&R_{\alpha\beta}(R_{\alpha}p^A_{\alpha}(x)\cdot_{\alpha,\,\beta}q^A_{\beta}(y))+R_{\alpha\beta}(p^A_{\alpha}(x)\cdot_{\alpha,\,\beta}R_{\beta}q^A_{\beta}(y))+\lambda R_{\alpha\beta}(p^A_{\alpha}(x)\cdot_{\alpha,\,\beta}q^A_{\beta}(y))\\
		&&\hspace{1cm}(\text{by Eq.~(\ref{RBfbihom})})\\
		&=&R_{\alpha\beta}(p^A_{\alpha}R_{\alpha}(x)\cdot_{\alpha,\,\beta}q^A_{\beta}(y))+R_{\alpha\beta}(p^A_{\alpha}(x)\cdot_{\alpha,\,\beta }q^A_{\beta}R_{\beta}(y))+\lambda R_{\alpha\beta}(p^A_{\alpha}(x)\cdot_{\alpha,\,\beta}q^A_{\beta}(y))\\
		&=&R_{\alpha\,\beta}(R_{\alpha}(x)\ast_{\alpha ,\, \beta}y)+R_{\alpha\,\beta}(x\ast_{\alpha ,\, \beta}R_{\beta}(y))+\lambda R_{\alpha\,\beta}(x\ast_{\alpha ,\, \beta}y).
	\end{eqnarray*}
This completes the proof.
\end{proof}

\begin{defn}
	Let $ (A,\cdot_{\alpha,\,\beta}, R_{\omega},p^A_{\omega},q^A_{\omega})_{\alpha,\,\beta,\,\omega\in \Omega} $ be a Rota-Baxter family BiHom-$\Omega$-associative algebra and let $ (M,\rhd_{\alpha,\,\beta},\lhd_{\alpha,\,\beta},p_{\omega}^{M},q_{\omega}^{M})_{\alpha,\,\beta,\,\omega \in \Omega} $ be a bimodule over BiHom-$\Omega$-associative algebra $ (A,\cdot_{\alpha,\,\beta},p^A_{\omega},q^A_{\omega})_{\alpha,\,\beta,\,\omega\in \Omega} $. Then $M$ is a \textbf{Rota-Baxter family BiHom-$\Omega$-bimodule} over Rota-Baxter family BiHom-$\Omega$-associative algebra $ (A,\cdot_{\alpha,\,\beta }, R_{\omega},p^A_{\omega},q^A_{\omega})_{\alpha,\,\beta,\,\omega\in \Omega} $ if $M$ is endowed with a family of linear operators $ (T_{\omega})_{\omega\in \Omega} : M\rightarrow M $ such that
	\[p_{\alpha}^M\circ T_{\alpha}=T_{\alpha}\circ p_{\alpha}^M,\quad q_{\alpha}^M\circ T_{\alpha}=T_{\alpha}\circ q_{\alpha}^M,\]
	\begin{align}\label{RBbimodulefamily-1}
		R_{\alpha}(a)\rhd_{\alpha,\,\beta}T_{\beta}(m)=T_{\alpha\,\beta}(a\rhd_{\alpha,\,\beta}T_{\beta}(m)+R_{\alpha}(a)\rhd_{\alpha,\,\beta}m+\lambda a\rhd_{\alpha,\,\beta}m),
	\end{align}
	\begin{align}\label{RBbimodulefamily-2}
		T_{\alpha}(m)\lhd_{\alpha,\,\beta}R_{\beta}(a)=T_{\alpha\,\beta}(m\lhd_{\alpha,\,\beta}R_{\beta}(a)+T_{\alpha}(m)\lhd_{\alpha,\,\beta}a+\lambda m\lhd_{\alpha,\,\beta}a),
	\end{align}
	for all $ a\in A,\; m\in M,\;\alpha,\,\beta\in \Omega $.

	We call $ (A,\cdot_{\alpha,\,\beta}, R_{\omega},p^A_{\omega},q^A_{\omega})_{\alpha,\,\beta,\,\omega\in \Omega} $ the \textbf{regular Rota-Baxter family BiHom-$\Omega$-bimodule}.
\end{defn}

\begin{prop}\label{prop-Rbf-semi-dirct}
	Let $ (A, \cdot_{\alpha,\,\beta}, R_{\omega},p^A_{\omega},q^A_{\omega})_{\alpha,\,\beta,\,\omega\in \Omega} $ be a Rota-Baxter family BiHom-$\Omega$-associative algebra and let $ (M,\rhd_{\alpha,\,\beta},\lhd_{\alpha,\,\beta},p_{\omega}^{M},q_{\omega}^{M})_{\alpha,\,\beta,\,\omega\in \Omega}$ be a bimodule over the BiHom-$\Omega$-associative algebra $ (A,\cdot_{\alpha,\,\beta}, p^A_{\omega},q^A_{\omega})_{\alpha,\,\beta,\,\omega\in \Omega} $. If we define a family of linear maps on vector space $A\oplus M$ by
	\[T_{\alpha}^{\oplus}(a,m):=\big(R_{\alpha}(a), T_{\alpha}(m)\big),\]
	for all $(a,m)\in A\oplus M,\;\alpha\in \Omega$. Then the semi-direct product BiHom-$\Omega$-associative algebra $A \ltimes M$ equipped with operator $( T_{\alpha}^{\oplus})_{\alpha\in \Omega}$ is a Rota-Baxter family BiHom-$\Omega$-associative algebra if and only if $ (M,\rhd_{\alpha,\,\beta},\lhd_{\alpha,\,\beta},T_{\omega},p_{\omega}^{M},q_{\omega}^{M})_{\alpha,\,\beta,\,\omega\in \Omega} $ is a Rota-Baxter family BiHom-$\Omega$-bimodule over $ A $. This new Rota-Baxter family BiHom-$\Omega$-associative algebra is called the \textbf{semi-direct product} (or \textbf{trivial extension} ) of $A$ by $M$.
\end{prop}
\begin{proof}
	It is a direct calculation.
\end{proof}
\begin{remark}
	Proposition~\ref{prop-Rbf-semi-dirct} is a special case in Lemma~\ref{lemma} when one take $ \psi_{\alpha,\,\beta}$ and $ \chi_{\omega}$ to be zero for all $\alpha,\,\beta\in\Omega$ and $\omega\in\Omega$.
\end{remark}

\begin{prop}\label{defstar}
	Let $ (A,\cdot_{\alpha,\,\beta}, R_{\omega},p^A_{\omega},q^A_{\omega})_{\alpha,\,\beta,\,\omega\in \Omega} $ be a Rota-Baxter family BiHom-$\Omega$-associative algebra of weight $\lambda$. Define a binary operation on $A$ by
	\begin{align*}
		a\star_{\alpha,\,\beta}b := a\cdot_{\alpha,\,\beta}R_{\beta}(b)+R_{\alpha}(a)\cdot_{\alpha,\,\beta}b+\lambda a\cdot_{\alpha,\,\beta}b,
	\end{align*}
	for all $ a,b\in A ,\,\alpha,\,\beta\in \Omega.$ Then
	\begin{enumerate}
		\item ~\cite[Theorem~2.9]{BHOalg} the quadruple $ (A, \star_{\alpha,\,\beta} ,p^A_{\omega},q^A_{\omega})_{\alpha,\,\beta,\,\omega\in \Omega} $ is a new BiHom-$\Omega$-associative algebra and denote it by $ A_{\star} $.
	
		\item the family of linear maps $ (R_{\omega})_{\omega\in \Omega} : (A, \star_{\alpha,\,\beta},p^A_{\omega},q^A_{\omega})_{\alpha,\,\beta,\,\omega\in\Omega}\rightarrow (A,\cdot_{\alpha,\,\beta}, p^A_{\omega},q^A_{\omega})_{\alpha,\,\beta,\,\omega\in \Omega} $ is a BiHom-$\Omega$-associative algebra homomorphism.
	\end{enumerate}
\end{prop}
\begin{proof}
It is a direct calculation.
\end{proof}

Next, we construct a bimodule structure over the BiHom-$\Omega$-associative algebra $A_{\star}$ as follows.

\begin{prop}\label{defrhd}
	Let $ (A,\cdot_{\alpha,\,\beta}, R_{\omega},p^A_{\omega},q^A_{\omega})_{\alpha,\,\beta,\,\omega\in \Omega} $ be a Rota-Baxter family BiHom-$\Omega$-associative algebra of weight $\lambda$ and $ (M,\rhd_{\alpha,\,\beta},\lhd_{\alpha,\,\beta},T_{\omega},p_{\omega}^{M},q_{\omega}^{M})_{\alpha,\,\beta,\,\omega\in \Omega} $ be a Rota-Baxter family BiHom-$\Omega$-bimodule over $A$. We define two families of bilinear maps $ ( \blacktriangleright_{\alpha,\,\beta})_{\alpha,\,\beta\in \Omega} $ and $ (\blacktriangleleft_{\alpha,\,\beta})_{\alpha,\,\beta\in \Omega} $ as follows.
	\[\blacktriangleright_{\alpha,\,\beta} : A\otimes M\rightarrow M,\]
	\begin{align*}
		a\blacktriangleright_{\alpha,\,\beta}m := R_{\alpha}(a)\rhd_{\alpha,\,\beta}m-T_{\alpha\,\beta}(a\rhd_{\alpha,\,\beta}m),
	\end{align*}
	\[\blacktriangleleft_{\alpha,\,\beta} : M\otimes A\rightarrow M,\]
	\begin{align*}
		m\blacktriangleleft_{\alpha,\,\beta}a := m\lhd_{\alpha,\,\beta}R_{\beta}(a)-T_{\alpha\beta}(m\lhd_{\alpha,\,\beta}a),
	\end{align*}
for all $ a\in A ,\; m\in M ,\; \alpha,\,\beta\in \Omega$. Then $M_{\star}:=(M,\blacktriangleright_{\alpha,\,\beta},\blacktriangleleft_{\alpha,\,\beta}, p_{\omega}^M,q_{\omega}^M)_{\alpha,\,\beta,\,\omega\in \Omega}$ is a bimodule over $ A_{\star} $.
\end{prop}

\begin{proof}
	For any $a,b\in A,\, m\in M,\;\alpha,\,\beta,\,\gamma\in \Omega$, we first prove that $ (M,\blacktriangleright_{\alpha,\,\beta},p_{\omega}^{M},q_{\omega}^{M})_{\alpha,\,\beta,\,\omega\in \Omega} $ is a left module over BiHom-$\Omega$-associative algebra $ A_{\star}$.
	\begin{align*}
		&p^A_{\alpha}(a)\blacktriangleright_{\alpha,\,\beta\,\gamma}(b\blacktriangleright_{\beta,\,\gamma}m)\\
		=&R_{\alpha}p^A_{\alpha}(a)\rhd_{\alpha,\,\beta\,\gamma}(R_{\beta}(b)\rhd_{\beta,\,\gamma}m-T_{\beta\,\gamma}(b\rhd_{\beta,\,\gamma}m))-T_{\alpha\,\beta\,\gamma}(p^A_{\alpha}(a)\rhd_{\alpha,\,\beta\,\gamma}(R_{\beta}(b)\rhd_{\beta,\,\gamma}m-T_{\beta\,\gamma}(b\rhd_{\beta,\,\gamma}m)))\\
		=&R_{\alpha}p^A_{\alpha}(a)\rhd_{\alpha,\,\beta\,\gamma}(R_{\beta}(b)\rhd_{\beta,\,\gamma}m)-R_{\alpha}(p^A_{\alpha}(a))\rhd_{\alpha,\,\beta\,\gamma}T_{\beta\,\gamma}(b\rhd_{\beta,\,\gamma}m) -T_{\alpha\,\beta\,\gamma}(p^A_{\alpha}(a)\rhd_{\alpha,\,\beta\,\gamma}(R_{\beta}(b)\rhd_{\beta,\,\gamma}m))\\
		&+T_{\alpha\,\beta\,\gamma}(p^A_{\alpha}(a)\rhd_{\alpha,\,\beta\,\gamma}T_{\beta\,\gamma}(b\rhd_{\beta,\,\gamma}m))\\
		=&R_{\alpha}p^A_{\alpha}(a)\rhd_{\alpha,\,\beta\,\gamma}(R_{\beta }(b)\rhd_{\beta,\,\gamma}m)-T_{\alpha\,\beta\,\gamma}(p^A_{\alpha}(a)\rhd_{\alpha,\,\beta\,\gamma}T_{\beta\,\gamma}(b\rhd_{\beta,\,\gamma}m) +R_{\alpha}(p^A_{\alpha}(a))\rhd_{\alpha,\,\beta\,\gamma}(b\rhd_{\beta,\,\gamma}m) \\
		&+\lambda p^A_{\alpha}(a)\rhd_{\alpha,\,\beta\,\gamma}(b\rhd_{\beta,\,\gamma}m)) -T_{\alpha\,\beta\,\gamma}(p^A_{\alpha}(a)\rhd_{\alpha,\,\beta\,\gamma}(R_{\beta}(b)\rhd_{\beta,\,\gamma}m))+T_{\alpha\,\beta\,\gamma}(p^A_{\alpha}(a)\rhd_{\alpha,\,\beta\,\gamma}T_{\beta\,\gamma}(b\rhd_{\beta,\,\gamma}m))\\
		&\hspace{1cm}(\text{by Eq.~(\ref{RBbimodulefamily-1})}\\
		=&R_{\alpha}p^A_{\alpha}(a)\rhd_{\alpha,\,\beta\,\gamma}(R_{\beta}(b)\rhd_{\beta,\,\gamma}m)-T_{\alpha\,\beta\,\gamma}(R_{\alpha}p^A_{\alpha}(a)\rhd_{\alpha,\,\beta\,\gamma}(b\rhd_{\beta,\,\gamma}m))- T_{\alpha\,\beta\,\gamma}(p^A_{\alpha}(a)\rhd_{\alpha,\,\beta\,\gamma}(R_{\beta}(b)\rhd_{\beta,\,\gamma}m))\\
		&-\lambda T_{\alpha\,\beta\,\gamma}(p^A_{\alpha}(a)\rhd_{\alpha,\,\beta\,\gamma}(b\rhd_{\beta,\,\gamma}m)),\\
		=&p^A_{\alpha}R_{\alpha}(a)\rhd_{\alpha,\,\beta\,\gamma}(R_{\beta}(b)\rhd_{\beta,\,\gamma}m)-T_{\alpha\,\beta\,\gamma}(p^A_{\alpha}(a)\rhd_{\alpha,\,\beta\,\gamma}(R_{\beta}(b)\rhd_{\beta,\,\gamma}m))- T_{\alpha\,\beta\,\gamma}(p^A_{\alpha}R_{\alpha}(a)\rhd_{\alpha,\,\beta\,\gamma}(b\rhd_{\beta,\,\gamma}m))\\
		&-\lambda T_{\alpha\,\beta\,\gamma}(p^A_{\alpha}(a)\cdot_{\alpha,\,\beta\,\gamma}(b\rhd_{\beta,\,\gamma}m))\\
		=&(R_{\alpha}(a)\cdot_{\alpha,\,\beta}R_{\beta}(b))\rhd_{\alpha\,\beta,\,\gamma}q_{\gamma}^{M}(m)-T_{\alpha\,\beta\,\gamma}((a\cdot_{\alpha,\,\beta }R_{\beta}(b)+R_{\alpha}(a)\cdot_{\alpha,\,\beta}b+\lambda a\cdot_{\alpha,\,\beta}b)\rhd_{\alpha\,\beta,\,\gamma}q_{\gamma}^{M}(m))\\&\hspace{1cm}(\text{by Eq.~(\ref{lmod})})\\
		=&R_{\alpha\beta}(a\star_{\alpha,\,\beta}b)\rhd_{\alpha\,\beta,\,\gamma}q_{\gamma}^{M}(m)-T_{\alpha\,\beta\,\gamma}((a\star_{\alpha,\,\beta}b)\rhd_{\alpha\,\beta,\,\gamma}q_{\gamma}^{M}(m))\\
		=&(a\star_{\alpha,\,\beta}b)\blacktriangleright_{\alpha\beta,\,\gamma}q_{\gamma}^{M}(m).
	\end{align*}
	
 Similarly, we obtain that $ (M,\blacktriangleleft_{\alpha,\,\beta},p^M_{\omega},q^M_{\omega})_{\alpha,\,\beta,\,\omega\in \Omega} $ is a right module over BiHom-$\Omega$-associative algebra $ A_{\star} $ and Eq.~(\ref{bimod}) holds for operations $ ( \blacktriangleright_{\alpha,\,\beta})_{\alpha,\,\beta\in \Omega} $ and $ ( \blacktriangleleft_{\alpha,\,\beta})_{\alpha,\,\beta\in \Omega} $.  Thus, $M_{\star}$ is a bimodule over BiHom-$\Omega$-associative algebra $ A_{\star}$. This completes the proof.
\end{proof}

\section{Cohomology of Rota-Baxter family BiHom-$\Omega$-associative algebras}\label{sec3}

In this section, we assume that $\Omega$ is a semigroup with unit $ 1 \in \Omega $. The unital condition of $\Omega$ is only useful in the coboundary operator of the cohomology at the degree 0 level.

\subsection{Cohomology of BiHom-$\Omega$-associative algebras}\label{sec-cohomo-BiHomassof}

In this subsection, inspired by the cohomology theory of BiHom-associative algebras in \cite{hochsch-BiHom-asso}, we first study the cohomology theory for BiHom-$\Omega$-associative algebras. Then, we introduce the BiHom-$\Omega$-Gerstenhaber bracket over the cochain complex of BiHom-$\Omega$-associative algebras. 

\smallskip 
From now on, if $V_{1},...,V_{n},W$ are vector spaces and $n\geq 1$, then we denote
\begin{align*}
	\mathrm{Hom}_{\Omega}(V_{1}\otimes\cdots V_{n},W)=\prod_{(\alpha_{1},\dots,\alpha_{n})\in\Omega^{n}}\mathrm{Hom}(V_{1}\otimes\dots\otimes V_{n},W),
\end{align*}
whose elements can be written as $f=(f_{\alpha_{1},\dots,\alpha_{n}}: V_{1}\otimes\cdots\otimes V_{n}\rightarrow W)_{\alpha_{1},\dots,\alpha_{n}\in\Omega}.$

Let $(M,\rhd_{\alpha,\,\beta},\lhd_{\alpha,\,\beta},p_{\omega}^M,q_{\omega}^M)_{\alpha,\,\beta,\,\omega\in \Omega}$ be a bimodule over BiHom-$\Omega$-associative algebra $( A,\cdot_{\alpha,\,\beta},p_{\omega},\\q_{\omega})_{\alpha,\,\beta,\,\omega\in \Omega} $. Now we describe the cochain complex $(\mathrm{C}_{\Omega}^{\bullet}(A,M),\delta_{\mathrm{Alg}}^{\bullet})$ of the BiHom-$\Omega$-associative algebra $A$ with coefficients in bimodule $M$. For $ n\geqslant 0 $, we define the space $ \mathrm{C}_{\Omega}^{n}(A,M) $ consisting of all families of multilinear maps of the form $f=(f_{_{\alpha_{1},\dots,\alpha_{n}}})_{{\alpha_{1},\dots,\alpha_{n}} \in \Omega} \in \mathrm{Hom}_{\Omega}(A^{\otimes n}, M)$ satisfying  
\begin{align*}
	p_{\alpha_{1}\dots\alpha_{n}}^{M}\circ f_{\alpha_{1},\dots,\alpha_{n}}=f_{\alpha_{1},\dots,\alpha_{n}} \circ (p_{\alpha_{1}},\dots,p_{\alpha_{n}}),\\
	q_{\alpha_{1}\dots\alpha_{n}}^{M}\circ f_{\alpha_{1},\dots,\alpha_{n}}=f_{\alpha_{1},\dots,\alpha_{n}} \circ (q_{\alpha_{1}},\dots,q_{\alpha_{n}}),
\end{align*}
for all $\alpha_{1},\dots,\alpha_{n}\in\Omega$.
The coboundary operator of the BiHom-$\Omega$-associative algebra $A$ with coefficients in the bimodule $M$:
\[ \delta_{\mathrm{Alg}}^{n}:\mathrm{C}_{\Omega}^{n}(A,M)\rightarrow \mathrm{C}_{\Omega}^{n+1}(A,M) \]
 is defined by
\[ \delta^0_{\mathrm{Alg}}(m)_{\alpha}(a_{1}):=a_{1}\rhd_{\alpha,1}m-m\lhd_{1,\alpha}a_{1}, \]
\begin{align}\label{BiHomOmHochschdiff}
	&(\delta_{\mathrm{Alg}}^{n}f)_{\alpha_{1},\dots,\alpha_{n+1}}(a_{1},\dots,a_{n+1}):=p_{\alpha_{1}}^{n-1}(a_{1})\rhd_{\alpha_{1},\alpha_{2}\dots\alpha_{n+1}}f_{\alpha_{2},\dots,\alpha_{n+1}}(a_{2},\dots,a_{n+1})\nonumber\\
	&+\sum_{i=1}^{n}(-1)^{i}f_{\alpha_{1},\dots,\alpha_{i}\alpha_{i+1},\dots,\alpha_{n+1}}(p_{\alpha_{1}}(a_{1}),\dots,p_{\alpha_{i-1}}(a_{i-1}),a_{i}\cdot_{\alpha_{i},\alpha_{i+1}}a_{i+1},q_{\alpha_{i+2}}(a_{i+2}),\dots,q_{\alpha_{n+1}}(a_{n+1}))\\
		&+(-1)^{n+1}f_{\alpha_{1},\dots,\alpha_{n}}(a_{1},\dots,a_{n})\lhd_{\alpha_{1}\dots\alpha_{n},\alpha_{n+1}}q_{\alpha_{n+1}}^{n-1}(a_{n+1}),\nonumber
\end{align}
for all $f=(f_{\alpha_{1},\dots,\alpha_{n}})_{\alpha_{1},\dots,\alpha_{n}\in\Omega}\in \mathrm{C}_{\Omega}^{n}(A,M),\, m\in M ,\; a_{1},a_{2},\dots,a_{n+1}\in A,\; \alpha_{1},\dots,\alpha_{n+1}\in\Omega$.

\begin{defn}
	An $n$-cochain $ f=(f_{\alpha_{1},\dots,\alpha_{n}})_{\alpha_{1},\dots,\alpha_{n}\in\Omega}\in \mathrm{C}_{\Omega}^{n}(A,M) $ is called an $ \bf{n\text{-}cocycle} $ if
	\[(\delta_{\mathrm{Alg}}^{n}f)_{\alpha_{1},\dots,\alpha_{n+1}}=0\]
	and the element of the form $(\delta_{\mathrm{Alg}}^{n-1}g)_{\alpha_{1},\dots,\alpha_{n}}$is called an $\bf{n\text{-}coboundary}$, where $g=(g_{\alpha_{1},\dots,\alpha_{n-1}})_{\alpha_{1},\dots,\alpha_{n-1}\in\Omega}\\\in \mathrm{C}_{\Omega}^{n-1}(A,M)$. The spaces consisting of n-cocycles and n-coboundaries are denoted $\mathrm{Z}_{\Omega}^n(A,M)$ and $\mathrm{B}_{\Omega}^n(A,M)$, respectively. Then the quotient space
	\[\mathrm{H}_{\Omega}^{n}(A,M)=\mathrm{Z}_{\Omega}^{n}(A,M)/\mathrm{B}_{\Omega}^{n}(A,M)\]
	is called the n-th cohomology group of $A$ with coefficients in bimodule $M$. We call $\big(\mathrm{C}_{\Omega}^{\bullet}(A,M),\delta_{\mathrm{Alg}}^{\bullet}\big)$ the \textbf{cochain complex of BiHom-$\Omega$-associative algebra $A$ with coefficients in bimodule $M$}. Its cohomology, denote by $\mathrm{H}_{\Omega}^{\bullet}(A,M)$, is called the \textbf{cohomology of BiHom-$\Omega$-associative algebra $A$ with coefficients in bimodule $M$}. 
\end{defn}

In particular, when $M$ is the regular bimodule, the cochain complex $\big(\mathrm{C}_{\Omega}^{\bullet}(A,A),\delta_{\mathrm{Alg}}^{\bullet}\big)$ is simply denoted by $\big(\mathrm{C}_{\Omega}^{\bullet}(A),\delta_{\mathrm{Alg}}^{\bullet}\big)$. The corresponding cohomology, simply denoted by $\mathrm{H}_{\Omega}^{\bullet}(A)$, is called the cohomology of the BiHom-$\Omega$-associative algebra $A$.

\begin{remark}\label{2-cocycle}
A 2-cocycle in $\mathrm{C}_{\Omega}^{2}(A,M)$ is a family of bilinear maps $ ( H_{\alpha,\,\beta})_{\alpha,\,\beta  \in \Omega}:A\otimes A\rightarrow M $ satisfying
\begin{align}\label{Hoc1}
	&H_{\alpha,\,\beta }\circ (p_{\alpha}\otimes p_{\beta})=p_{\alpha\,\beta}^{M}\circ H_{\alpha,\,\beta }, \;\;\;\;\; H_{\alpha,\,\beta }\circ (q_{\alpha}\otimes q_{\beta})=q_{\alpha\,\beta}^{M}\circ H_{\alpha,\,\beta },
\end{align}
\begin{align}\label{cocycle}
	\begin{split}
		&p_{\alpha}(x)\rhd_{\alpha,\,\beta\,\gamma}H_{\beta,\,\gamma}(y, z)- H_{\alpha \, \beta,\,\gamma}(x\cdot_{\alpha,\,\beta}y, q_{\gamma}(z))+ H_{\alpha,\,\beta\,\gamma}(p_{\alpha}(x), y\cdot_{\beta,\,\gamma}z)\\
		&\hspace{1cm}- H_{\alpha,\,\beta}(x, y)\lhd_{\alpha\,\beta,\,\gamma} q_{\gamma}(z)=0,
	\end{split}
\end{align}
for all $ x, y, z\in A ,\;\alpha,\,\beta,\,\gamma \in \Omega$. The space of 2-cocycles $\mathrm{Z}_{\Omega}^{2}(A,M)=\mathrm{Ker}\delta_{\mathrm{Alg}}^{2}\subseteq \mathrm{C}_{\Omega}^{2}(A,M)$ consists of all families of bilinear maps $ f = ( f_{\alpha,\,\beta})_{\alpha,\,\beta\in \Omega}:A\otimes A\rightarrow M$ satisfying $ (\delta_{\mathrm{Alg}}^{2}f)_{\alpha,\,\beta,\,\gamma}=0$, for all $\alpha,\,\beta,\,\gamma\in\Omega$.
\end{remark}

Next, we are going to introduce a Lie bracket on the underlying space of cochain complex of BiHom-$\Omega$-associative algebras. Let $(A,\mu_{\alpha,\,\beta},p_{\omega},q_{\omega})_{\alpha,\,\beta,\,\omega\in\Omega}$ be a BiHom-$\Omega$-associative algbera. If $f\in \mathrm{C}_{\Omega}^{n}(A)$, we denote $|f|=n-1$. Now, we give the definition of compositions on $\mathrm{C}_{\Omega}^{\bullet}(A):=\oplus_{n\geq 1}\mathrm{C}_{\Omega}^{n}(A)$ as follows.


\begin{defn}\label{def-composition}
	For any $f\in \mathrm{C}_{\Omega}^n(A),\, g_{i}\in \mathrm{C}_{\Omega}^{m_{i}}(A),\; 1\leq i\leq n$, we define the composition
	\[\diamond^{\Omega}: \mathrm{C}_{\Omega}^n(A)\otimes \mathrm{C}_{\Omega}^{m_{1}}(A)\otimes \cdots \otimes \mathrm{C}_{\Omega}^{m_{n}}(A)\rightarrow \mathrm{C}_{\Omega}^{m_{1}+\cdots+m_{n}}(A) \] by
	\begin{align*}
		\big(&f\diamond^{\Omega} (g_{1},\dots,g_{n})\big)_{\alpha_{1},\dots,\alpha_{m_{1}+\cdots+m_{n}}}(a_{1},\dots,a_{m_{1}+\cdots+m_{n}})\\
		=&f\Big(p_{\alpha_{1}\dots\alpha_{m_{1}}}^{\sum_{l> 1}^{n}|g_{l}|}\circ g_{1},\, p_{\alpha_{m_{1}+1}\dots\alpha_{m_{1}+m_{2}}}^{\sum_{l>2}^{n}|g_{l}|}\circ q_{\alpha_{m_{1}+1}\dots\alpha_{m_{1}+m_{2}}}^{|g_{1}|}\circ g_{2},\dots,p_{\alpha_{m_{1}+\cdots+m_{i-1}+1}\dots\alpha_{m_{1}+\cdots+m_{i}}}^{\sum_{l>i}^{n}|g_{l}|}\circ q_{\alpha_{m_{1}+\cdots+m_{i-1}+1}\dots\alpha_{m_{1}+\cdots+m_{i}}}^{\sum_{l<i}|g_{l}|}\circ g_{i},\\
		&\dots,q_{\alpha_{m_{1}+\cdots+m_{n-1}+1}\dots\alpha_{m_{1}+\cdots+m_{n}}}^{\sum_{l<n}|g_{l}|}\circ g_{n}\Big)(a_{1},\dots,a_{m_{1}+\cdots+m_{n}})\\
		=&f\Big(p_{\alpha_{1}\dots\alpha_{m_{1}}}^{\sum_{l> 1}^{n}|g_{l}|}\circ g_{1}(a_{1},\dots,a_{m_{1}}),\,p_{\alpha_{m_{1}+1}\dots\alpha_{m_{1}+m_{2}}}^{\sum_{l>2}^{n}|g_{l}|}\circ q_{\alpha_{m_{1}+1}\dots\alpha_{m_{1}+m_{2}}}^{|g_{1}|}\circ g_{2}(a_{m_{1}+1},\dots,a_{m_{1}+m_{2}}),\dots,\\
		&p_{\alpha_{m_{1}+\cdots+m_{i-1}+1}\dots\alpha_{m_{1}+\cdots+m_{i}}}^{\sum_{l>i}^{n}|g_{l}|}\circ q_{\alpha_{m_{1}+\cdots+m_{i-1}+1}\dots\alpha_{m_{1}+\cdots+m_{i}}}^{\sum_{l<i}|g_{l}|}\circ g_{i}(a_{m_{1}+\cdots+m_{i-1}+1},\dots,a_{m_{1}+\cdots+m_{i-1}+m_{i}}),\dots,\\
		&q_{\alpha_{m_{1}+\cdots+m_{n-1}+1}\dots\alpha_{m_{1}+\cdots+m_{n}}}^{\sum_{l<n}|g_{l}|}\circ g_{n}(a_{m_{1}+\cdots+m_{n-1}+1},\dots,a_{m_{1}+\cdots+m_{n-1}+m_{n}})  \Big),
	\end{align*}
for all $	\alpha_{1},\dots,\alpha_{m_{1}+\cdots+m_{n}}\in \Omega,\; a_{1},\dots,a_{m_{1}+\cdots+m_{n}}\in A$.
	
	In particular, for any $f\in \mathrm{C}_{\Omega}^n(A),\,g\in \mathrm{C}_{\Omega}^m(A)$ and $1\leq i\leq n$, we define the composition $ \diamond_{i}^{\Omega}: \mathrm{C}_{\Omega}^n(A)\otimes \mathrm{C}_{\Omega}^m(A)\rightarrow \mathrm{C}_{\Omega}^{n+m-1}(A) $ by
	\small{\begin{align}\label{eq-composition}
		f\diamond_{i}^{\Omega}g=&\Big((f\diamond_{i}^{\Omega}g)_{\alpha_{1},\dots,\alpha_{n+m-1}}\Big)_{\alpha_{1},\dots,\alpha_{n+m-1}\in\Omega}\nonumber\\
		:=&\Big(f_{\alpha_{1},\dots,\alpha_{i-1},\alpha_{i}\dots\alpha_{i+m-1},\alpha_{i+m},\dots,\alpha_{n+m-1}}(p_{\alpha_{1}}^{m-1},\dots,p_{\alpha_{i-1}}^{m-1},g_{\alpha_{i},\dots,\alpha_{i+m-1}},q_{\alpha_{i+m}}^{m-1},\dots,q_{\alpha_{n+m-1}}^{m-1})\Big)_{\alpha_{1},\dots,\alpha_{n+m-1}\in\Omega}.
	\end{align}}
\end{defn}
\begin{remark}
	With the notation of Definition~\ref{def-composition}, it is not difficult to verify that the definition of $\diamond_{i}^{\Omega}$ is well defined. That is $f\diamond_{i}^{\Omega}g\in \mathrm{C}_{\Omega}^{n+m-1}(A)$. 
\end{remark}
By~\cite[Proposition~4.1]{hochsch-BiHom-asso}, we know that the composition $\diamond_{i}^{\Omega}$ defines a non-symmetric operad structure on $\mathrm{C}_{\Omega}^{\bullet}(A)$ with the identity element $\text{id}_{A}$. Inspired by~\cite{L}, we give the concept of BiHom-$\Omega$-Gerstenhaber bracket as follows.
\begin{defn}\label{def-bracket}
	The \textbf{BiHom-$\Omega$-Gerstenhaber bracket} on $\mathrm{C}_{\Omega}^{\bullet}(A)=\oplus_{n\geq 1}\mathrm{C}_{\Omega}^n(A)$ is a bracket $[-,-]_{G}^{\Omega}$ of degree -1 defined by
	\begin{align*}
		[f,g]_{G}^{\Omega}=\sum_{i=1}^n(-1)^{(m-1)(i-1)}f\diamond_{i}^{\Omega}g-(-1)^{(n-1)(i-1)}g\diamond_{i}^{\Omega}f,
	\end{align*}
for all $f\in \mathrm{C}_{\Omega}^n(A),\,g\in \mathrm{C}_{\Omega}^m(A)$.
\end{defn}
Next, we give two examples to explain how to use $[-,-]_{G}^{\Omega}$ for calculations.
\begin{exam} \label{Ex: Gerstenhaber bracket}
	If $\mu=(\mu_{\alpha_{1},\,\alpha_{2}})_{\alpha_{1},\,\alpha_{2}\in\Omega}\in \mathrm{C}_{\Omega}^2(A), f=(f_{\alpha_{1},\,\alpha_{2},\,\alpha_{3}})_{\alpha_{1},\,\alpha_{2},\,\alpha_{3}\in\Omega}\in\mathrm{C}_{\Omega}^3(A)$, then by Definition~\ref{def-bracket}, we have
	\begin{align*}
		[\mu,\mu]_{G}^{\Omega}=&\sum_{i=1}^2(-1)^{i-1}\mu\diamond_{i}^{\Omega}\mu+\sum_{i=1}^2(-1)^{i-1}\mu\diamond_{i}^{\Omega}\mu\\
		=&2(\mu\diamond_{1}^{\Omega}\mu-\mu\diamond_{2}^{\Omega}\mu)\\
		=&\Big(2(\mu_{\alpha_{1}\,\alpha_{2},\,\alpha_{3}}(\mu_{\alpha_{1},\,\alpha_{2}}\otimes q_{\alpha_{3}})-\mu_{\alpha_{1},\,\alpha_{2}\,\alpha_{3}}(p_{\alpha_{1}}\otimes \mu_{\alpha_{2},\,\alpha_{3}}))\Big)_{\alpha_{1},\,\alpha_{2},\,\alpha_{3}\in\Omega},
	\end{align*}
and 
\begin{align*}
	[\mu,f]_{G}^{\Omega}=&\sum_{i=1}^{2}(-1)^{2(i-1)}\mu\diamond_{i}^{\Omega}f-(-1)^{i-1}f\diamond_{i}^{\Omega}\mu\\
	=&\mu\diamond_{1}^{\Omega}f-f\diamond_{1}^{\Omega}\mu+\mu\diamond_{2}^{\Omega}f+f\diamond_{2}^{\Omega}\mu\\
	=&\Big(\mu_{\alpha_{1}\alpha_{2}\alpha_{3},\,\alpha_{4}}(f_{\alpha_{1},\,\alpha_{2},\,\alpha_{3}}\otimes q_{\alpha_{4}}^2)-f_{\alpha_{1}\alpha_{2},\,\alpha_{3},\,\alpha_{4}}(\mu_{\alpha_{1},\,\alpha_{2}}\otimes q_{\alpha_{3}}\otimes q_{\alpha_{4}})+\mu_{\alpha_{1},\,\alpha_{2}\alpha_{3}\alpha_{4}}(p_{\alpha_{1}}^2\otimes f_{\alpha_{2},\,\alpha_{3}\alpha_{4}})\\
	&+f_{\alpha_{1},\,\alpha_{2}\alpha_{3},\,\alpha_{4}}(p_{\alpha_{1}}\otimes \mu_{\alpha_{2},\,\alpha_{3}}\otimes q_{\alpha_{4}})\Big)_{\alpha_{1},\,\alpha_{2},\,\alpha_{3},\,\alpha_{4}\in\Omega}.
\end{align*}
\end{exam}

For any $f\in \mathrm{C}_{\Omega}^{n+1}(A),\,g\in \mathrm{C}_{\Omega}^{m+1}(A)$ and by Definition~\ref{def-bracket}, we have $[f,g]_{G}^{\Omega}\in \mathrm{C}_{\Omega}^{n+m+1}(A)$. Hence, the degree of bracket $[-,-]_{G}^{\Omega}$ on space $\mathrm{C}_{\Omega}^{\bullet +1}(A)$ is 0. Combining BiHom-associative algebras~\cite{hochsch-BiHom-asso} and $\Omega$-associative algebras~\cite{L}, we come to the following conclusion.
\begin{prop}\label{graded-Liealg}
	 If $\mathrm{C}_{\Omega}^{\bullet +1}(A)=\oplus_{n\geq 0}\mathrm{C}_{\Omega}^{n+1}(A)$, then $(\mathrm{C}_{\Omega}^{\bullet +1}(A),[-,-]_{G}^{\Omega})$ is a graded Lie algebra.
\end{prop}
\begin{proof}
	The proof is similar to the way of ~\cite{hochsch-BiHom-asso}.
\end{proof}
Since $(\mathrm{C}_{\Omega}^{\bullet +1}(A),[-,-]_{G}^{\Omega})$ is a graded Lie algebra, we get
\begin{align*}
	[f,g]_{G}^{\Omega}=&-(-1)^{|f||g|}[g,f]_{G}^{\Omega},\\
	(-1)^{|f||h|}[f,[g,h]_{G}^{\Omega}]_{G}^{\Omega}+(-1)&^{|g||f|}[g,[h,f]_{G}^{\Omega}]_{G}^{\Omega}+(-1)^{|h||g|}[h,[f,g]_{G}^{\Omega}]_{G}^{\Omega}=0,
\end{align*}
for all $f,g,h\in \mathrm{C}_{\Omega}^{\bullet +1}(A)$.

 Now we give an important result about the structure of BiHom-$\Omega$-associative algebras.
 \begin{prop}\label{prop-MC}
 	If $\mu=( \mu_{\alpha,\,\beta})_{\alpha,\,\beta\in \Omega}\in \mathrm{C}_{\Omega}^2(A)$. Then $(A,\mu_{\alpha,\,\beta},p_{\omega},q_{\omega})_{\alpha,\,\beta,\,\omega\in \Omega}$ is a BiHom-$\Omega$-associative algebra if and only if $\mu$ is a Maurer-Cartan element of graded Lie algebra $(\mathrm{C}_{\Omega}^{\bullet +1}(A),[-,-]_{G}^{\Omega})$, i.e. $[\mu,\mu]_{G}^{\Omega}=0$.
 \end{prop}
\begin{proof}
	This is a direct corollary of Example~\ref{Ex: Gerstenhaber bracket}.
\end{proof}

 \begin{coro} \label{coro: dg Lie}
 	If $(A,\mu_{\alpha,\,\beta},p_{\omega},q_{\omega})_{\alpha,\,\beta,\,\omega\in \Omega}$ is a BiHom-$\Omega$-associative algebra, then $(\mathrm{C}_{\Omega}^{\bullet +1}(A),[-,-]_{G}^{\Omega},\\\delta=[\mu,-]_{G}^{\Omega})$ is a differential graded Lie algebra, where $\mu=(\mu_{\alpha,\,\beta})_{\alpha,\,\beta\in\Omega}$.
 \end{coro}

 \begin{prop}
 	If we define the operation on $\mathrm{C}_{\Omega}^{\bullet +1}(A)$ by
 	\[\delta_{\mathrm{alg}}(f):=(-1)^{|f|}\delta(f)=(-1)^{|f|}[\mu,f]_{G}^{\Omega},\quad \text{for all $f\in \mathrm{C}_{\Omega}^{\bullet +1}(A)$,}\]
 	then $\delta_{\mathrm{alg}}$ is a differential of the cochain complex of BiHom-$\Omega$-associative algebra $(A,\mu_{\alpha,\,\beta},p_{\omega},q_{\omega})_{\alpha,\,\beta,\,\omega\in \Omega}$. Moreover, this differential $\delta_{\mathrm{alg}}$ is exactly the coboundary operator $\delta_{\mathrm{Alg}}$ of BiHom-$\Omega$-associative algebra $A$ as defined in Eq.~(\ref{BiHomOmHochschdiff}).
 \end{prop}
 \begin{proof}
	According to Corollary~\ref{coro: dg Lie}, we have $\delta_{\mathrm{alg}} \circ \delta_{\mathrm{alg}}$ = 0. Moreover,
 \begin{align*}
 	\delta_{\mathrm{alg}}^{n}(f)=&(-1)^{|f|}\delta (f)=(-1)^{n-1}[\mu,f]_{G}^{\Omega}\\
 	=&(-1)^{n-1}\Big(\sum_{i=1}^{2}(-1)^{(n-1)(i-1)}\mu\diamond_{i}^{\Omega}f-(-1)^{n-1}\sum_{i=1}^n(-1)^{i-1}f\diamond_{i}^{\Omega}\mu\Big)\\
 	=&\Big(\mu_{\alpha_{1},\,\alpha_{2},\dots,\alpha_{n+1}}(p_{\alpha_{1}}^{n-1}\otimes f_{\alpha_{2},\dots,\alpha_{n+1}})\\
 	&+\sum_{i=1}^n(-1)^if_{\alpha_{1},\dots,\alpha_{i-1},\alpha_{i}\alpha_{i+1},\,\alpha_{i+2},\dots,\alpha_{n+1}}\big(p_{\alpha_{1}}\otimes \cdots \otimes p_{\alpha_{i-1}}\otimes\mu_{\alpha_{i},\,\alpha_{i+1}}\otimes q_{\alpha_{i+2}}\otimes\cdots\otimes q_{\alpha_{n+1}}\big)\\
 	&+(-1)^{n-1}\mu_{\alpha_{1}\dots\alpha_{n},\,\alpha_{n+1}}(f_{\alpha_{1},\dots,\alpha_{n}}\otimes q_{\alpha_{n+1}}^{n-1})\Big)_{\alpha_{1},\dots,\alpha_{n+1}\in\Omega}\hspace{1cm}(\text{by Eq.~(\ref{eq-composition})})\\
 	=&\Big((\delta_{\mathrm{Alg}}^{n}f)_{\alpha_{1},\dots,\alpha_{n+1}}\Big)_{\alpha_{1},\dots,\alpha_{n+1}\in\Omega}.
 \end{align*}
This completes the proof.
 \end{proof}

 \medskip
\subsection{Cohomology of Rota-Baxter family on BiHom-$\Omega$-associative algebras}\label{sec-coho-RBf}

Let $ (A, \cdot_{\alpha,\,\beta }, R_{\omega},p_{\omega},\\q_{\omega})_{\alpha,\,\beta,\,\omega\in \Omega} $ be a Rota-Baxter family BiHom-$\Omega$-associative algebra of weight $\lambda$ and $ (M, \rhd_{\alpha,\,\beta},\lhd_{\alpha,\,\beta},\\T_{\omega},p_{\omega}^{M},q_{\omega}^{M})_{\alpha,\,\beta,\,\omega\in \Omega} $ be a Rota-Baxter family BiHom-$\Omega$-bimodule over $A$. According to Proposition~\ref{defstar} and Proposition~\ref{defrhd}, we get a new BiHom-$\Omega$-associative algebra $ A_{\star} $ and a new bimodule $M_{\star}$ over it. Now we define \[\mathrm{C}_{\mathrm{RBF}_{\lambda}}^{n}(A,M):=\mathrm{C}_{\Omega}^{n}(A_{\star},M_{\star}),\]
and a differential $\partial^{n} : \mathrm{C}_{\mathrm{RBF}_{\lambda}}^{n}(A,M)\longrightarrow \mathrm{C}_{\mathrm{RBF}_{\lambda}}^{n+1}(A,M)$
by
\[\big(\partial^0(m)\big)_{\alpha}(a):=a\blacktriangleright_{\alpha,1}m-m\blacktriangleleft_{1,\alpha}a=R_{\alpha}(a)\rhd_{\alpha,1}m-T_{\alpha}(a\rhd_{\alpha,1}m)-m\lhd_{1,\alpha}R_{\alpha}(a)+T_{\alpha}(m\lhd_{1,\alpha}a),\]
and
	\begin{align}\label{defof-partial}	&(\partial^{n}(f))_{\alpha_{1},\dots,\alpha_{n+1}}(a_{1},\dots,a_{n+1})\nonumber\\
	=&p_{\alpha_{1}}^{n-1}(a_{1})\blacktriangleright_{\alpha_{1},\alpha_{2}\dots\alpha_{n+1}}f_{\alpha_{2},\dots,
\alpha_{n+1}}(a_{2},\dots,a_{n+1})+\sum_{i=1}^{n}(-1)^{i}f_{\alpha_{1},\dots,\alpha_{i-1},\alpha_{i}\alpha_{i+1},
\alpha_{i+2},\dots,\alpha_{n+1}}(p_{\alpha_{1}}(a_{1}),\dots,p_{\alpha_{i-1}}(a_{i-1}),\nonumber\\
		& a_{i}\star_{\alpha_{i},\alpha_{i+1}}a_{i+1},q_{\alpha_{i+2}}(a_{i+2}),\dots,q_{\alpha_{n+1}}
(a_{n+1}))+(-1)^{n+1}f_{\alpha_{1},\dots,\alpha_{n}}\blacktriangleleft_{\alpha_{1}\dots\alpha_{n},
\alpha_{n+1}}q_{\alpha_{n+1}}^{n-1}(a_{n+1})\nonumber\\
		=&R_{\alpha_{1}}(p_{\alpha_{1}}^{n-1}(a_{1}))\rhd_{\alpha_{1},\alpha_{2}\dots\alpha_{n+1}}
f_{\alpha_{2},\dots,\alpha_{n+1}}(a_{2},\dots,a_{n+1})-T_{\alpha_{1}\dots\alpha_{n+1}}(p_{\alpha_{1}}^{n-1}(a_{1})
\rhd_{\alpha_{1},\alpha_{2}\dots\alpha_{n+1}}f_{\alpha_{2},\dots,\alpha_{n+1}}(a_{2},\dots,a_{n+1}))\nonumber\\
		&+\sum_{i=1}^{n}(-1)^{i}f_{\alpha_{1},\dots,\alpha_{i}\alpha_{i+1},\dots,\alpha_{n+1}}
(p_{\alpha_{1}}(a_{1}),\dots,p_{\alpha_{i-1}}(a_{i-1}),a_{i}\cdot_{\alpha_{i},\alpha_{i+1}}
R_{\alpha_{i+1}}(a_{i+1})+R_{\alpha_{i}}(a_{i})\cdot_{\alpha_{i},\alpha_{i+1}}a_{i+1}\\
		&+\lambda a_{i}\cdot_{\alpha_{i},\alpha_{i+1}}a_{i+1},q_{\alpha_{i+2}}(a_{i+2}),\dots,q_{\alpha_{n+1}}
(a_{n+1}))+(-1)^{n+1}f_{\alpha_{1},\dots,\alpha_{n}}(a_{1},\dots,a_{n})\lhd_{\alpha_{1}\dots\alpha_{n},\alpha_{n+1}}
R_{\alpha_{n+1}}q_{\alpha_{n+1}}^{n-1}(a_{n+1})\nonumber\\
		&-(-1)^{n+1}T_{\alpha_{1}\dots\alpha_{n+1}}(f_{\alpha_{1},\dots,\alpha_{n}}(a_{1},\dots,a_{n})
\lhd_{\alpha_{1}\dots\alpha_{n},\alpha_{n+1}}q_{\alpha_{n+1}}^{n-1}(a_{n+1})),\nonumber
\end{align}
for all $n\geq 1,\; a_{1},\dots,a_{n+1}\in A,\, \alpha_{1},\dots,\alpha_{n+1}\in \Omega$.

\begin{defn}
	We call $ (\mathrm{C}_{\mathrm{RBF}_{\lambda}}^{\bullet}(A,M),\partial^{\bullet}) $ the \textbf{cochain complex of Rota-Baxter family $ ( R_{\omega})_{\omega\in \Omega}$ of weight $\lambda$} on BiHom-$\Omega$-associative algebra $A$ with coefficients in bimodule $M$. Its cohomology, denote by  $ \mathrm{H}_{\mathrm{RBF}_{\lambda}}^{\bullet}(A,M) $, is called the \textbf{cohomology of Rota-Baxter family $ ( R_{\omega})_{\omega\in \Omega}$ of weight $\lambda$} on BiHom-$\Omega$-associative algebra $A$ with coefficients in bimodule $M$. 
\end{defn}

	In particular, when $M$ is the regular bimodule, the cochain complex $(\mathrm{C}_{\mathrm{RBF}_{\lambda}}^{\bullet}(A,A),\partial^{\bullet})$ is simply denoted by $(\mathrm{C}_{\mathrm{RBF}_{\lambda}}^{\bullet}(A),\partial^{\bullet})$. The corresponding cohomology, simply denoted by $\mathrm{H}_{\mathrm{RBF}_{\lambda}}^{\bullet}(A)$, is called the cohomology of Rota-Baxter family $ ( R_{\omega})_{\omega\in \Omega} $.

\begin{remark}
A 1-cocycle in $\mathrm{C}_{\mathrm{RBF}_{\lambda}}^{1}(A,M)$ is a family of linear maps $(f_{\alpha})_{\alpha\in\Omega}: A\rightarrow M$ satisfying
	\begin{align*}
		&\hspace{2cm}p_{\alpha}^{M}\circ f_{\alpha}=f_{\alpha}\circ p_{\alpha},\quad q_{\alpha}^{M}\circ f_{\alpha}=f_{\alpha}\circ q_{\alpha},\\
		(\partial^{1}f)_{\alpha,\,\beta}(x,y)=&R_{\alpha}(x)\rhd_{\alpha,\,\beta}f_{\beta}(y)-T_{\alpha\,\beta}(x\rhd_{\alpha,\,\beta}f_{\beta}(y))-f_{\alpha\,\beta}(x\cdot_{\alpha,\,\beta}R_{\beta}(y)+R_{\alpha}(x)\cdot_{\alpha,\,\beta}y+\lambda x\cdot_{\alpha,\,\beta}y)\\
		&+f_{\alpha}(x)\lhd_{\alpha,\,\beta}R_{\beta}(y)-T_{\alpha\,\beta}(f_{\alpha}(x)\lhd_{\alpha,\,\beta}y)=0,
	\end{align*}
for all $x,y\in A,\,\alpha,\,\beta\in\Omega$.
\end{remark}

\subsection{Cohomology of Rota-Baxter family BiHom-$\Omega$-associative algebras}

In this subsection, we will combine the cohomology of  BiHom-$\Omega$-associative algebras and the cohomology of Rota-Baxter family on BiHom-$\Omega$-associative algebras to study the cohomology theory for Rota-Baxter family BiHom-$\Omega$-associative algebras.

Let $ (M,\rhd_{\alpha,\,\beta},\lhd_{\alpha,\,\beta},T_{\omega},p_{\omega}^{M},q_{\omega}^{M})_{\alpha,\,\beta,\,\omega\in \Omega} $ be a Rota-Baxter family BiHom-$\Omega$-bimodule over Rota-Baxter family BiHom-$\Omega$-associative algebra $(A,\cdot_{\alpha,\,\beta},R_{\omega},p_{\omega},q_{\omega})_{\alpha,\,\beta,\,\omega\in \Omega}$. Now, let's construct a chain map
\[\Phi^{\bullet} : \mathrm{C}_{\Omega}^{\bullet}(A,M) \rightarrow \mathrm{C}_{\mathrm{RBF}_{\lambda}}^{\bullet}(A,M),\]
that is
\begin{align*}
	\xymatrix{
		\mathrm{C}^{0}_{\Omega}(A,M)\ar[r]^-{\delta_{\mathrm{Alg}}^0}\ar[d]^-{\Phi^0}& \mathrm{C}^{1}_{\Omega}(A,M)\ar@{.}[r]\ar[d]^-{\Phi^1}& \mathrm{C}^{n}_{\Omega}(A,M)\ar[r]^-{\delta_{\mathrm{Alg}}^n}\ar[d]^-{\Phi^n}& \mathrm{C}^{n+1}_{\Omega}(A,M)\ar[d]^{\Phi^{n+1}}\ar@{.}[r]&\\
		C ^0_{\mathrm{RBF}_{\lambda}}(A,M)\ar[r]^-{\partial^0}& C ^1_{\mathrm{RBF}_{\lambda}}(A,M)\ar@{.}[r]& C ^n_{\mathrm{RBF}_{\lambda}}(A,M)\ar[r]^-{\partial^n}& C ^{n+1}_{\mathrm{RBF}_{\lambda}}(A,M)\ar@{.}[r]&	.}
\end{align*}

Define $ \Phi^{0}=\text{Id}_{\mathrm{Hom}(\bfk,M)}=\text{Id}_{M} $. For $ n=1 $ and $ f=(f_{\alpha})_{\alpha\in\Omega}\in \mathrm{C}_{\Omega}^{1}(A,M) $, we define
\begin{align}\label{defof-Phi-1}
	\Phi^{1}(f)_{\alpha}(a):=f_{\alpha}(R_{\alpha}(a))-T_{\alpha}(f_{\alpha}(a)) , \;\;\text{for all }\alpha\in \Omega,\,a\in A.
\end{align}
For $ n\geq 2 $ and $ f=(f_{\alpha_{1},\dots,
	\alpha_{n}})_{\alpha_{1},\dots,\alpha_{n}\in\Omega} \in \mathrm{C}_{\Omega}^{n}(A,M) $, we define

\begin{equation}\label{defof-Phi}
	\begin{split}
		&\quad\Phi^{n}(f)_{\alpha_{1},\dots,\alpha_{n}}(a_{1},\dots,a_{n})\\
		&:=f_{\alpha_{1},\dots,\alpha_{n}}(R_{\alpha_{1}}(a_{1}),\dots,R_{\alpha_{n}}(a_{n}))-\sum_{k=0}^{n-1} \lambda^{n-k-1}\underset{1\leq i_{1}<i_{2}<\dots<i_{k}\leq n}{\sum}T_{\alpha_{1}\dots\alpha_{n}}\circ f_{\alpha_{1},\dots,\alpha_{n}}\\
		&(a_{1},\dots,a_{i_{1}-1},R_{\alpha_{i_{1}}}(a_{i_{1}}),a_{i_{1}+1},\dots,a_{i_{2}-1},R_{\alpha_{i_{2}}}(a_{i_{2}}),a_{i_{2}+1},\dots,a_{i_{k}-1},R_{\alpha_{i_{k}}}(a_{i_{k}}),a_{i_{k}+1},\dots,a_{n}),
	\end{split}
\end{equation}
for all $ a_{1},\dots,a_{n}\in A,\; \alpha_{1},\dots,\alpha_{n}\in \Omega$.

Similar to~\cite[Proposition III.5]{RBFassoconformal}, we get $ \partial^n\circ\Phi^{n}=\Phi^{n+1}\circ\delta_{\mathrm{Alg}}^n $, i.e. the map $ \Phi^{\bullet} : \mathrm{C}_{\Omega}^{\bullet}(A,M) \rightarrow \mathrm{C}_{\mathrm{RBF}_{\lambda}}^{\bullet}(A,M) $ is a chain map.

\begin{defn}\label{diffof-RBFbihomOassoalg}
	Let $ (M,\rhd_{\alpha,\,\beta},\lhd_{\alpha,\,\beta},T_{\omega},p_{\omega}^{M},q_{\omega}^{M})_{\alpha,\,\beta,\,\omega\in \Omega} $ be a Rota-Baxter family BiHom-$\Omega$-bimodule over the Rota-Baxter family BiHom-$\Omega$-associative algebra $(A,\cdot_{\alpha,\,\beta},R_{\omega},p_{\omega},q_{\omega})_{\alpha,\,\beta,\,\omega\in \Omega}$. We call $ (\mathrm{C}_{\mathrm{RBFA}_{\lambda}}^{\bullet}(A,M) , d^{\bullet}) $ the \textbf{cochain complex of Rota-Baxter family BiHom-$\Omega$-associative algebra $ A$ with coefficients in $ M$}, where
	\[\mathrm{C}_{\mathrm{RBFA}_{\lambda}}^{0}(A,M)=\mathrm{C}_{\Omega}^{0}(A,M),\]
	\[ \mathrm{C}_{\mathrm{RBFA}_{\lambda}}^{n}(A,M)=\mathrm{C}_{\Omega}^{n}(A,M)\oplus\mathrm{C}_{\mathrm{RBF}_{\lambda}}^{n-1}(A,M),\quad \text{for all } n\geq 1,\]
	and the differential $ d^{n}: \mathrm{C}_{\mathrm{RBFA}_{\lambda}}^{n}(A,M)\rightarrow\mathrm{C}_{\mathrm{RBFA}_{\lambda}}^{n+1}(A,M) $ is given by
	\[d^{n}(f,g)_{\alpha_{1},\dots,\,\alpha_{n+1},\,\beta_{1},\dots,\,\beta_{n}}=(\delta_{\mathrm{Alg}}^{n}(f)_{\alpha_{1},\dots,\alpha_{n+1}},-\partial^{n-1}(g)_{\beta_{1},\dots,\,\beta_{n}}-\Phi^{n}(f)_{\beta_{1},\dots,\,\beta_{n}})\]
	for any $ f\in \mathrm{C}_{\Omega}^{n}(A,M),\; g\in \mathrm{C}_{\mathrm{RBF}_{\lambda}}^{n-1}(A,M)$ and $\alpha_{1},\dots,\alpha_{n+1},\,\beta_{1},\dots,\beta_{n}\in \Omega$. Its cohomology, 
	denoted by $ \mathrm{H}_{\mathrm{RBFA}_{\lambda}}^{\bullet}(A,M) $,
	is called the \textbf{cohomology of Rota-Baxter family BiHom-$\Omega$-associative algebra $ A $ with coefficients in $M$}.
\end{defn}

In particular, when $M$ is the Rota-Baxter family BiHom-$\Omega$-bimodule, the cochain complex $(\mathrm{C}^{\bullet}_{\mathrm{RBFA}_{\lambda}}(A,A),d^{\bullet})$ is simply denoted by $(\mathrm{C}_{\mathrm{RBFA}_{\lambda}}^{\bullet}(A),d^{\bullet})$. The corresponding cohomology, simply denoted by $ \mathrm{H}_{\mathrm{RBFA}_{\lambda}}^{\bullet}(A)$, is called the cohomology of Rota-Baxter family BiHom-$\Omega$-associative algebra $A$.

\begin{remark}\label{RBFA-2-cocycle}
A pair $(f_{\alpha_{1},\,\alpha_{2}},h_{\beta_{1}})_{\alpha_{1},\,\alpha_{2},\,\beta_{1}\in\Omega}$ is called a 2-cocycle in $\mathrm{C}_{\mathrm{RBFA}_{\lambda}}^{2}(A,M)$ if $(f_{\alpha_{1},\,\alpha_{2}})_{\alpha_{1},\,\alpha_{2}\in\Omega}\in\mathrm{C}_{\Omega}^{2}(A,M)$ and $(h_{\beta_{1}})_{\beta_{1}\in\Omega}\in\mathrm{C}_{\Omega}^{1}(A,M)$ satisfy
\[d^{2}(f,h)_{\alpha_{1},\,\alpha_{2},\,\alpha_{3},\,\beta_{1},\,\beta_{2}}=0,\]
	i.e. $\delta_{\mathrm{Alg}}^{2}(f)_{\alpha_{1},\,\alpha_{2},\,\alpha_{3}}=0$ and $-\partial^{1}(h)_{\beta_{1},\,\beta_{2}}=\Phi^{2}(f)_{\beta_{1},\,\beta_{2}}$, for all $\alpha_{1},\,\alpha_{2},\,\alpha_{3},\,\beta_{1},\,\beta_{2}\in\Omega$.
\end{remark}

\section{Deformations of Rota-Baxter family BiHom-$\Omega$-associative algebras}\label{sec4}
In this section, we will study the deformations of BiHom-$\Omega$-associative algebras and Rota-Baxter family BiHom-$\Omega$-associative algebras.

\subsection{Deformations of BiHom-$\Omega$-associative algebras}
In this subsection, we study linear deformations of BiHom-$\Omega$-associative algebras. The results of this section are similar to classical ones about deformation of associative algebras \cite{assoalgdeforma}.

\begin{defn}
	A \textbf{linear deformation of BiHom-$\Omega$-associative algebra} $(A,\mu_{\alpha,\,\beta},p_{\omega},q_{\omega})_{\alpha,\,\beta,\,\omega\in \Omega}$ is a parametrized sum $ \mu_{\alpha,\,\beta}^{t}=\mu_{\alpha,\,\beta}+t\mu_{\alpha,\,\beta}^{1}$ consisting of the multiplication $(\mu_{\alpha,\,\beta})_{\alpha,\,\beta\in \Omega}$ and a family of bilinear maps $(\mu_{\alpha,\,\beta}^1)_{\alpha,\,\beta\in\Omega} : A\otimes A\rightarrow A$ such that
	$(A[[t]]/(t^{2}),\mu_{\alpha,\,\beta}^t,p_{\omega},q_{\omega})_{\alpha,\,\beta,\,\omega\in \Omega}$ is a BiHom-$\Omega$-associative algebra. In this case, we say that $(\mu_{\alpha,\,\beta}^1)_{\alpha,\,\beta  \in \Omega}$ is a family of deformations of the BiHom-$\Omega$-associative algebra $A$. 
\end{defn}

Therefore, for a linear deformation $ \mu_{\alpha,\,\beta}^{t}=\mu_{\alpha,\,\beta}+t\mu_{\alpha,\,\beta}^{1}$, we must have
\[p_{\alpha\,\beta}\circ\mu_{\alpha,\,\beta}^t(a,b)=\mu_{\alpha,\,\beta}^t(p_{\alpha}(a),p_{\beta}(b)),\quad q_{\alpha\,\beta}\circ\mu_{\alpha,\,\beta}^t(a,b)=\mu_{\alpha,\,\beta}^t(q_{\alpha}(a),q_{\beta}(b)),\]
\[\mu_{\alpha\,\beta,\,\gamma}^t\big(\mu_{\alpha,\,\beta}^t(a,b),q_{\gamma}(c)\big)=\mu_{\alpha,\,\beta\,\gamma}^t\big(p_{\alpha}(a),\mu_{\beta,\,\gamma}^t(b,c)\big),\]
for all $a,b,c\in A,\,\alpha,\,\beta,\,\gamma\in \Omega$. By equating the coefficients of $t$ and $t^2$, we get
\begin{align}
	p_{\alpha\,\beta}\circ\mu_{\alpha,\,\beta}^1(a,b)=\mu_{\alpha,\,\beta}^1(p_{\alpha}(a),p_{\beta}(b)),&\;\; q_{\alpha\,\beta}\circ\mu_{\alpha,\,\beta}^1(a,b)=\mu_{\alpha,\,\beta}^1(q_{\alpha}(a),q_{\beta}(b)),\label{linear-deform-1}\\
	\mu_{\alpha\,\beta,\,\gamma}(\mu_{\alpha,\,\beta}^1(a,b),\,q_{\gamma}(c))+\mu_{\alpha\,\beta,\,\gamma}^1(\mu_{\alpha,\,\beta}(a,b),q_{\gamma}(c))&=\mu_{\alpha,\,\beta\,\gamma}(p_{\alpha}(a),\,\mu_{\beta,\,\gamma}^1(b,c))\nonumber\\
	&\;\;+\mu_{\alpha,\,\beta\,\gamma}^1(p_{\alpha}(a),\mu_{\beta,\,\gamma}(b,c)),\label{linear-deform-2}\\
	\mu_{\alpha\,\beta,\,\gamma}^1\big(\mu_{\alpha,\,\beta}^1(a,b),q_{\gamma}(c)\big)&=\mu_{\alpha,\,\beta\,\gamma}^1\big(p_{\alpha}(a),\mu_{\beta,\,\gamma}^1(b,c)\big),\label{linear-deform-3}
\end{align}
Hence, by comparing Eqs.~(\ref{Hoc1})-(\ref{cocycle}) and Eqs.~(\ref{linear-deform-1})-(\ref{linear-deform-2}), we obtain that the family of deformations $(\mu^1_{\alpha,\,\beta})_{\alpha,\,\beta\in \Omega}$ is a 2-cocycle in $\mathrm{C}_{\Omega}^{2}(A)$. Moreover, by Eq.~(\ref{linear-deform-1}) and Eq.~(\ref{linear-deform-3}), we know that $(A,\mu_{\alpha,\,\beta}^1,p_{\omega},q_{\omega})_{\alpha,\,\beta,\,\omega\in \Omega}$ is a BiHom-$\Omega$-associative algebra.

\smallskip
Next, we introduce the definition of trivial deformations.
\begin{defn}
	Let $( N_{\omega})_{\omega\in \Omega} :A\rightarrow A$ be a family of linear maps. A family of deformations $( \mu_{\alpha,\,\beta}^1)_{\alpha,\,\beta\in \Omega}$ is said to be \textbf{trivial} if $(T_{\omega}^t)_{\omega\in\Omega}=(\text{id}+tN_{\omega})_{\omega\in\Omega}$ satisfies
	\begin{gather}
		p_{\alpha}\circ T_{\alpha}^t=T_{\alpha}^t\circ p_{\alpha},\quad q_{\alpha}\circ T_{\alpha}^t=T_{\alpha}^t\circ q_{\alpha},\label{tri-1}\\
		T_{\alpha\,\beta}^t\circ \mu_{\alpha,\,\beta}^t(a,b)=\mu_{\alpha,\,\beta}(T_{\alpha}^t(a),\,T_{\beta}^t(b)),\label{tri-2}
	\end{gather}
	for all $a,b\in A,\,\alpha,\,\beta\in \Omega$.
\end{defn}

\noindent Expanding the both sides of Eq.~(\ref{tri-1}), we have
\[p_{\alpha}\circ T_{\alpha}^t=p_{\alpha}\circ(id+tN_{\alpha})=p_{\alpha}+tp_{\alpha}\circ N_{\alpha},\]
\[T_{\alpha}^t\circ p_{\alpha}=(id+tN_{\alpha})\circ p_{\alpha}=p_{\alpha}+tN_{\alpha}\circ p_{\alpha}.\]
Similarly, we get
\[q_{\alpha}\circ T_{\alpha}^t=q_{\alpha}+tq_{\alpha}\circ N_{\alpha},\quad T_{\alpha}^t\circ q_{\alpha}=q_{\alpha}+tN_{\alpha}\circ q_{\alpha}.\]
For Eq.~(\ref{tri-2}), we have
\begin{align*}
	T_{\alpha\,\beta}^t\circ \mu_{\alpha,\,\beta}^t(a,b)=&(id+tN_{\alpha\,\beta})(\mu_{\alpha,\,\beta}+t\mu_{\alpha,\,\beta}^1)(a,b)\\
	=&\mu_{\alpha,\,\beta}(a,b)+t(\mu_{\alpha,\,\beta}^1(a,b)+N_{\alpha\,\beta}\mu_{\alpha,\,\beta}(a,b))+t^2N_{\alpha\,\beta}\mu_{\alpha,\,\beta}^1(a,b),\\
	\mu_{\alpha,\,\beta}(T_{\alpha}^t(a),\,T_{\beta}^t(b))=&\mu_{\alpha,\,\beta}\big((id+tN_{\alpha})(a),\,(id+tN_{\beta})(b)\big)\\
	=&\mu_{\alpha,\,\beta}\big(a+tN_{\alpha}(a),\,b+tN_{\beta}(b)\big)\\
	=&\mu_{\alpha,\,\beta}(a,b)+t\big(\mu_{\alpha,\,\beta}(a,\,N_{\beta}(b))+\mu_{\alpha,\,\beta}(N_{\alpha}(a),\,b)\big)+t^2\mu_{\alpha,\,\beta}\big(N_{\alpha}(a),\,N_{\beta}(b)\big).
\end{align*}
By comparing the coefficient of $t$ and $t^2$ on both sides of the equations, we obtain that the triviality of deformation is equivalent to the following equations:
\begin{gather}
	N_{\alpha}\circ p_{\alpha}=p_{\alpha}\circ N_{\alpha},\quad N_{\alpha}\circ q_{\alpha}=q_{\alpha}\circ N_{\alpha},\label{tri-3}\\
	\mu_{\alpha,\,\beta}^1(a,b)=\mu_{\alpha,\,\beta}(a,\,N_{\beta}(b))+\mu_{\alpha,\,\beta}(N_{\alpha}(a),\,b)-N_{\alpha\,\beta}\circ\mu_{\alpha,\,\beta}(a,b),\label{tri-4}\\
	N_{\alpha\,\beta}\circ\mu_{\alpha,\,\beta}^1(a,b)=\mu_{\alpha,\,\beta}(N_{\alpha}(a),\,N_{\beta}(b)).\label{tri-5}
\end{gather}
It follows from Eqs.~(\ref{tri-3})-(\ref{tri-5}) that $( N_{\omega})_{\omega\in \Omega}$ must satisfy the following conditions:
\begin{gather}
	N_{\alpha}\circ p_{\alpha}=p_{\alpha}\circ N_{\alpha},\quad N_{\alpha}\circ q_{\alpha}=q_{\alpha}\circ N_{\alpha},\label{Nijf-def-1}\\
	\mu_{\alpha,\,\beta}(N_{\alpha}\otimes N_{\beta})=N_{\alpha\,\beta}\big(\mu_{\alpha,\,\beta}(id\otimes N_{\beta})+\mu_{\alpha,\,\beta}(N_{\alpha}\otimes id)-N_{\alpha\,\beta}\circ\mu_{\alpha,\,\beta}(id\otimes id)\big)\label{Nijf-def-2}.
\end{gather}
We call a family of linear maps $( N_{\omega})_{\omega\in \Omega}: A\rightarrow A$ a Nijenhuis family on BiHom-$\Omega$-associative algebra $(A,\mu_{\alpha,\,\beta},p_{\omega},q_{\omega})_{\alpha,\,\beta,\,\omega\in \Omega}$ if $(N_{\omega})_{\omega\in \Omega}$ satisfies Eqs.~(\ref{Nijf-def-1})-(\ref{Nijf-def-2}), which is a generalization of the classical Nijenhuis operator~\cite{assoalgdeforma,Nijen1,Nijen2}.

\begin{prop}\label{a+b}
	Let $(N_{\omega})_{\omega\in \Omega}$ be a Nijenhuis family on BiHom-$\Omega$-associative algebra $(A,\mu_{\alpha,\,\beta},p_{\omega},\\q_{\omega})_{\alpha,\,\beta,\,\omega\in \Omega}$. If we define the operation on $A$ by
	\begin{align*}
		\mu_{\alpha,\,\beta}^N(a,b):=\mu_{\alpha,\,\beta}(N_{\alpha}(a),\,b)+\mu_{\alpha,\,\beta}(a ,\,N_{\beta}(b))-N_{\alpha\,\beta}\circ\mu_{\alpha,\,\beta}(a,b),
	\end{align*}
	for all $a,b\in A,\,\alpha,\,\beta\in \Omega$. Then
	\begin{enumerate}
		\item \label{prop-Nijf-BHOasso}	the quadruple $(A,\mu_{\alpha,\,\beta}^N,p_{\omega},q_{\omega})_{\alpha,\,\beta,\,\omega\in \Omega}$ is a new BiHom-$\Omega$-associative algebra. Moreover, $( N_{\omega})_{\omega\in \Omega}$ is a BiHom-$\Omega$-associative algebra homomorphism from $(A,\mu_{\alpha,\,\beta}^N,p_{\omega},q_{\omega})_{\alpha,\,\beta,\,\omega\in \Omega}$ to $(A,\mu_{\alpha,\,\beta},p_{\omega},q_{\omega})_{\alpha,\,\beta,\,\omega\in \Omega}$.
	\item \label{th-trideform} the family of linear maps $( \mu_{\alpha,\,\beta}^N)_{\alpha,\,\beta\in \Omega}$ is a trivial deformation of $A$.
\end{enumerate}

\end{prop}
\begin{proof}
	\ref{prop-Nijf-BHOasso}. For any $a,b,c\in A,\,\alpha,\,\beta,\,\gamma\in \Omega,$ we first prove Eq.~(\ref{eqalfabeta}) for $(A,\mu_{\alpha,\,\beta}^{N},p_{\omega},q_{\omega})_{\alpha,\,\beta,\,\gamma\in\Omega}$. 
	\begin{align*}
		p_{\alpha\,\beta}\circ \mu_{\alpha,\,\beta}^N(a,b)=&p_{\alpha\,\beta}\big(\mu_{\alpha,\,\beta}(N_{\alpha}(a),\,b)+\mu_{\alpha,\,\beta}(a,\,N_{\beta}(b))-N_{\alpha\,\beta}\circ\mu_{\alpha,\,\beta}(a,b)\big)\\
		=&\mu_{\alpha,\,\beta}(p_{\alpha}N_{\alpha}(a),\,p_{\beta}(b))+\mu_{\alpha,\,\beta}(p_{\alpha}(a),\,p_{\beta}N_{\beta}(b))-p_{\alpha\,\beta}N_{\alpha\,\beta}\mu_{\alpha,\,\beta}(a,b)\\
		=&\mu_{\alpha,\,\beta}(N_{\alpha}p_{\alpha}(a),\,p_{\beta}(b))+\mu_{\alpha,\,\beta}(p_{\alpha}(a),\,N_{\beta}p_{\beta}(b))-N_{\alpha\,\beta}p_{\alpha\,\beta}\mu_{\alpha,\,\beta}(a,b)\\
		&\hspace{1cm}(\text{by Eq.~(\ref{Nijf-def-1})})\\
		=&\mu_{\alpha,\,\beta}(N_{\alpha}p_{\alpha}(a),\,p_{\beta}(b))+\mu_{\alpha,\,\beta}(p_{\alpha}(a),\,N_{\beta}p_{\beta}(b))-N_{\alpha\,\beta}\mu_{\alpha,\,\beta}(p_{\alpha}(a),\,p_{\beta}(b))\\
		&\hspace{1cm}(\text{by Eq.~(\ref{eqalfabeta})})\\
		=&\mu_{\alpha,\,\beta}^N(p_{\alpha}(a),p_{\beta}(b)).
	\end{align*}
	Similarly, we get $ q_{\alpha\,\beta}\circ \mu_{\alpha,\,\beta}^N(a,b)=\mu_{\alpha,\,\beta}^N(q_{\alpha}(a),q_{\beta}(b)).$ Next, we prove Eq.~(\ref{eqasso}).
\begin{align*}
	&\mu_{\alpha\,\beta,\,\gamma}^{N}\big(\mu_{\alpha,\,\beta}^{N}(a,b),q_{\gamma}(c)\big)\\
	=&\mu_{\alpha\,\beta,\,\gamma}\big(N_{\alpha\,\beta}\mu_{\alpha,\,\beta}^{N}(a,b),q_{\gamma}(c)\big)+\mu_{\alpha\,\beta,\,\gamma}\big(\mu_{\alpha,\,\beta}^{N}(a,b),N_{\gamma}q_{\gamma}(c)\big)-N_{\alpha\,\beta\,\gamma}\mu_{\alpha\,\beta,\,\gamma}\big(\mu_{\alpha,\,\beta}^{N}(a,b),q_{\gamma}(c)\big)\\
	=&\mu_{\alpha\,\beta,\,\gamma}\Big(\mu_{\alpha,\,\beta}\big(N_{\alpha}(a),N_{\beta}(b)\big),q_{\gamma}(c)\Big)+\mu_{\alpha\,\beta,\,\gamma}\Big(\mu_{\alpha,\,\beta}\big(N_{\alpha}(a),b\big)+\mu_{\alpha,\,\beta}\big(a,N_{\beta}(b)\big)-N_{\alpha\,\beta}\mu_{\alpha,\,\beta}(a,b),q_{\gamma}N_{\gamma}(c)\Big)\\
	&-\mu_{\alpha\,\beta,\,\gamma}\big(N_{\alpha\,\beta}\mu_{\alpha,\,\beta}^{N}(a,b),N_{\gamma}q_{\gamma}(c)\big)\\
	=&\mu_{\alpha\,\beta,\,\gamma}\Big(\mu_{\alpha,\,\beta}\big(N_{\alpha}(a),N_{\beta}(b)\big),q_{\gamma}(c)\Big)+\mu_{\alpha\,\beta,\,\gamma}\Big(\mu_{\alpha,\,\beta}\big(N_{\alpha}(a),b\big),q_{\gamma}N_{\gamma}(c)\Big)+\mu_{\alpha\,\beta,\,\gamma}\Big(\mu_{\alpha,\,\beta}\big(a,N_{\beta}(b)\big),q_{\gamma}N_{\gamma}(c)\Big)\\
	&-\mu_{\alpha\,\beta,\,\gamma}\Big(\mu_{\alpha,\,\beta}\big(N_{\alpha}(a),N_{\beta}(b)\big),q_{\gamma}N_{\gamma}(c)\Big)-\mu_{\alpha\,\beta,\,\gamma}\Big(\mu_{\alpha,\,\beta}\big(N_{\alpha}(a),N_{\beta}(b)\big),q_{\gamma}N_{\gamma}(c)\Big)\\
	=&\mu_{\alpha,\,\beta\,\gamma}\Big(p_{\alpha}N_{\alpha}(a),\mu_{\beta,\,\gamma}\big(N_{\beta}(b),c\big)\Big)+\mu_{\alpha,\,\beta\,\gamma}\Big(p_{\alpha}N_{\alpha}(a),\mu_{\beta,\,\gamma}\big(b,N_{\gamma}(c)\big)\Big)+\mu_{\alpha,\,\beta\,\gamma}\Big(p_{\alpha}(a),\mu_{\beta,\,\gamma}\big(N_{\beta}(b),N_{\gamma}(c)\big)\Big)\\
	&-\mu_{\alpha,\,\beta\,\gamma}\Big(p_{\alpha}N_{\alpha}(a),\mu_{\beta,\,\gamma}\big(N_{\beta}(b),N_{\gamma}(c)\big)\Big)-\mu_{\alpha,\,\beta\,\gamma}\Big(p_{\alpha}N_{\alpha}(a),\mu_{\beta,\,\gamma}\big(N_{\beta}(b),N_{\gamma}(c)\big)\Big)\\
	=&\mu_{\alpha,\,\beta\,\gamma}\Big(N_{\alpha}p_{\alpha}(a),\mu_{\beta,\,\gamma}\big(N_{\beta}(b),c\big)+\mu_{\beta,\,\gamma}\big(b,N_{\gamma}(c)\big)-\mu_{\beta,\,\gamma}\big(N_{\beta}(b),N_{\gamma}(c)\big)\Big)+\mu_{\alpha,\,\beta\,\gamma}\big(p_{\alpha}(a),N_{\beta\,\gamma}\mu_{\beta,\,\gamma}(b,c)\big)\\
	&-N_{\alpha\,\beta\,\gamma}\mu_{\alpha,\,\beta\,\gamma}\big(p_{\alpha}(a),\mu_{\beta,\,\gamma}^{N}(b,c)\big)\\
	=&\mu_{\alpha,\,\beta\,\gamma}\big(N_{\alpha}p_{\alpha}(a),\mu_{\beta,\,\gamma}^{N}(b,c)\big)+\mu_{\alpha,\,\beta\,\gamma}\big(p_{\alpha}(a),N_{\beta\,\gamma}\mu_{\beta,\,\gamma}^{N}(b,c)\big)-N_{\alpha\,\beta\,\gamma}\mu_{\alpha,\,\beta\,\gamma}\big(p_{\alpha}(a),\mu_{\beta,\,\gamma}^{N}(b,c)\big)\\
	=&\mu_{\alpha,\,\beta\,\gamma}^{N}\big(p_{\alpha}(a),\mu_{\beta,\,\gamma}^{N}(b,c)\big)
\end{align*}
\noindent So we obtain that $(A,\mu_{\alpha,\,\beta}^N,p_{\omega},q_{\omega})_{\alpha,\,\beta,\,\omega\in \Omega}$ is a BiHom-$\Omega$-associative algebra. Furthermore, we have
	\begin{align*}
		\mu_{\alpha,\,\beta}\big(N_{\alpha}(a),\,N_{\beta}(b)\big)=&N_{\alpha\,\beta}\Big(\mu_{\alpha,\,\beta}\big(N_{\alpha}(a),\,b\big)+\mu_{\alpha,\,\beta}\big(a,\,N_{\beta}(b)\big)-N_{\alpha\,\beta}\mu_{\alpha,\,\beta}(a,b)\Big)\hspace{1cm}(\text{by Eq.~(\ref{Nijf-def-2})})\\
		=&N_{\alpha\,\beta}\circ\mu_{\alpha,\,\beta}^N(a,b),
	\end{align*}
	 then by Eq.~(\ref{Nijf-def-1}), we get that $( N_{\omega})_{\omega\in \Omega}$ is a BiHom-$\Omega$-associative algebra homomorphism. This completes the proof.

\ref{th-trideform}. First, we are going to prove that $\mu_{\alpha,\,\beta}+t\mu_{\alpha,\,\beta}^N$ is a linear deformation of $A$. By Item~\ref{prop-Nijf-BHOasso}, we get Eq.~(\ref{linear-deform-1}) and Eq.~(\ref{linear-deform-3}). So we only need to check Eq.~(\ref{linear-deform-2}) for $\mu_{\alpha,\,\beta}^N$, we have
	\begin{align*}
		\mu&_{\alpha,\,\beta\,\gamma}(p_{\alpha}\otimes \mu_{\beta,\,\gamma}^N)-\mu_{\alpha\,\beta,\,\gamma}^N(\mu_{\alpha,\,\beta}\otimes q_{\gamma})+\mu_{\alpha,\,\beta\,\gamma}^N(p_{\alpha}\otimes \mu_{\beta,\,\gamma})-\mu_{\alpha\,\beta,\,\gamma}(\mu_{\alpha,\,\beta}^N\otimes q_{\gamma})\\
		=&\delta_{\mathrm{Alg}}^2(\mu_{\alpha,\,\beta}^N)\hspace{1cm}(\text{by Eq.~(\ref{cocycle})})\\
		=&\delta_{\mathrm{Alg}}^2\delta_{\mathrm{Alg}}^1(N_{\alpha})=0.
	\end{align*}
	So we get Eq.~(\ref{linear-deform-2}). Hence $\mu_{\alpha,\,\beta}+t\mu_{\alpha,\,\beta}^N$ is a linear deformation of $A$. Next, we verify the triviality of $\mu_{\alpha,\,\beta}^N$. We just need to prove Eqs.~(\ref{tri-3})-(\ref{tri-5}). By Item~\ref{prop-Nijf-BHOasso} and the definition of $\mu_{\alpha,\,\beta}^N$, we get Eqs.~(\ref{tri-3})-(\ref{tri-5}). Thus, $( \mu_{\alpha,\,\beta}^N)_{\alpha,\,\beta\in \Omega}$ is a trivial deformation. This completes the proof.
\end{proof}

\begin{remark}
	 By Proposition~\ref{a+b}, we have a 2-cochain $(\psi_{\alpha,\,\beta}^N)_{\alpha,\,\beta\in \Omega}\in \mathrm{C}_{\Omega}^2(A)$ as follows.
	\begin{align}\label{def-psi}
		\psi_{\alpha,\,\beta}^N(a,b)=\mu_{\alpha,\,\beta}\big(N_{\alpha}(a),N_{\beta}(b)\big)-N_{\alpha\,\beta}\mu_{\alpha,\,\beta}^N(a,b),
	\end{align}
	for all $a,b\in A,\,\alpha,\,\beta\in \Omega$. It is obvious that $(\psi_{\alpha,\,\beta}^N)_{\alpha,\,\beta\in\Omega}=0$ if and only if $( N_{\omega})_{\omega\in \Omega}$ is a Nijenhuis family on $A$.
\end{remark}

Now we arrive at our main results in this subsection as follows.
\begin{theorem}
	Let $(A,\mu_{\alpha,\,\beta},p_{\omega},q_{\omega})_{\alpha,\,\beta,\,\omega\in \Omega}$ be a BiHom-$\Omega$-associative algebra. If $(\mu_{\alpha,\,\beta}^{N})_{\alpha,\,\beta\in\Omega}$ is defined by Proposition~\ref{a+b}, then
	\begin{enumerate}
		\item \label{it:bihom-ass}
		the quadruple $(A[[t]]/(t^{2}),\mu_{\alpha,\,\beta}+t\mu_{\alpha,\,\beta}^{N},p_{\omega},q_{\omega})_{\alpha,\,\beta,\,\omega\in \Omega}$ is a BiHom-$\Omega$-associative algebra.
		\item \label{it:2-cocycle}
 the quadruple $(A,\mu_{\alpha,\,\beta}^N,p_{\omega},q_{\omega})_{\alpha,\,\beta,\,\omega\in \Omega}$ is a BiHom-$\Omega$-associative algebra if and only if $( \psi_{\alpha,\,\beta}^N)_{\alpha,\,\beta\in \Omega}$ is a 2-cocycle in $\mathrm{C}_{\Omega}^2(A)$.
	\end{enumerate}
\end{theorem}
\begin{proof}
	\ref{it:bihom-ass}. For any $a,b,c\in A,\,\alpha,\,\beta,\,\gamma\in \Omega$, we only need to verify that the multiplication $\mu_{\alpha,\,\beta}+t\mu_{\alpha,\,\beta}^N$ satisfy Eqs.~(\ref{eqalfabeta})-(\ref{eqasso}). First of all, by Eq.~(\ref{eqalfabeta}) and Proposition~\ref{a+b}~\ref{prop-Nijf-BHOasso}, then we have
	\[p_{\alpha\,\beta}\circ (\mu_{\alpha,\,\beta}+t\mu_{\alpha,\,\beta}^N)(a,b)=(\mu_{\alpha,\,\beta}+t\mu_{\alpha,\,\beta}^N)(p_{\alpha}(a),p_{\beta}(b)),\]
	\[q_{\alpha\,\beta}\circ (\mu_{\alpha,\,\beta}+t\mu_{\alpha,\,\beta}^N)(a,b)=(\mu_{\alpha,\,\beta}+t\mu_{\alpha,\,\beta}^N)(q_{\alpha}(a),q_{\beta}(b)).\]
	Next, for the BiHom-$\Omega$-associativity of $\mu_{\alpha,\,\beta}+t\mu_{\alpha,\,\beta}^N$, we have
	\[(\mu_{\alpha\,\beta,\,\gamma}+t\mu_{\alpha\,\beta,\,\gamma}^N)\big((\mu_{\alpha,\,\beta}
+t\mu_{\alpha,\,\beta}^N)(a,b),q_{\gamma}(c)\big)=(\mu_{\alpha,\,\beta\,\gamma}
+t\mu_{\alpha,\,\beta\,\gamma}^N)\big(p_{\alpha}(a),(\mu_{\beta,\,\gamma}+t\mu_{\beta,\,\gamma}^N)(b,c)\big),\]
	which is equivalent to
\begin{align}
		&\mu_{\alpha\,\beta,\,\gamma}\big(\mu_{\alpha,\,\beta}(a,b),\,q_{\gamma}(c)\big)= \mu_{\alpha,\,\beta\,\gamma}\big(p_{\alpha}(a),\,\mu_{\beta,\,\gamma}(b,c)\big),\label{asso-eq1}\\
		&\mu_{\alpha\,\beta,\,\gamma}\big(\mu_{\alpha,\,\beta}^N(a,b),\,q_{\gamma}(c)\big)+\mu_{\alpha\,\beta,\,\gamma}^N
\big(\mu_{\alpha,\,\beta}(a,b),q_{\gamma}(c)\big)\nonumber\\
		& \ \hspace{3.7cm}= \mu_{\alpha,\,\beta\,\gamma}\big(p_{\alpha}(a),\,\mu_{\beta,\,\gamma}^N(b,c)\big)+\mu_{\alpha,\,\beta\,\gamma}^N\big(p_{\alpha}(a),\mu_{\beta,\,\gamma}(b,c)\big),\label{asso-eq2}\\
		& \ \mu_{\alpha\,\beta,\,\gamma}^N\big(\mu_{\alpha,\,\beta}^N(a,b),q_{\gamma}(c)\big)= \mu_{\alpha,\,\beta\,\gamma}^N\big(p_{\alpha}(a),\mu_{\beta,\,\gamma}^N(b,c)\big).\label{asso-eq3}
	\end{align}
	From Eq.~(\ref{eqasso}) and Proposition~\ref{a+b}~\ref{prop-Nijf-BHOasso}, we know that Eq.~(\ref{asso-eq1}) and Eq.~(\ref{asso-eq3}) are true. So now we only need to prove Eq.~(\ref{asso-eq2}), we have
	\begin{align*}
		&\mu_{\alpha\,\beta,\,\gamma}\big(\mu_{\alpha,\,\beta}^N(a,b),\,q_{\gamma}(c)\big)+\mu_{\alpha\,\beta,\,\gamma}^N\big(\mu_{\alpha,\,\beta}(a,b),q_{\gamma}(c)\big)\\
		=&\mu_{\alpha\,\beta,\,\gamma}\Big(\mu_{\alpha,\,\beta}\big(N_{\alpha}(a),\,b\big)+\mu_{\alpha,\,\beta}\big(a,\,N_{\beta}(b)\big)-N_{\alpha\,\beta}\mu_{\alpha,\,\beta}(a,b),\,q_{\gamma}(c)\Big)+\mu_{\alpha\,\beta,\,\gamma}\big(N_{\alpha\,\beta}\mu_{\alpha,\,\beta}(a,b),\,q_{\gamma}(c)\big)\\
		&+\mu_{\alpha\,\beta,\,\gamma}\big(\mu_{\alpha,\,\beta}(a,b),\,N_{\gamma}q_{\gamma}(c)\big)-N_{\alpha\,\beta\,\gamma}\mu_{\alpha\,\beta,\,\gamma}\big(\mu_{\alpha,\,\beta}(a,b),\,q_{\gamma}(c)\big)\\
		=&\mu_{\alpha\,\beta,\,\gamma}\big(\mu_{\alpha,\,\beta}(N_{\alpha}(a),\,b),\,q_{\gamma}(c)\big)+\mu_{\alpha\,\beta,\,\gamma}\Big(\mu_{\alpha,\,\beta}\big(a,\,N_{\beta}(b)\big),\,q_{\gamma}(c)\Big)-\mu_{\alpha\,\beta,\,\gamma}\big(N_{\alpha\,\beta}\mu_{\alpha,\,\beta}(a,b),\,q_{\gamma}(c)\big)\\
		&+\mu_{\alpha\,\beta,\,\gamma}\big(N_{\alpha\,\beta}\mu_{\alpha,\,\beta}(a,b),\,q_{\gamma}(c)\big)+\mu_{\alpha\,\beta,\,\gamma}\big(\mu_{\alpha,\,\beta}(a,b),\,q_{\gamma}N_{\gamma}(c)\big)-N_{\alpha\,\beta\,\gamma}\mu_{\alpha\,\beta,\,\gamma}\big(\mu_{\alpha,\,\beta}(a,b),\,q_{\gamma}(c)\big)\\
		&\hspace{1cm}(\text{by Eq.~(\ref{Nijf-def-1})})\\
		=&\mu_{\alpha,\,\beta\,\gamma}\big(p_{\alpha}N_{\alpha}(a),\,\mu_{\beta,\,\gamma}(b,c)\big)+\mu_{\alpha,\,\beta\,\gamma}\Big(p_{\alpha}(a),\,\mu_{\beta,\,\gamma}\big(N_{\beta}(b),\,c\big)\Big)+\mu_{\alpha,\,\beta\,\gamma}\Big(p_{\alpha}(a),\,\mu_{\beta,\,\gamma}\big(b,\,N_{\gamma}(c)\big)\Big)\\
		&-N_{\alpha\,\beta\,\gamma}\mu_{\alpha,\,\beta\,\gamma}\big(p_{\alpha}(a),\,\mu_{\beta,\,\gamma}(b,c)\big).\hspace{1cm}(\text{by Eq.~(\ref{eqasso})})\\
		=&\mu_{\alpha,\,\beta\,\gamma}\Big(p_{\alpha}(a),\,\mu_{\beta,\,\gamma}\big(N_{\beta}(b),\,c\big)+\mu_{\beta,\,\gamma}\big(b,\,N_{\gamma}(c)\big)-N_{\beta\,\gamma}\mu_{\beta,\,\gamma}(b,c)\Big)+\mu_{\alpha,\,\beta\,\gamma}\big(N_{\alpha}p_{\alpha}(a),\,\mu_{\beta,\,\gamma}(b,c)\big)\\
		&+\mu_{\alpha,\,\beta\,\gamma}\big(p_{\alpha}(a),\,N_{\beta\,\gamma}\mu_{\beta,\,\gamma}(b,c)\big)-N_{\alpha\,\beta\,\gamma}\mu_{\alpha,\,\beta\,\gamma}\big(p_{\alpha}(a),\,\mu_{\beta,\,\gamma}(b,c)\big)\hspace{1cm}(\text{by Eq.~(\ref{Nijf-def-1})})\\
		=&\mu_{\alpha,\,\beta\,\gamma}\big(p_{\alpha}(a),\,\mu_{\beta,\,\gamma}^N(b,c)\big)+\mu_{\alpha,\,\beta\,\gamma}^N
\big(p_{\alpha}(a),\mu_{\beta,\,\gamma}(b,c)\big).
	\end{align*}
Thus, $(A[[t]]/(t^{2}),\mu_{\alpha,\,\beta}+t\mu_{\alpha,\,\beta}^{N},p_{\omega},q_{\omega})_{\alpha,\,\beta,\,\omega\in \Omega}$ is a BiHom-$\Omega$-associative algebra.	

\ref{it:2-cocycle}. By Definition~\ref{def-BHO} and Remark~\ref{2-cocycle}, we only need to check the following equation:
		\[(\delta_{\mathrm{Alg}}^2 \psi^N)_{\alpha,\,\beta,\,\gamma}(a,b,c)=\mu_{\alpha\,\beta,\,\gamma}^N
\big(\mu_{\alpha,\,\beta}^N(a,b),q_{\gamma}(c)\big)-\mu_{\alpha,\,\beta\,\gamma}^N
\big(p_{\alpha}(a),\mu_{\beta,\,\gamma}^N(b,c)\big),\]
		for all $a,b,c\in A,\,\alpha,\,\beta,\,\gamma\in \Omega$. Then we have
		\begin{align*}
			\mu&_{\alpha\,\beta,\,\gamma}^N\big(\mu_{\alpha,\,\beta}^N(a,b),q_{\gamma}(c)\big)-\mu_{\alpha,\,\beta\,\gamma}^N\big(p_{\alpha}(a),\mu_{\beta,\,\gamma}^N(b,c)\big)\\
			=&\mu_{\alpha\,\beta,\,\gamma}\Big(N_{\alpha\,\beta}\mu_{\alpha,\,\beta}^N(a,b),\,q_{\gamma}(c)\Big)+\mu_{\alpha\,\beta,\,\gamma}\Big(\mu_{\alpha,\,\beta}^N(a,b),\,N_{\gamma}q_{\gamma}(c)\Big)-N_{\alpha\,\beta\,\gamma}\mu_{\alpha\,\beta,\,\gamma}\Big(\mu_{\alpha,\,\beta}^N(a,b),\,q_{\gamma}(c)\Big)\\
			&-\mu_{\alpha,\,\beta\,\gamma}\Big(N_{\alpha}p_{\alpha}(a),\,\mu_{\beta,\,\gamma}^N(b,c)\Big)-\mu_{\alpha,\,\beta\,\gamma}\Big(p_{\alpha}(a),\,N_{\beta\,\gamma}\mu_{\beta,\,\gamma}^N(b,c)\Big)+N_{\alpha\,\beta\,\gamma}\mu_{\alpha,\,\beta\,\gamma}\Big(p_{\alpha}(a),\,\mu_{\beta,\,\gamma}^N(b,c)\Big)\\
			=&\mu_{\alpha\,\beta,\,\gamma}\Big(N_{\alpha\,\beta}\mu_{\alpha,\,\beta}^N(a,b),\,q_{\gamma}(c)\Big)+\mu_{\alpha\,\beta,\,\gamma}\Big(\mu_{\alpha,\,\beta}\big(N_{\alpha}(a),\,b\big)+\mu_{\alpha,\,\beta}\big(a,\,N_{\beta}(b)\big)-N_{\alpha\,\beta}\mu_{\alpha,\,\beta}(a,b),\,q_{\gamma}N_{\gamma}(c)\Big)\\
			&+N_{\alpha\,\beta\,\gamma}\Big(\mu_{\alpha,\,\beta\,\gamma}\big(p_{\alpha}(a),\,\mu_{\beta,\,\gamma}^N(b,c)\big)-\mu_{\alpha\,\beta,\,\gamma}\big(\mu_{\alpha,\,\beta}^N(a,b),\,q_{\gamma}(c)\big)\Big)-\mu_{\alpha,\,\beta\,\gamma}\big(p_{\alpha}(a),\,N_{\beta\,\gamma}\mu_{\beta,\,\gamma}^N(b,c)\big)\\
			&-\mu_{\alpha,\,\beta\,\gamma}\Big(p_{\alpha}N_{\alpha}(a),\,\mu_{\beta,\,\gamma}\big(N_{\beta}(b),c\big)+\mu_{\beta,\,\gamma}\big(b,\,N_{\gamma}(c)\big)-N_{\beta\,\gamma}\mu_{\beta\,\gamma}(b,c)\Big)\hspace{1cm}(\text{by Eq.~(\ref{Nijf-def-1})})\\
			=&\mu_{\alpha\,\beta,\,\gamma}\Big(N_{\alpha\,\beta}\mu_{\alpha,\,\beta}^N(a,b),\,q_{\gamma}(c)\Big)+\mu_{\alpha,\,\beta\,\gamma}\Big(p_{\alpha}N_{\alpha}(a),\,\mu_{\beta,\,\gamma}\big(b,\,N_{\gamma}(c)\big)\Big)+\mu_{\alpha,\,\beta\,\gamma}\Big(p_{\alpha}(a),\,\mu_{\beta,\,\gamma}\big(N_{\beta}(b),\,N_{\gamma}(c)\big)\Big)\\
			&-\mu_{\alpha\,\beta,\,\gamma}\Big(N_{\alpha\,\beta}\mu_{\alpha,\,\beta}(a,b),\,q_{\gamma}N_{\gamma}(c)\Big)+N_{\alpha\,\beta\,\gamma}\Big(\mu_{\alpha,\,\beta\,\gamma}\big(p_{\alpha}(a),\,\mu_{\beta,\,\gamma}^N(b,c)\big)-\mu_{\alpha\,\beta,\,\gamma}\big(\mu_{\alpha,\,\beta}^N(a,b),\,q_{\gamma}(c)\big)\Big)\\
			&-\mu_{\alpha,\,\beta\,\gamma}\Big(p_{\alpha}N_{\alpha}(a),\,\mu_{\beta,\,\gamma}\big(N_{\beta}(b),\,c\big)\Big)-\mu_{\alpha,\,\beta\,\gamma}\Big(p_{\alpha}N_{\alpha}(a),\,\mu_{\beta,\,\gamma}\big(b,\,N_{\gamma}(c)\big)\Big)+\mu_{\alpha,\,\beta\,\gamma}\Big(p_{\alpha}N_{\alpha}(a),\,N_{\beta\,\gamma}\mu_{\beta,\,\gamma}(b,c)\Big)\\
			&-\mu_{\alpha,\,\beta\,\gamma}\Big(p_{\alpha}(a),\,N_{\beta\,\gamma}\mu_{\beta,\,\gamma}^N(b,c)\Big)\hspace{1cm}(\text{by Eq.~(\ref{eqasso})})\\
			=&\mu_{\alpha\,\beta,\,\gamma}\Big(N_{\alpha\,\beta}\mu_{\alpha,\,\beta}^N(a,b),\,q_{\gamma}(c)\Big)+\mu_{\alpha,\,\beta\,\gamma}\Big(p_{\alpha}(a),\,\mu_{\beta,\,\gamma}\big(N_{\beta}(b),\,N_{\gamma}(c)\big)\Big)-\mu_{\alpha\,\beta,\,\gamma}\Big(N_{\alpha\,\beta}\mu_{\alpha,\,\beta}(a,b),\,q_{\gamma}N_{\gamma}(c)\Big)\\
			&+N_{\alpha\,\beta\,\gamma}\Big(\mu_{\alpha\,\beta,\,\gamma}^N\big(\mu_{\alpha,\,\beta}(a,b),q_{\gamma}(c)\big)-\mu_{\alpha,\,\beta\,\gamma}^N\big(p_{\alpha}(a),\mu_{\beta,\,\gamma}(b,c)\big)\Big)-\mu_{\alpha,\,\beta\,\gamma}\Big(p_{\alpha}N_{\alpha}(a),\,\mu_{\beta,\,\gamma}\big(N_{\beta}(b),\,c\big)\Big)\\
			&+\mu_{\alpha,\,\beta\,\gamma}\Big(p_{\alpha}N_{\alpha}(a),\,N_{\beta\,\gamma}\mu_{\beta,\,\gamma}(b,c)\Big)-\mu_{\alpha,\,\beta\,\gamma}\Big(p_{\alpha}(a),\,N_{\beta\,\gamma}\mu_{\beta,\,\gamma}^N(b,c)\Big)\hspace{1cm}(\text{by Eq.~(\ref{asso-eq2})})\\
			=&\mu_{\alpha,\,\beta\,\gamma}\Big(p_{\alpha}(a),\,\mu_{\beta,\,\gamma}\big(N_{\beta}(b),\,N_{\gamma}(c)\big)-N_{\beta\,\gamma}\mu_{\beta,\,\gamma}^N(b,c)\Big)-\mu_{\alpha\,\beta,\,\gamma}\Big(N_{\alpha\,\beta}\mu_{\alpha,\,\beta}(a,b),\,N_{\gamma}q_{\gamma}(c)\Big)\\
			&+\mu_{\alpha,\,\beta\,\gamma}\Big(N_{\alpha}p_{\alpha}(a),\,N_{\beta\,\gamma}\mu_{\beta,\,\gamma}(b,c)\Big)+N_{\alpha\,\beta\,\gamma}\mu_{\alpha\,\beta,\,\gamma}^N\big(\mu_{\alpha,\,\beta}(a,b),q_{\gamma}(c)\big)-N_{\alpha\,\beta\,\gamma}\mu_{\alpha\,\beta,\,\gamma}^N(p_{\alpha}(a),\mu_{\beta,\,\gamma}(b,c))\\
			&-\mu_{\alpha\,\beta,\,\gamma}\Big(\mu_{\alpha,\,\beta}\big(N_{\alpha}(a),\,N_{\beta}(b)\big)-N_{\alpha\,\beta}\mu_{\alpha,\,\beta}^N(a,b),\,q_{\gamma}(c)\Big)\hspace{1cm}(\text{by Eq.~(\ref{eqasso}) and Eq.~(\ref{Nijf-def-1})})\\
			=&\mu_{\alpha,\,\beta\,\gamma}\Big(p_{\alpha}(a),\,\psi_{\beta,\,\gamma}^N(b,c)\Big)-\psi_{\alpha\,\beta,\,\gamma}^N\big(\mu_{\alpha,\,\beta}(a,b),q_{\gamma}(c)\big)+\psi_{\alpha,\,\beta\,\gamma}^N\big(p_{\alpha}(a),\mu_{\beta,\,\gamma}(b,c)\big)\\
			&-\mu_{\alpha\,\beta,\,\gamma}\Big(\psi_{\alpha,\,\beta}^N(a,b),\,q_{\gamma}(c)\Big)\hspace{1cm}(\text{by Eq.~(\ref{def-psi})})\\
			=&(\delta_{\mathrm{Alg}}^2\psi^N)_{\alpha,\,\beta,\,\gamma}(a,b,c).\hspace{1cm}(\text{by Eq.~(\ref{BiHomOmHochschdiff})})
		\end{align*}
		Thus, by Proposition~\ref{a+b}~\ref{prop-Nijf-BHOasso}, we get
		\[(\delta_{\mathrm{Alg}}^2\psi^N)_{\alpha,\,\beta,\,\gamma}(a,b,c)=\mu_{\alpha\,\beta,\,\gamma}^N\big(\mu_{\alpha,\,\beta}^N(a,b),q_{\gamma}(c)\big)-\mu_{\alpha,\,\beta\,\gamma}^N\big(p_{\alpha}(a),\mu_{\beta,\,\gamma}^N(b,c)\big)\\
		=0.\]
		This completes the proof.
\end{proof}

\subsection{Deformations of Rota-Baxter family BiHom-$\Omega$-associative algebras}

In this subsection, we will study the deformations of Rota-Baxter family BiHom-$\Omega$-associative algebras and interpret them via cohomology groups of Rota-Baxter family BiHom-$\Omega$-associative algebras defined in Section~\ref{sec3}.

Let $ (A,\mu_{\alpha,\,\beta},R_{\omega},p_{\omega},q_{\omega})_{\alpha,\,\beta,\,\omega\in \Omega} $ be a Rota-Baxter family BiHom-$\Omega$-associative algebra of weight $\lambda$. We define 
\[\mu_{\alpha,\,\beta}^{t}=\sum_{i=0}^{\infty} \mu_{\alpha,\,\beta}^{i}t^{i} : A[[t]] \times A[[t]]\rightarrow A[[t]], \quad (\mu_{\alpha,\,\beta}^{i})_{\alpha,\,\beta\in\Omega}\in \mathrm{C}_{\Omega}^{2}(A), \]\[ R_{\omega}^{t}=\sum_{i=0}^{\infty}R_{\omega}^{i}t^{i} : A[[t]]\rightarrow A[[t]],\quad (R_{\omega}^{i})_{\omega\in\Omega}\in \mathrm{C}_{\mathrm{RBF}_{\lambda}}^{1}(A),\]
for all $\alpha,\,\beta,\,\omega\in \Omega$.
\begin{defn}
	A \textbf{1-parameter formal deformation of Rota-Baxter family BiHom-$\Omega$-associative algebra $ (A,\mu_{\alpha,\,\beta}, R_{\omega},p_{\omega},q_{\omega})_{\alpha,\,\beta,\,\omega\in \Omega} $} is a pair $ (\mu_{\alpha,\,\beta}^{t}, R_{\omega}^{t})_{\alpha,\,\beta,\,\omega\in \Omega} $ such that $(A[[t]],\mu_{\alpha,\,\beta}^t,R_{\omega}^t,p_{\omega},q_{\omega})_{\alpha,\,\beta,\,\omega\in\Omega}$ is a Rota-Baxter family BiHom-$\Omega$-associative algebra structure over $ \mathbf{k}[[t]] $ and we have a convention that $ (\mu_{\alpha,\,\beta}^{0},R_{\omega}^{0})_{\alpha,\,\beta,\,\omega\in\Omega}=(\mu_{\alpha,\,\beta}, R_{\omega})_{\alpha,\,\beta,\,\omega\in\Omega}$.

	Power series $ (\mu_{\alpha,\,\beta}^{t})_{\alpha,\,\beta\in \Omega} $ and $ (R_{\omega}^{t})_{\omega\in \Omega} $ determine a 1-parameter formal deformation of Rota-Baxter family BiHom-$\Omega$-associative algebra $ (A,\mu_{\alpha,\,\beta}, R_{\omega},p_{\omega},q_{\omega})_{\alpha,\,\beta,\,\omega\in \Omega} $ if and only if
	\[\mu_{\alpha,\,\beta\gamma}^{t}(p_{\alpha}(a), \mu_{\beta ,\gamma}^{t}(b, c))=\mu_{\alpha\beta ,\gamma}^{t}(\mu_{\alpha,\,\beta}^{t}(a, b), q_{\gamma}(c)),\]
	\[\mu_{\alpha,\,\beta}^{t}(R_{\alpha}(a), R_{\beta}(b))=R_{\alpha\beta}^{t}(\mu_{\alpha,\,\beta}^{t}(a, R_{\beta}^{t}(b))+\mu_{\alpha,\,\beta}^{t}(R_{\alpha}^{t}(a), b)+\lambda \mu_{\alpha,\,\beta}^{t}(a, b)),\]
	for all $a,b,c\in A,\,\alpha,\,\beta,\,\gamma\in \Omega$.
\end{defn}

By expanding these equations and comparing the coefficient of $ t^{n} $, we obtain that $ ( \mu_{\alpha,\,\beta}^{i})_{\alpha,\,\beta\in \Omega}  $ and $ ( R_{\omega}^{i})_{\omega\in \Omega}  $ have to satisfy:

\begin{align}\label{*-1}
	\sum_{i=0}^{n}\mu_{\alpha\beta ,\gamma}^{i}\circ(\mu_{\alpha,\,\beta}^{n-i}\otimes q_{\gamma})=\sum_{i=0}^{n}\mu_{\alpha,\,\beta\gamma}^{i}\circ (p_{\alpha}\otimes \mu_{\beta ,\gamma}^{n-i}),
\end{align}

\begin{equation}\label{*-2}
	\begin{split}
		\sum_{i+j+k=n ; i,j,k\geq 0}\mu_{\alpha,\,\beta}^{i}\circ (R_{\alpha}^{j}\otimes R_{\beta}^{k})&=\sum_{i+j+k=n ; i,j,k\geq 0}R_{\alpha\beta}^{i}\circ \mu_{\alpha,\,\beta}^{j}\circ(id \otimes R_{\beta}^{k})+\sum_{i+j+k=n ; i,j,k\geq 0}R_{\alpha\beta}^{i}\circ\mu_{\alpha,\,\beta}^{j}\circ (R_{\alpha}^{k}\otimes id )\\
		&+\lambda \sum_{i+j=n;i,j\geq 0}R_{\alpha\beta}^{i}\circ \mu_{\alpha,\,\beta}^{j},\hspace{1cm} \text{for all $ n\geq 0,\;\alpha,\,\beta,\,\gamma\in \Omega $}.
	\end{split}
\end{equation}

Obviously, when $ n=0 $, Eqs.~(\ref{*-1})-(\ref{*-2}) reduce to Eq.~(\ref{eqasso}) and Eq.~(\ref{RBfbihom}), respectively.

\begin{prop}\label{RBFBHO-mu1-2cocycle}
	If $ (\mu_{\alpha,\,\beta}^{t}, R_{\omega}^{t})_{\alpha,\,\beta,\,\omega\in \Omega} $ is a 1-parameter formal deformation of Rota-Baxter family BiHom-$\Omega$-associative algebra $A$ of weight $\lambda$. Then $ ( \mu_{\alpha,\,\beta}^{1}, R_{\omega}^{1})_{\alpha,\,\beta,\,\omega\in \Omega} $ is a 2-cocycle in the cochain complex $ \mathrm{C}_{\mathrm{RBFA}_{\lambda}}^{\bullet}(A). $
\end{prop}

\begin{proof}
	For any $\alpha,\,\beta,\,\gamma,\,\omega,\,\eta\in \Omega$ and $ n=1 $, then Eqs.~(\ref{*-1})-(\ref{*-2}) become
	\[\mu_{\alpha\,\beta ,\gamma}^{1}\circ(\mu_{\alpha,\,\beta}\otimes q_{\gamma})+\mu_{\alpha\,\beta,\,\gamma}\circ (\mu_{\alpha,\,\beta}^{1}\otimes q_{\gamma})=\mu_{\alpha,\,\beta\,\gamma}^{1}\circ(p_{\alpha}\otimes \mu_{\beta,\,\gamma})+\mu_{\alpha,\,\beta\,\gamma}\circ (p_{\alpha}\otimes \mu_{\beta,\,\gamma}^{1}),\]
	and
	\begin{align*}
		&\mu_{\omega,\,\eta}^{1}(R_{\omega}\otimes R_{\eta})-\big( R_{\omega\,\eta}\circ\mu_{\omega,\,\eta}^{1}\circ(id \otimes R_{\eta})+R_{\omega\,\eta}\circ\mu_{\omega,\,\eta}^{1}\circ(R_{\omega}\otimes id )+\lambda R_{\omega\,\eta}\circ \mu_{\omega,\,\eta}^{1}\big) \\
		=&-\big( \mu_{\omega,\,\eta}\circ (R_{\omega}\otimes R_{\eta}^{1})-R_{\omega\,\eta}\circ \mu_{\omega,\,\eta}\circ(id\otimes R_{\eta}^{1})\big)-\big( \mu_{\omega,\,\eta}\circ (R_{\omega}^{1}\otimes R_{\eta})-R_{\omega\,\eta}\circ \mu_{\omega,\,\eta}\circ (R_{\omega}^{1}\otimes id)\big) \\
		&+\big( R_{\omega\,\eta}^{1}\circ \mu_{\omega,\,\eta}\circ(id \otimes R_{\eta})+R_{\omega\,\eta}^{1}\circ \mu_{\omega,\,\eta}\circ (R_{\omega}\otimes id )+\lambda R_{\omega\,\eta}^{1}\circ \mu_{\omega,\,\eta}\big).
	\end{align*}
	
	Note that the first equation is exactly $ \delta_{\mathrm{Alg}}^{2}(\mu^{1})_{\alpha,\,\beta,\,\gamma}=0$. For the second equation, by Eq.~(\ref{defof-partial}) and Eq.~(\ref{defof-Phi}), we have $\Phi^{2}(\mu^{1})_{\omega,\,\eta}=-\partial^{1}(R^{1})_{\omega,\,\eta}.$ Thus, by Definition~\ref{diffof-RBFbihomOassoalg} and Remark~\ref{RBFA-2-cocycle}, we obtain that $ ( \mu_{\alpha,\,\beta}^{1}, R_{\omega}^{1})_{\alpha,\,\beta,\,\omega\in \Omega} $ is a 2-cocycle in $ \mathrm{C}_{\mathrm{RBFA}_{\lambda}}^{\bullet}(A). $
\end{proof}

\begin{coro}
	In particular, if $ (\mu_{\alpha,\,\beta}^{t}, R_{\omega}^{t})_{\alpha,\,\beta,\,\omega\in \Omega} $ is a 1-parameter formal deformation of Rota-Baxter family BiHom-$\Omega$-associative algebra $A$ of weight $\lambda$, then we have the following results.
	\begin{enumerate}
		\item\label{it:mu-2-cocycle} The family of bilinear maps $(\mu_{\alpha,\,\beta}^1)_{\alpha,\,\beta\in\Omega}$ is a 2-cocycle in cochain complex $\mathrm{C}_{\Omega}^2(A)$.
		\item\label{it:R-1-cocycle} The family of linear maps $(R_{\omega}^1)_{\omega\in\Omega}$ is a 1-cocycle in cochain complex $\mathrm{C}_{\mathrm{RBF}_{\lambda}}^1(A)$.
	\end{enumerate}
\end{coro}
\begin{proof}
	\ref{it:mu-2-cocycle}. By Proposition~\ref{RBFBHO-mu1-2cocycle}, we get $\delta_{\mathrm{Alg}}^2(\mu^1)_{\alpha,\,\beta,\,\gamma}=0$, for all $\alpha,\,\beta,\,\gamma\in\Omega$. Thus, $(\mu_{\alpha,\,\beta}^1)_{\alpha,\,\beta\in\Omega}$ is a 2-cocycle in cochain complex $\mathrm{C}_{\Omega}^2(A)$.
	
	\ref{it:R-1-cocycle}. By Eq.~(\ref{RBfbihom}) and Eq.~(\ref{*-2}), when $(\mu_{\alpha,\,\beta}^t)_{\alpha,\,\beta\in\Omega}=(\mu_{\alpha,\,\beta})_{\alpha,\,\beta\in\Omega}$ and $n=1$, we have
	\begin{align*}
		&\quad\mu_{\alpha,\,\beta}(R_{\alpha}^1,R_{\beta})+\mu_{\alpha,\,\beta}(R_{\alpha},R_{\beta}^{1})\\
		&=R_{\alpha\,\beta}^1\big(\mu_{\alpha,\,\beta}(id,R_{\beta})+\mu_{\alpha,\,\beta}(R_{\alpha},id)\big)+R_{\alpha\,\beta}\big(\mu_{\alpha,\,\beta}(id,R_{\beta}^{1})+\mu_{\alpha,\,\beta}(R_{\alpha}^1,id)\big)+\lambda R_{\alpha\,\beta}^1\mu_{\alpha,\,\beta},
	\end{align*}
then by Eq.~(\ref{defof-partial}), we get $\partial^1(R^1)_{\alpha,\,\beta}=0$, for all $\alpha,\,\beta\in\Omega$. Thus, $(R_{\omega}^1)_{\omega\in\Omega}$ is a 1-cocycle in cochain complex $\mathrm{C}_{\mathrm{RBF}_{\lambda}}^1(A)$.
\end{proof}

\begin{defn}
	Let $ (\mu_{\alpha,\,\beta}^{t}, R_{\omega}^{t})_{\alpha,\,\beta,\,\omega\in \Omega}$ be a 1-parameter formal deformation of Rota-Baxter family BiHom-$\Omega$-associative algebra $ (A,\mu_{\alpha,\,\beta}, R_{\omega},p_{\omega},q_{\omega})_{\alpha,\,\beta,\,\omega\in \Omega} $. Then we call 2-cocycle $ ( \mu_{\alpha,\,\beta}^{1}, R_{\omega}^{1})_{\alpha,\,\beta,\,\omega\in \Omega} $ the \textbf{infinitesimal} of the 1-parameter formal deformation $ ( \mu_{\alpha,\,\beta}^{t}, R_{\omega}^{t})_{\alpha,\,\beta,\,\omega\in \Omega} $.
\end{defn}

\begin{defn}
	Two 1-parameter formal deformations $ (\mu_{\alpha,\,\beta}^{t}, R_{\omega}^{t})_{\alpha,\,\beta,\,\omega\in \Omega} $ and $ ( \bar{\mu}_{\alpha,\,\beta}^{t}, \bar{R}_{\omega}^{t})_{\alpha,\,\beta,\,\omega\in \Omega} $ of Rota-Baxter family BiHom-$\Omega$-associative algebra $A$ are said to be \textbf{equivalent} if there exists a power series formal homomorphism
	\[\psi_{\omega}^{t}=\sum_{i=0}\psi_{\omega}^{i}t^{i}: A[[t]]\rightarrow A[[t]],\quad \text{for all }\omega\in\Omega,\]
	where $ (\psi_{\omega}^{i})_{\omega\in \Omega}: A\rightarrow A $ is a family of linear maps with $ (\psi_{\omega}^{0})_{\omega\in\Omega}=id_{A} $, and for all $\alpha,\,\beta,\,\omega\in \Omega$,
	\[\psi_{\omega}^{t}\circ p_{\omega}=p_{\omega}\circ \psi_{\omega}^{t}, \quad \psi_{\omega}^{t}\circ q_{\omega}=q_{\omega}\circ \psi_{\omega}^{t},\]
	\begin{align}\label{**-1}
		\psi_{\alpha\beta}^{t}\circ \bar{\mu}_{\alpha,\,\beta}^{t}=\mu_{\alpha,\,\beta}^{t}\circ (\psi_{\alpha}^{t}\otimes \psi_{\beta}^{t}),
	\end{align}
	\begin{align}\label{**-2}
		\psi_{\omega}^{t}\circ\bar{R}_{\omega}^{t}=R_{\omega}^{t}\circ \psi_{\omega}^{t}.
	\end{align}
\end{defn}

\begin{theorem}\label{theorem-2-cocycle}
	The infinitesimals of two equivalent one-parameter formal deformations of Rota-Baxter family BiHom-$\Omega$-associative algebra $ (A,\mu_{\alpha,\,\beta}, R_{\omega},p_{\omega},q_{\omega})_{\alpha,\,\beta,\,\omega\in \Omega} $ are in the same cohomology class in $ \mathrm{H}_{\mathrm{RBFA}_{\lambda}}^{\bullet}(A). $
\end{theorem}
\begin{proof}
	Let $( \psi_{\omega}^{t})_{\omega\in \Omega} : \big(A[[t]],\bar{\mu}_{\alpha,\,\beta}^{t}, \bar{R}_{\omega}^{t},p_{\omega},q_{\omega}\big)_{\alpha,\,\beta,\,\omega\in \Omega}\rightarrow (A[[t]],\mu_{\alpha,\,\beta}^{t}, R_{\omega}^{t},p_{\omega},q_{\omega})_{\alpha,\,\beta,\,\omega\in \Omega} $ be a formal isohomomorphism. Expanding the identities and collecting coefficients of $t$, by Eqs.~(\ref{**-1})-(\ref{**-2}), for any $\alpha,\,\beta,\,\omega\in \Omega$, on the one hand,
	\[\sum_{i+j=n;\,i,j\geq 0}\psi_{\alpha\,\beta}^{i}\circ \bar{\mu}_{\alpha,\,\beta}^{j}=\sum_{i+j+k=n ;\,i,j,k\geq 0}\mu_{\alpha,\,\beta}^{i}(\psi_{\alpha}^{j}\otimes \psi_{\beta}^{k}),\]
	
\noindent	when $ n=1 $, by $ (\psi_{\omega}^{0})_{\omega\in\Omega}=id_{A} $ we have
	\[\bar{\mu}_{\alpha,\,\beta}^{1}+\psi_{\alpha\,\beta}^{1}\circ\mu_{\alpha,\,\beta}=\mu_{\alpha,\,\beta}^{1}+\mu_{\alpha,\,\beta}(\psi_{\alpha}^{1}\otimes id)+\mu_{\alpha,\,\beta}(id \otimes \psi_{\beta}^{1}),\]
	so by Eq.~(\ref{BiHomOmHochschdiff}), we have
	\[\bar{\mu}_{\alpha,\,\beta}^{1}-\mu_{\alpha,\,\beta}^{1}=\delta_{\mathrm{Alg}}^{1}(\psi^{1})_{\alpha,\,\beta}.\]
	On the other hand, we have
	\[\sum_{i+j=n;\,i,j\geq 0}\psi_{\omega}^{i}\circ \bar{R}_{\omega}^{j}=\sum_{i+j=n;\,i,j\geq 0}R_{\omega}^{i}\circ \psi_{\omega}^{j},\]
	when $ n=1 $, by $ \psi_{\omega}^{0}=id_{A} $ we have
	\[\bar{R}_{\omega}^{1}+\psi_{\omega}^{1}\circ R_{\omega}=R_{\omega}\circ \psi_{\omega}^{1}+R_{\omega}^{1},\]
	by Eq.~(\ref{defof-Phi-1}), we have
	\[\bar{R}_{\omega}^{1}-R_{\omega}^{1}=-\Phi^{1}(\psi^{1})_{\omega}.\]
	Thus, we have
	\begin{align*}
		(\bar{\mu}_{\alpha,\,\beta}^{1},\bar{R}_{\omega}^{1})_{\alpha,\,\beta,\,\omega\in\Omega}-(\mu_{\alpha,\,\beta}^{1},R_{\omega}^{1})_{\alpha,\,\beta,\,\omega\in\Omega}=&(\bar{\mu}_{\alpha,\,\beta}^1-\mu_{\alpha,\,\beta}^1,\bar{R}_{\omega}^1-R_{\omega}^1)_{\alpha,\,\beta,\,\omega\in\Omega}\\
		=&(\delta_{\mathrm{Alg}}^{1}(\psi^{1})_{\alpha,\,\beta}, -\Phi^{1}(\psi^{1})_{\omega})_{\alpha,\,\beta,\,\omega\in\Omega}\\
		=&\big(d^{1}(\psi^{1})_{\alpha,\,\beta,\,\omega}\big)_{\alpha,\,\beta,\,\omega\in\Omega} \in \mathit{B}_{\mathrm{RBFA}_{\lambda}}^{\bullet}(A)\subseteq \mathrm{C}_{\mathrm{RBFA}_{\lambda}}^{\bullet}(A).
	\end{align*}
This completes the proof.
\end{proof}

\begin{coro}
	In particular, when $R_{\omega}^t=R_{\omega}$ for all $\omega\in\Omega$, the corresponding cohomology controls formal deformations of BiHom-$\Omega$-associative product $(\mu_{\alpha,\,\beta}^t)_{\alpha,\,\beta\in\Omega}$.
\end{coro}
\begin{proof}
	By Theorem~\ref{theorem-2-cocycle}, we get \[\bar{\mu}_{\alpha,\,\beta}^{1}-\mu_{\alpha,\,\beta}^{1}=\delta_{\mathrm{Alg}}^{1}(\psi^{1})_{\alpha,\,\beta},\quad \text{for all } \alpha,\,\beta\in \Omega.\]
	Therefore, the infinitesimals of two equivalent 1-parameter formal deformations of $A$ give rise to a same cohomology class in $\mathrm{H}_{\Omega}^2(A)$. This completes the proof.
\end{proof}

Next, we introduce the rigidity of Rota-Baxter family BiHom-$\Omega$-associative algebras.
\begin{defn}
	A Rota-Baxter family BiHom-$\Omega$-associative algebra $(A,\mu_{\alpha,\,\beta},R_{\omega},p_{\omega},q_{\omega})_{\alpha,\,\beta,\,\omega\in \Omega}$ is said to be \textbf{rigid} if any 1-parameter formal deformation $(\mu_{\alpha,\,\beta}^t,R_{\omega}^t)_{\alpha,\,\beta,\,\omega\in \Omega}$ of $A$ is equivalent to the undeformed one $(\bar{\mu}_{\alpha,\,\beta}^t=\mu_{\alpha,\,\beta},\bar{R}_{\omega}^t=R_{\omega})_{\alpha,\,\beta,\,\omega\in \Omega}$.
\end{defn}

\begin{theorem}
	Let $(A,\mu_{\alpha,\,\beta},R_{\omega},p_{\omega},q_{\omega})_{\alpha,\,\beta,\,\omega\in \Omega}$ be a Rota-Baxter family BiHom-$\Omega$-associative algebra of weight $\lambda$. If $\mathrm{H}^2_{\mathrm{RBFA}_{\lambda}}(A)=0$, then $(A,\mu_{\alpha,\,\beta}, R_{\omega},p_{\omega},q_{\omega})_{\alpha,\,\beta,\,\omega\in \Omega}$ is rigid.
\end{theorem}
\begin{proof}
	Let $(\mu_{\alpha,\,\beta}^t,R_{\omega}^t)_{\alpha,\,\beta,\,\omega\in \Omega}$ be a 1-parameter formal deformation of Rota-Baxter family BiHom-$\Omega$-associative algebra $(A,\mu_{\alpha,\,\beta},R_{\omega},p_{\omega},q_{\omega})_{\alpha,\,\beta,\,\omega\in \Omega}$. By Proposition~\ref{RBFBHO-mu1-2cocycle}, we know that $(\mu_{\alpha,\,\beta}^1, R_{\omega}^1)_{\alpha,\,\beta,\,\omega\in \Omega}$ is a 2-cocycle, so we get $(\mu_{\alpha,\,\beta}^1,R_{\omega}^1)_{\alpha,\,\beta,\,\omega\in \Omega}\in Ker(d^2)$. Then by $\mathrm{H}^2_{\mathrm{RBFA}_{\lambda}}(A)=0$, that is $Ker(d^2)=Im(d^1)$. So, we have $(\mu_{\alpha,\,\beta}^1,R_{\omega}^1)_{\alpha,\,\beta,\,\omega\in \Omega}\in Im(d^1)$, i.e. there exists a 1-cochain $(\phi_{\alpha},x)_{\alpha\in\Omega}\in \mathrm{C}_{\mathrm{RBFA}_{\lambda}}^1(A)$ such that
	\[(\mu_{\alpha,\,\beta}^1,R_{\omega}^1)=d^1(\phi,x)_{\alpha,\,\beta,\,\omega}=\big(\delta_{\mathrm{Alg}}^1(\phi)_{\alpha,\,\beta},-\partial^0(x)_{\omega}-\Phi^1(\phi)_{\omega}\big),\quad \text{for all }\alpha,\,\beta,\,\omega\in\Omega.\]
	Let $\psi_{\alpha}^1=\phi_{\alpha}+\delta_{\mathrm{Alg}}^0(x)$, for all $\alpha\in\Omega$. Owing to $\delta_{\mathrm{Alg}}^1\circ \delta_{\mathrm{Alg}}^0=0$ and $\Phi^1\circ \delta_{\mathrm{Alg}}^0=\Phi^0\circ \partial^0=id\circ \partial^0=\partial^0$, we have $\mu_{\alpha,\,\beta}^1=\delta_{\mathrm{Alg}}^1(\psi_{\alpha}^1)=(\delta_{\mathrm{Alg}}^1(\psi^1))_{\alpha,\,\beta}$ and $R_{\omega}^1=-\Phi^1(\psi_{\omega}^1).$ We set $\psi_{\alpha}^t=id_{A}-t\psi_{\alpha}^1$ and define
	\[\bar{\mu}_{\alpha,\,\beta}^t=(\psi_{\alpha\,\beta}^t)^{-1}\circ\mu_{\alpha,\,\beta}^t\circ(\psi_{\alpha}^t\otimes \psi_{\beta}^t),\]
	\[\bar{R}_{\omega}^t=(\psi_{\omega}^t)^{-1}\circ R_{\omega}^t\circ \psi_{\omega}^t.\]
	According to $(\psi_{\alpha}^{t})_{\alpha\in\Omega}$ is commutative with $(p_{\omega})_{\omega\in\Omega},\,(q_{\omega})_{\omega\in\Omega}$, we get that $(\mu_{\alpha,\,\beta}^t,R_{\omega}^t)_{\alpha,\,\beta,\,\omega\in \Omega}$ is equivalent to the deformation $(\bar{\mu}_{\alpha,\,\beta}^t,\bar{R}_{\omega}^t)_{\alpha,\,\beta,\,\omega\in \Omega}$. Furthermore,
	\begin{align*}
		\bar{\mu}_{\alpha,\,\beta}^t(a,b)
		=&(\psi_{\alpha\,\beta}^t)^{-1}\circ\mu_{\alpha,\,\beta}^t\circ(\psi_{\alpha}^t\otimes\psi_{\beta}^t)(a,b)\hspace{1cm}(\text{mod }t^2)\\
		=&(id_{A}+t\psi_{\alpha\,\beta}^1)\circ(\mu_{\alpha,\,\beta}+t\mu_{\alpha,\,\beta}^1)\circ\big((id_{A}-t\psi_{\alpha}^1)\otimes (id_{A}-t\psi_{\beta}^1)\big)(a,b)\hspace{1cm}(\text{mod }t^2)\\
		=&\mu_{\alpha,\,\beta}(a,b)+t\big(\psi_{\alpha\,\beta}^1\mu_{\alpha,\,\beta}(a,b)+\mu_{\alpha,\,\beta}^1(a,b)-\mu_{\alpha,\,\beta}(\psi_{\alpha}^1(a),\,b)-\mu_{\alpha,\,\beta}(a,\,\psi_{\beta}^1(b))\big)\\
		=&\mu_{\alpha,\,\beta}(a,b)+t\big(\psi_{\alpha\,\beta}^1\mu_{\alpha,\,\beta}(a,b)+(\delta_{\mathrm{Alg}}^1\psi^1)_{\alpha,\,\beta}(a,b)-\mu_{\alpha,\,\beta}(\psi_{\alpha}^1(a),\,b)-\mu_{\alpha,\,\beta}(a,\,\psi_{\beta}^1(b))\big)\\
		&\hspace{1cm}(\text{by $\mu_{\alpha,\,\beta}^1=\big(\delta_{\mathrm{Alg}}^1(\psi^1)\big)_{\alpha,\,\beta}$})\\
		=&\mu_{\alpha,\,\beta}(a,b)+t\big(\psi_{\alpha\,\beta}^1\mu_{\alpha,\,\beta}(a,b)+\mu_{\alpha,\,\beta}(a,\,\psi_{\beta}^1(b))-\psi_{\alpha\,\beta}^{1}\mu_{\alpha,\,\beta}(a,b)+\mu_{\alpha,\,\beta}(\psi_{\alpha}^{1}(a),\,b)\\
		&-\mu_{\alpha,\,\beta}(\psi_{\alpha}^1(a),\,b)-\mu_{\alpha,\,\beta}(a,\,\psi_{\beta}^1(b))\big)\hspace{1cm}(\text{by Eq.~(\ref{BiHomOmHochschdiff})})\\
		=&\mu_{\alpha,\,\beta}(a,b).
	\end{align*}
Similarly, we get $\bar{R}_{\omega}^t=R_{\omega}.$ So, we get $(\bar{\mu}_{\alpha,\,\beta}^1)_{\alpha,\,\beta\in\Omega}=0,\,(\bar{R}_{\omega}^1)_{\omega\in\Omega}=0$. Thus, the coefficient of $t$ in the formal expression of $(\bar{\mu}_{\alpha,\,\beta}^t,\bar{R}_{\omega}^t)_{\alpha,\,\beta,\,\omega\in\Omega}$ vanishes. By repeating this process, we obtain that the deformation $(\mu_{\alpha,\,\beta}^t,R_{\omega}^t)_{\alpha,\,\beta,\,\omega\in \Omega}$ is equivalent to $(\mu_{\alpha,\,\beta},R_{\omega})_{\alpha,\,\beta,\,\omega\in \Omega}$. Hence, $(A,\mu_{\alpha,\,\beta},R_{\omega},p_{\omega},q_{\omega})_{\alpha,\,\beta,\,\omega\in \Omega}$ is rigid. This completes the proof.
	\end{proof}

\section{Abelian extensions of Rota-Baxter family BiHom-$\Omega$-associative algebras}\label{sec5}
In this section, we mainly study the abelian extensions of Rota-Baxter family BiHom-$\Omega$-associative algebras. We show that the cohomology $\mathrm{H}_{\mathrm{RBFA}_{\lambda}}^2(A,M)$ can be interpreted as equivalence classes of abelian extensions of Rota-Baxter family BiHom-$\Omega$-associative algebras.

{\em Convention:} In this section, let $(A,\mu_{\alpha,\,\beta},R_{\omega},p_{\omega}^A,q_{\omega}^A)_{\alpha,\,\beta,\,\omega\in \Omega}$ and $(M,\mu_{\alpha,\,\beta}^M,T_{\omega},p_{\omega}^M,q_{\omega}^M)_{\alpha,\,\beta,\,\omega\in \Omega}$ be two Rota-Baxter family BiHom-$\Omega$-associative algebras, where $\mu_{\alpha,\,\beta}^M:=0$ for any $\alpha,\,\beta\in \Omega$. That is to say, $(M,T_{\omega},p_{\omega},q_{\omega})_{\omega\in\Omega}$ is a trivial Rota-Baxter family BiHom-$\Omega$-associative algebra.

\begin{defn}\label{abel-ext}
	An \textbf{abelian extension} of Rota-Baxter family BiHom-$\Omega$-associative algebras is a short exact sequence of Rota-Baxter family BiHom-$\Omega$-associative algebras
	\[0\longrightarrow(M,0,T_{\omega},p_{\omega}^M,q_{\omega}^M)_{\omega\in \Omega}\overset{i_{\alpha}}{\longrightarrow}(E,\mu_{\alpha,\,\beta}^E,T_{\omega}^E,p_{\omega}^E,q_{\omega}^E)_{\alpha,\,\beta,\,\omega\in \Omega}\overset{\rho_{\alpha}}{\longrightarrow}(A,\mu_{\alpha,\,\beta},R_{\omega},p_{\omega}^A,q_{\omega}^A)_{\alpha,\,\beta,\,\omega\in \Omega}\longrightarrow 0,\]
	that is, satisfying
	\begin{align}
		i_{\alpha}\circ p_{\alpha}^M=p_{\alpha}^E\circ i_{\alpha},\quad i_{\alpha}\circ q_{\alpha}^M=q_{\alpha}^E\circ i_{\alpha},\label{ip=pi}
	\end{align}
	\[\rho_{\alpha}\circ p_{\alpha}^E=p_{\alpha}^A\circ\rho_{\alpha},\quad \rho_{\alpha}\circ q_{\alpha}^E=q_{\alpha}^A\circ \rho_{\alpha},\] \[\rho_{\alpha\,\beta}\circ\mu_{\alpha,\,\beta}^E=\mu_{\alpha,\,\beta}(\rho_{\alpha}\otimes \rho_{\beta}),\quad \text{ for all }\alpha,\,\beta\in \Omega,\] and there exists a commutative diagram:
\begin{align}\label{iT=Ti}
		\xymatrix{
	0\ar[r] &M\ar[d]^{T_{\alpha}}\ar[r]^{i_{\alpha}}&E\ar[d]^{T_{\alpha}^E}\ar[r]^{\rho_{\alpha}}&A\ar[d]^{R_{\alpha}}\ar[r]&0\\
	0\ar[r]&M\ar[r]_{i_{\alpha}}&E\ar[r]_{\rho_{\alpha}}&A\ar[r]&0.}
\end{align}
In this case, we call $(E,\mu_{\alpha,\,\beta}^E,T_{\omega}^E,p_{\omega}^E,q_{\omega}^E)_{\alpha,\,\beta,\,\omega\in \Omega}$ an abelian extension of Rota-Baxter family BiHom-$\Omega$-associative algebra $(A,\mu_{\alpha,\,\beta},R_{\omega},p_{\omega}^A,q_{\omega}^A)_{\alpha,\,\beta,\,\omega\in \Omega}$ by $(M,T_{\omega},p_{\omega}^M,q_{\omega}^M)_{\omega\in \Omega}$.
\end{defn}

A \textbf{section} of an abelian extension $(E,\mu_{\alpha,\,\beta}^E,T_{\omega}^E,p_{\omega}^E,q_{\omega}^E)_{\alpha,\,\beta,\,\omega\in \Omega}$ of $(A,\mu_{\alpha,\,\beta},R_{\omega},p_{\omega}^A,q_{\omega}^A)_{\alpha,\,\beta,\,\omega\in \Omega}$ by $(M,T_{\omega},p_{\omega}^M,q_{\omega}^M)_{\omega\in \Omega}$ is a family of linear maps $( s_{\alpha})_{\alpha\in \Omega}: A\rightarrow E$ satisfying
\begin{align}\label{section}
	p_{\alpha}^E\circ s_{\alpha}=s_{\alpha}\circ p_{\alpha}^A,\quad q_{\alpha}^E\circ s_{\alpha}=s_{\alpha}\circ q_{\alpha}^A,\quad \rho_{\alpha}\circ s_{\alpha}=id_{A},
\end{align}
for all $\alpha\in \Omega$.

Let $(E,\mu_{\alpha,\,\beta}^E,T_{\omega}^E,p_{\omega}^E,q_{\omega}^E)_{\alpha,\,\beta,\,\omega\in \Omega}$ be an abelian extension of $(A,\mu_{\alpha,\,\beta},R_{\omega},p_{\omega}^A,q_{\omega}^A)_{\alpha,\,\beta,\,\omega\in \Omega}$ by $(M,T_{\omega},\\p_{\omega}^M,q_{\omega}^M)_{\omega\in \Omega}$ and let $( s_{\alpha})_{\alpha\in \Omega}: A\rightarrow E$ be a section of $E$. We define the actions $ (\rhd_{\alpha,\,\beta})_{\alpha,\,\beta\in\Omega}: A\otimes M\rightarrow M $ and $(\lhd_{\alpha,\,\beta})_{\alpha,\,\beta\in\Omega}: M\otimes A\rightarrow M$ by
\[a\rhd_{\alpha,\,\beta}m:=\mu_{\alpha,\,\beta}^E\big(s_{\alpha}(a),i_{\beta}(m)\big),\quad m\lhd_{\alpha,\,\beta}a:=\mu_{\alpha,\,\beta}^E\big(i_{\alpha}(m),s_{\beta}(a)\big),\]
for all $a\in A,\, m\in M,\, \alpha,\,\beta\in \Omega$.

Next, we show that an abelian extension induces a bimodule structure by actions $( \rhd_{\alpha,\,\beta})_{\alpha,\,\beta\in \Omega}$ and $( \rhd_{\alpha,\,\beta})_{\alpha,\,\beta\in\Omega}$.

\begin{prop}\label{ext-prop1}
	Under the above actions, $(M,\rhd_{\alpha,\,\beta},\lhd_{\alpha,\,\beta},T_{\omega},p_{\omega}^M,q_{\omega}^M)_{\alpha,\,\beta,\,\omega\in \Omega}$ is a Rota-Baxter family BiHom-$\Omega$-bimodule over Rota-Baxter family BiHom-$\Omega$-associative algebra $(A,\mu_{\alpha,\,\beta},R_{\omega},p_{\omega}^A,\\q_{\omega}^A)_{\alpha,\,\beta,\,\omega\in \Omega}$.
\end{prop}
\begin{proof}
	For any $a,b,c\in A,\; \alpha,\,\beta,\,\gamma\in \Omega,\; m\in M$, owing to $\rho_{\alpha}\circ s_{\alpha}=id_{A}$, we have
	\begin{align*}
		\rho&_{\alpha\,\beta}\Big(s_{\alpha\,\beta}\mu_{\alpha,\,\beta}(a,b)-\mu_{\alpha,\,\beta}^E\big(s_{\alpha}(a),a_{\beta}(b)\big)\Big)\\
		=&\rho_{\alpha\,\beta}s_{\alpha\,\beta}\mu_{\alpha,\,\beta}(a,b)-\mu_{\alpha,\,\beta}\big(\rho_{\alpha}s_{\alpha}(a),\rho_{\beta}s_{\beta}(b)\big)\\
		=&\mu_{\alpha,\,\beta}(a,b)-\mu_{\alpha,\,\beta}(a,b)=0,
	\end{align*}
then we get $s_{\alpha\,\beta}\mu_{\alpha,\,\beta}(a,b)-\mu_{\alpha,\,\beta}^E\big(s_{\alpha}(a),s_{\beta}(b)\big)\in M$. Similarly, we have $T_{\alpha}^Es_{\alpha}(a)-s_{\alpha}R_{\alpha}(a)\in M$. Furthermore, by $\mu_{\alpha,\,\beta}^M=0$, then we have
\[\mu_{\alpha\,\beta,\,\gamma}^E\big(s_{\alpha\,\beta}\mu_{\alpha,\,\beta}(a,b),\,i_{\gamma}(m)\big)=\mu_{\alpha\,\beta,\,\gamma}^E\Big(\mu_{\alpha,\,\beta}^E\big(s_{\alpha}(a),\,s_{\beta}(b)\big),i_{\gamma}(m)\Big).\]
Now, we prove Eq.~(\ref{lmod-1}).
\begin{align*}
	p_{\alpha\,\beta}^M(a\rhd_{\alpha,\,\beta}m)=&p_{\alpha\,\beta}^M\mu_{\alpha,\,\beta}^E\big(s_{\alpha}(a),i_{\beta}(m)\big)\\
	=&p_{\alpha\,\beta}^E\mu_{\alpha,\,\beta}^E\big(s_{\alpha}(a),i_{\beta}(m)\big)\\
	=&\mu_{\alpha,\,\beta}^E\big(p_{\alpha}^Es_{\alpha}(a),\,p_{\beta}^Ei_{\beta}(m)\big)\hspace{1cm}(\text{by Eq.~(\ref{eqalfabeta})})\\
	=&\mu_{\alpha,\,\beta}^E\big(s_{\alpha}p_{\alpha}^A(a),i_{\beta}p_{\beta}^M(m)\big)\hspace{1cm}(\text{by Eq.~(\ref{ip=pi}) and Eq.~(\ref{section})})\\
	=&p_{\alpha}^A(a)\rhd_{\alpha,\,\beta}p_{\beta}^M(m).
\end{align*}
Similarly, we get Eq.~(\ref{lmod-2}). Next, we check Eq.~(\ref{lmod}).
\begin{align*}
	\mu_{\alpha,\,\beta}(a,b)\rhd_{\alpha\,\beta,\,\gamma}q_{\gamma}^M(m)=&\mu_{\alpha\,\beta,\,\gamma}^E\big(s_{\alpha\,\beta}\mu_{\alpha,\,\beta}(a,b),i_{\gamma}q_{\gamma}^E(m)\big)=\mu_{\alpha\,\beta,\,\gamma}^E\Big(\mu_{\alpha,\,\beta}^E\big(s_{\alpha}(a),s_{\beta}(b)\big),i_{\gamma}q_{\gamma}^M(m)\Big)\\
	=&\mu_{\alpha\,\beta,\,\gamma}^E\Big(\mu_{\alpha,\,\beta}^E\big(s_{\alpha}(a),\,s_{\beta}(b)\big),\,q_{\gamma}^Ei_{\gamma}(m)\Big)\\
	=&\mu_{\alpha,\,\beta\,\gamma}^E\Big(p_{\alpha}^Es_{\alpha}(a),\,\mu_{\beta,\,\gamma}^E\big(s_{\beta}(b),\,i_{\gamma}(m)\big)\Big)\hspace{1cm}(\text{by Eq.~(\ref{eqasso})})\\
	=&\mu_{\alpha,\,\beta\,\gamma}^E\Big(s_{\alpha}p_{\alpha}^A(a),\,\mu_{\beta,\,\gamma}^E\big(s_{\beta}(b),\,i_{\gamma}(m)\big)\Big)\hspace{1cm}(\text{by Eq.~(\ref{section})})\\
	=&\mu_{\alpha,\,\beta\,\gamma}^E\big(s_{\alpha}p_{\alpha}^A(a),\, b\rhd_{\beta,\,\gamma}m\big)\\
	=&\mu_{\alpha,\,\beta\,\gamma}^E\big(s_{\alpha}p_{\alpha}^A(a),\,i_{\beta\,\gamma}( b\rhd_{\beta,\,\gamma}m)\big)\\
	=&p_{\alpha}^A(a)\rhd_{\alpha,\,\beta\,\gamma}(b\rhd_{\beta,\,\gamma}m).
\end{align*}
So we get that $(M,\rhd_{\alpha,\,\beta},p_{\omega}^M,q_{\omega}^M)_{\alpha,\,\beta,\,\omega\in \Omega}$ is a left module over $A$. By the same way, we further obtain that $(M,\rhd_{\alpha,\,\beta},\lhd_{\alpha,\,\beta},p_{\omega}^M,q_{\omega}^M)_{\alpha,\,\beta,\,\omega\in \Omega}$ is a bimodule over $A$. Since $(M,T_{\omega},p_{\omega}^M,q_{\omega}^M)_{\omega\in\Omega}$ is a trivial Rota-Baxter family BiHom-$\Omega$-associative algebra, we get 
\[T_{\alpha}\circ p_{\alpha}^M=p_{\alpha}^M\circ T_{\alpha},\; T_{\alpha}\circ q_{\alpha}^M=q_{\alpha}^M\circ T_{\alpha}.\]
Then by Eq.~(\ref{iT=Ti}) and $T_{\alpha}^Es_{\alpha}(a)-s_{\alpha}R_{\alpha}(a)\in M$, we obtain that Eqs.~(\ref{RBbimodulefamily-1})-(\ref{RBbimodulefamily-2}) hold. Thus, $(M,\rhd_{\alpha,\,\beta},\lhd_{\alpha,\,\beta}, T_{\omega}, p_{\omega}^M,q_{\omega}^M)_{\alpha,\,\beta,\,\omega\in \Omega}$ is a Rota-Baxter family BiHom-$\Omega$-bimodule over $A$. This completes the proof.
\end{proof}

Inspired by Proposition~\ref{ext-prop1}, we define $( \psi_{\alpha,\,\beta})_{\alpha,\,\beta\in \Omega}: A\otimes A\rightarrow M$ and $( \chi_{\omega})_{\omega\in \Omega}: A\rightarrow M$ by
\begin{align}
	\psi_{\alpha,\,\beta}(a,b):=&\mu_{\alpha,\,\beta}^E\big(s_{\alpha}(a),\, s_{\beta}(b)\big)-s_{\alpha\,\beta}\mu_{\alpha,\,\beta}(a,b),\label{def-eq-psi}\\
	\chi_{\omega}(a):=&T_{\omega}^Es_{\omega}(a)-s_{\omega}R_{\omega}(a),\label{def-eq-chi}
\end{align}
for all $a,b\in A,\,\alpha,\,\beta,\,\omega\in \Omega$. Then we have the following results.

\begin{prop}\label{prop-psichi-2-cocycle}
	The pair $(\psi_{\alpha,\,\beta}, \chi_{\omega})_{\alpha,\,\beta,\,\omega\in \Omega}$ is a 2-cocycle in the cochain complex $\mathrm{C}_{\mathrm{RBFA}_{\lambda}}^{2}(A,M)$.
\end{prop}
\begin{proof}
	For any $a,b,c\in A,\;\alpha,\,\beta,\,\gamma,\,\omega,\,\omega_{1}\in \Omega,$ by Eqs.~(\ref{eqalfabeta}),~(\ref{pR=Rp}) and Eqs.~(\ref{section})-(\ref{def-eq-chi}), we have
	\[p_{\alpha\,\beta}^M\circ \psi_{\alpha,\,\beta}=\psi_{\alpha,\,\beta}\circ(p_{\alpha}^A\otimes p_{\beta}^A),\quad p_{\omega}^M\circ \chi_{\omega}=\chi_{\omega}\circ p_{\omega}^A,\]
	\[q_{\alpha\,\beta}^M\circ \psi_{\alpha,\,\beta}=\psi_{\alpha,\,\beta}\circ(q_{\alpha}^A\otimes q_{\beta}^A),\quad q_{\omega}^M\circ \chi_{\omega}=\chi_{\omega}\circ q_{\omega}^A.\]
	With a simple calculation, we obtain $(\psi_{\alpha,\,\beta})_{\alpha,\,\beta\in\Omega}\in \mathrm{C}_{\Omega}^2(A,M),\; (\chi_{\omega})_{\omega\in\Omega}\in \mathrm{C}_{\mathrm{RBF}_{\lambda}}^1(A,M)$. By Definition~\ref{diffof-RBFbihomOassoalg}, we get \[d^2(\psi,\chi)_{\alpha,\,\beta,\,\gamma,\,\omega,\,\omega_{1}}=\big(\delta_{\mathrm{Alg}}^2(\psi)_{\alpha,\,\beta,\,\gamma},\, -\partial^1(\chi)_{\omega,\,\omega_{1}}-\Phi^2(\psi)_{\omega,\,\omega_{1}}\big).\] Now we are going to prove $\delta_{\mathrm{Alg}}^2(\psi)_{\alpha,\,\beta,\,\gamma}=0$.
	\begin{align*}
		&\delta_{\mathrm{Alg}}^2(\psi)_{\alpha,\,\beta,\,\gamma}(a,b,c)\\
		=&p_{\alpha}^A(a)\rhd_{\alpha,\,\beta\,\gamma}\psi_{\beta,\,\gamma}(b,c)-\psi_{\alpha\,\beta,\,\gamma}\big(\mu_{\alpha,\,\beta}(a,b),\, q_{\gamma}^A(c)\big)+\psi_{\alpha,\,\beta\,\gamma}\big(p_{\alpha}^A(a),\,\mu_{\beta,\,\gamma}(b,c)\big)-\psi_{\alpha,\,\beta}(a,b)\lhd_{\alpha\,\beta,\,\gamma}q_{\gamma}^A(c)\\
		=&p_{\alpha}^A(a)\rhd_{\alpha,\,\beta\,\gamma}\mu_{\beta,\,\gamma}^E\big(s_{\beta}(b),s_{\gamma}(c)\big)-p_{\alpha}^A(a)\rhd_{\alpha,\,\beta\,\gamma}s_{\beta\,\gamma}\mu_{\beta,\,\gamma}(b,c)-\mu_{\alpha\,\beta,\,\gamma}^E\big(s_{\alpha\,\beta}\mu_{\alpha,\,\beta}(a,b),\,s_{\gamma}q_{\gamma}^A(c)\big)\\
		&+s_{\alpha\,\beta\,\gamma}\mu_{\alpha\,\beta,\,\gamma}\big(\mu_{\alpha,\,\beta}(a,b),\,q_{\gamma}^A(c)\big)+\mu_{\alpha,\,\beta\,\gamma}^E\big(s_{\alpha}p_{\alpha}^A(a),\, s_{\beta\,\gamma}\mu_{\beta,\,\gamma}(b,c)\big)-s_{\alpha\,\beta\,\gamma}\mu_{\alpha,\,\beta\,\gamma}\big(p_{\alpha}^A(a),\,\mu_{\beta,\,\gamma}(b,c)\big)\\
		&-\mu_{\alpha,\,\beta}^E\big(s_{\alpha}(a),\,s_{\beta}(b)\big)\lhd_{\alpha\,\beta,\,\gamma}q_{\gamma}^A(c)+s_{\alpha\,\beta}\mu_{\alpha,\,\beta}(a,b)\lhd_{\alpha\,\beta,\,\gamma}q_{\gamma}^A(c)\\
		=&\mu_{\alpha,\,\beta\,\gamma}^E\Big(s_{\alpha}p_{\alpha}^A(a),\,\mu_{\beta,\,\gamma}^E\big(s_{\beta}(b),\, s_{\gamma}(c)\big)\Big)-\mu_{\alpha,\,\beta\,\gamma}^E\big(s_{\alpha}p_{\alpha}^A(a),\, s_{\beta\,\gamma}\mu_{\beta,\,\gamma}(b,c)\big)-\mu_{\alpha\,\beta,\,\gamma}^E\big(s_{\alpha\,\beta}\mu_{\alpha,\,\beta}(a,b),\,s_{\gamma}q_{\gamma}^A(c)\big)\\
		&+s_{\alpha\,\beta\,\gamma}\mu_{\alpha\,\beta,\,\gamma}\big(\mu_{\alpha,\,\beta}(a,b),\, q_{\gamma}^A(c)\big)+\mu_{\alpha,\,\beta\,\gamma}^E\big(s_{\alpha}p_{\alpha}^A(a),\, s_{\beta\,\gamma}\mu_{\beta,\,\gamma}(b,c)\big)-s_{\alpha\,\beta\,\gamma}\mu_{\alpha,\,\beta\,\gamma}\big(p_{\alpha}^A(a),\, \mu_{\beta,\,\gamma}(b,c)\big)\\
		&-\mu_{\alpha\,\beta,\,\gamma}^E\Big(\mu_{\alpha,\,\beta}^E\big(s_{\alpha}(a),\, s_{\beta}(b)\big),\, s_{\gamma}q_{\gamma}^A(c)\Big)+\mu_{\alpha\,\beta,\,\gamma}^E\big(s_{\alpha\,\beta}\mu_{\alpha,\,\beta}(a,b),\, s_{\gamma}q_{\gamma}^A(c)\big)\\
		=&\mu_{\alpha,\,\beta\,\gamma}^E\Big(p_{\alpha}^Es_{\alpha}(a),\, \mu_{\beta,\,\gamma}^E\big(s_{\beta}(b),\, s_{\gamma}(c)\big)\Big)-\mu_{\alpha\,\beta,\,\gamma}^E\Big(\mu_{\alpha,\,\beta}^E\big(s_{\alpha}(a),\, s_{\beta}(b)\big),\, q_{\gamma}^Es_{\gamma}(c)\Big)\hspace{1cm}(\text{by Eq.~(\ref{section})})\\
		=&0.\hspace{1cm}(\text{by Eq.~(\ref{eqasso})})
	\end{align*}
Similarly, we have $\partial^1(\chi)_{\omega,\,\omega_{1}}+\Phi^2(\psi)_{\omega,\,\omega_{1}}=0$. Thus,  $(\psi_{\alpha,\,\beta}, \chi_{\omega})_{\alpha,\,\beta,\,\omega\in \Omega}$ is a 2-cocycle.
\end{proof}

Next, we show that the definition of $\rhd_{\alpha,\,\beta},\lhd_{\alpha,\,\beta},\psi_{\alpha,\,\beta}$ and $\chi_{\omega}$ are independent of the choice of section $s_{\alpha}$, for all $\alpha,\,\beta,\,\omega\in \Omega$.
\begin{prop}
	\begin{enumerate}
		\item \label{it: bimod} Different sections give the same Rota-Baxter family BiHom-$\Omega$-bimodule structure on $(M,T_{\omega},p_{\omega}^M, q_{\omega}^M)_{\omega\in \Omega}$.
		\item \label{it: cohomo} The cohomological class of $(\psi_{\alpha,\,\beta}, \chi_{\omega})_{\alpha,\,\beta,\,\omega\in \Omega}$ is independent of the choice of sections.
	\end{enumerate}
\end{prop}
\begin{proof}
		\ref{it: bimod} We just prove the case of left module action $(\rhd_{\alpha,\,\beta})_{\alpha,\,\beta\in \Omega}$. The proof of right module action $(\lhd_{\alpha,\,\beta})_{\alpha,\,\beta\in \Omega}$ is similar. If $(s_{\alpha}^1)_{\alpha\in \Omega}$ and $( s_{\alpha}^2)_{\alpha\in \Omega}$ are different sections, then we have
		\[a\rhd_{\alpha,\,\beta}^1m:=\mu_{\alpha,\,\beta}^E\big(s_{\alpha}^1(a),\,i_{\beta}(m)\big),\quad a\rhd_{\alpha,\,\beta}^2m:=\mu_{\alpha,\,\beta}^E\big(s_{\alpha}^2(a),\, i_{\beta}(m)\big),\]
		for all $a\in A,\,m\in M,\;\alpha,\,\beta\in \Omega$. Now, we define a family of linear maps $( \eta_{\alpha})_{\alpha\in \Omega}: A\rightarrow M$ by
		\[\eta_{\alpha}(a):=s_{\alpha}^1(a)-s_{\alpha}^2(a),\quad \text{for all } a\in A,\;\alpha\in \Omega.\]
		Then by $\mu_{\alpha,\,\beta}^M=0$, we have 
		\begin{align*}
			a\rhd_{\alpha,\,\beta}^1m=&\mu_{\alpha,\,\beta}^E\big(s_{\alpha}^1(a),\, i_{\beta}(m)\big)=\mu_{\alpha,\,\beta}^E\big(\eta_{\alpha}(a)+s_{\alpha}^2(a),\, i_{\beta}(m)\big)\\
			=&\mu_{\alpha,\,\beta}^M\big(\eta_{\alpha}(a), m\big)+\mu_{\alpha,\,\beta}^E\big(s_{\alpha}^2(a),\, i_{\beta}(m)\big)\\
			=&a\rhd_{\alpha,\,\beta}^2m.
		\end{align*}
	Hence, different sections give the same left module structure on $M$. This completes the proof.
	
	\ref{it: cohomo} For any $a,b\in A,\; \alpha,\,\beta,\,\omega\in \Omega,$ here we continue to use the notation in~\ref{it: bimod}, for different sections $( s_{\alpha}^1)_{\alpha\in\Omega}$ and $( s_{\alpha}^2)_{\alpha\in\Omega}$, we define the corresponding $(\psi_{\alpha,\,\beta}^1, \chi_{\omega}^1)_{\alpha,\,\beta,\,\omega\in\Omega}$ and $(\psi_{\alpha,\,\beta}^2, \chi_{\omega}^2)_{\alpha,\,\beta,\,\omega\in\Omega}$ as follows:
	
	\centerline{$\psi_{\alpha,\,\beta}^1(a,b)=\mu_{\alpha,\,\beta}^E\big(s_{\alpha}^1(a),\, s_{\beta}^1(b)\big)-s_{\alpha\,\beta}^1\mu_{\alpha,\,\beta}(a,b),\quad \chi_{\omega}^1(a)=T_{\omega}^Es_{\omega}^1(a)-s_{\omega}^1R_{\omega}(a),$}
	\centerline{$\psi_{\alpha,\,\beta}^2(a,b)=\mu_{\alpha,\,\beta}^E\big(s_{\alpha}^2(a),\, s_{\beta}^2(b)\big)-s_{\alpha\,\beta}^2\mu_{\alpha,\,\beta}(a,b),\quad \chi_{\omega}^2(a)=T_{\omega}^Es_{\omega}^2(a)-s_{\omega}^2R_{\omega}(a).$}
	We are going to prove that $(\psi_{\alpha,\,\beta}^1, \chi_{\omega}^1)_{\alpha,\,\beta,\,\omega\in\Omega}-(\psi_{\alpha,\,\beta}^2,\chi_{\omega}^2)_{\alpha,\,\beta,\,\omega\in\Omega}\in Im(d^{1})$, we have
	\begin{align*}
		\psi_{\alpha,\,\beta}^1(a,b)-\psi_{\alpha,\,\beta}^2(a,b)=&\mu_{\alpha,\,\beta}^E\big(s_{\alpha}^1(a),\, s_{\beta}^1(b)\big)-s_{\alpha\,\beta}^1\mu_{\alpha,\,\beta}(a,b)-\mu_{\alpha,\,\beta}^E\big(s_{\alpha}^2(a),\, s_{\beta}^2(b)\big)+s_{\alpha\,\beta}^2\mu_{\alpha,\,\beta}(a,b)\\
		=&\mu_{\alpha,\,\beta}^E\big(\eta_{\alpha}(a)+s_{\alpha}^2(a),\, \eta_{\beta}(b)+s_{\beta}^2(b)\big)-\eta_{\alpha\,\beta}\mu_{\alpha,\,\beta}(a,b)-s_{\alpha\,\beta}^2\mu_{\alpha,\,\beta}(a,b)\\
		&-\mu_{\alpha,\,\beta}^E\big(s_{\alpha}^2(a),s_{\beta}^2(b)\big)+s_{\alpha\,\beta}^2\mu_{\alpha,\,\beta}(a,b)\\
		=&\mu_{\alpha,\,\beta}^E\big(\eta_{\alpha}(a),\, s_{\beta}^2(b)\big)+\mu_{\alpha,\,\beta}^E\big(s_{\alpha}^2(a),\, \eta_{\beta}(b)\big)-\eta_{\alpha\,\beta}\mu_{\alpha,\,\beta}(a,b)\\
		=&\eta_{\alpha}(a)\lhd_{\alpha,\,\beta}^2b+a\rhd_{\alpha,\,\beta}^2\eta_{\beta}(b)-\eta_{\alpha\,\beta}\mu_{\alpha,\,\beta}(a,b)\\
		=&\big(\delta_{\mathrm{Alg}}^1(\eta)\big)_{\alpha,\,\beta}(a,b)
	\end{align*}
Similarly, we get $\chi_{\omega}^1(a)-\chi_{\omega}^2(a)=-\big(\Phi^1(\eta)\big)_{\omega}(a)$. So we obtain that \[(\psi_{\alpha,\,\beta}^1,\chi_{\omega}^1)_{\alpha,\,\beta,\,\omega\in\Omega}-(\psi_{\alpha,\,\beta}^2,\chi_{\omega}^2)_{\alpha,\,\beta,\,\omega\in\Omega}=\big(\delta_{\mathrm{Alg}}^1(\eta))_{\alpha,\,\beta},\, -\Phi^1(\eta)_{\omega}\big)_{\alpha,\,\beta,\,\omega\in\Omega}\in Im(d^{1}).\] 
This completes the proof.
\end{proof}

\begin{defn}\label{def-isoext}
	Two abelian extensions are said to be \textbf{isomorphic} if there exists an isomorphism $\phi=(\phi_{\alpha})_{\alpha\in\Omega}: E\rightarrow E'$ on Rota-Baxter family BiHom-$\Omega$-associative algebras such that the following diagram commute:
	\begin{align*}
		\xymatrix{
				0\ar[r] &(M,T_{\omega}^M,p_{\alpha}^M,q_{\alpha}^M)_{\alpha,\,\omega\in \Omega}\ar@{=}[d]\ar[r]^{i_{\alpha}^1\quad}&(E,\mu_{\alpha,\,\beta}^E,T_{\omega}^E,p_{\alpha}^E,q_{\alpha}^E)_{\alpha,\,\beta,\,\omega\in \Omega}\ar[d]^{\phi_{\alpha}}\ar[r]^{\rho_{\alpha}^1}&(A,\mu_{\alpha,\,\beta},R_{\omega},p_{\alpha}^A,q_{\alpha}^A)_{\alpha,\,\beta,\,\omega\in \Omega}\ar@<.5ex>[l]^{s_{\alpha}^1}\ar@{=}[d]\ar[r]&0\\
				0\ar[r]&(M,T_{\omega}^M,p_{\alpha}^M,q_{\alpha}^M)_{\alpha,\,\omega\in \Omega}\ar[r]^{i_{\alpha}^2\quad}&(\bar{E},\bar{\mu}_{\alpha,\,\beta}^E,\bar{T}_{\omega}^E,\bar{p}_{\alpha}^E,\bar{q}_{\alpha}^E)_{\alpha,\,\beta,\,\omega\in \Omega}\ar[r]^{\rho_{\alpha}^2}&(A,\mu_{\alpha,\,\beta},R_{\omega},p_{\alpha}^A,q_{\alpha}^A)_{\alpha,\,\beta,\,\omega\in \Omega}\ar@<.5ex>[l]^{s_{\alpha}^2}\ar[r]&0.}
	\end{align*}
\end{defn}
\noindent Note that two extension with same $(i_{\alpha})_{\alpha\in\Omega}$ and $(\rho_{\alpha})_{\alpha\in\Omega}$ but different $(s_{\alpha})_{\alpha\in\Omega}$ are always isomorphic.

In fact, the section $( s_{\alpha})_{\alpha\in \Omega}$ determines the following splitting
\[\xymatrix{0\ar[r]&M\ar[r]^{i_{\alpha}}&E\ar@<.5ex>[l]^{t_{\alpha}}\ar[r]^{\rho_{\alpha}}&A\ar@<.5ex>[l]^{s_{\alpha}}\ar[r]&0},\]
where $t_{\alpha}\circ i_{\alpha}=\text{id}_{\text{M}},\; t_{\alpha}\circ s_{\alpha}=0$ and $i_{\alpha}\circ t_{\alpha}+s_{\alpha}\circ \rho_{\alpha}=id_{\text{E}}$ for all $\alpha\in \Omega$. By~\cite{RBA,RBFassoconformal}, there is an isomorphism of vector spaces:

\renewcommand{\atopwithdelims}[2]{%
	\genfrac{(}{)}{0pt}{}{#1}{#2}}
\[(\rho_{\alpha}, t_{\alpha}): E\cong A\oplus M : \atopwithdelims{s_{\alpha}}{i_{\alpha}}.\]
Thus, we will study the Rota-Baxter family BiHom-$\Omega$-associative algebra structure on $A\oplus M$, where $( \mu_{\alpha,\,\beta}^{\psi})_{\alpha,\,\beta\in\Omega}$, $( T_{\omega}^{\chi})_{\omega\in\Omega}$, $( p_{\omega})_{\omega\in\Omega}, (q_{\omega})_{\omega\in\Omega}$ are defined by
\begin{align}
	\mu_{\alpha,\,\beta}^{\psi}\big((a,m),\, (b,n)\big):=&\big(\mu_{\alpha,\,\beta}(a,b),\, a\rhd_{\alpha,\,\beta}n+m\lhd_{\alpha,\,\beta}b+\psi_{\alpha,\,\beta}(a,b)\big),\label{def-mu-psi}\\
	T_{\omega}^{\chi}(a,m):=&\big(R_{\omega}(a),\, \chi_{\omega}(a)+T_{\omega}^M(m)\big),\label{def-T-chi}\\
	p_{\omega}(a,m):=&\big(p_{\omega}^A(a),\, p_{\omega}^M(m)\big),\label{def-p}\\
	q_{\omega}(a,m):=&\big(q_{\omega}^A(a),\, q_{\omega}^M(m)\big),\label{def-q}
\end{align}
for all $(a,m),\, (b,n)\in A\oplus M,$ and $\alpha,\,\beta,\,\omega\in \Omega$. In particular, if $(\psi_{\alpha,\,\beta})_{\alpha,\,\beta\in\Omega}=0,\, (\chi_{\omega})_{\omega\in\Omega}=0$, then $(A\oplus M,\mu_{\alpha,\,\beta}^{\psi},T_{\omega}^{\chi},p_{\omega},q_{\omega})_{\alpha,\,\beta,\,\omega\in \Omega}$ becomes the semi-direct product of $(A,\mu_{\alpha,\,\beta},R_{\omega},p_{\omega}^A,q_{\omega}^A)_{\alpha,\,\beta,\,\omega\in \Omega}$ by $(M,T_{\omega}^M,p_{\omega}^M,q_{\omega}^M)_{\omega\in \Omega}$. Moreover, we get an abelian extension
\[0\longrightarrow(M,T_{\omega},p_{\omega}^M,q_{\omega}^M)_{\omega\in \Omega}\overset{i_{\alpha}}{\longrightarrow}(A\oplus M,\mu_{\alpha,\,\beta}^{\psi},T_{\omega}^{\chi},p_{\omega},q_{\omega})_{\alpha,\,\beta,\,\omega\in \Omega}\overset{\rho_{\alpha}}{\longrightarrow}(A,\mu_{\alpha,\,\beta},R_{\omega},p_{\omega}^A,q_{\omega}^A)_{\alpha,\,\beta,\,\omega\in \Omega}\longrightarrow 0,\]
which is isomorphic to the original one in Definition~\ref{abel-ext}.

Let $(M,T_{\omega}^M,p_{\omega}^M,q_{\omega}^M)_{\omega\in \Omega}$ be a Rota-Baxter family BiHom-$\Omega$-bimodule over the Rota-Baxter family BiHom-$\Omega$-associative algebra $(A,\mu_{\alpha,\,\beta},R_{\omega},p_{\omega}^A,q_{\omega}^A)_{\alpha,\,\beta,\,\omega\in \Omega}$. Recall the structure on $A\oplus M$ that was already defined in Eqs.~(\ref{def-mu-psi})-(\ref{def-q}). We have the following result.

\begin{lemma}\label{lemma}
	The quintuple $(A\oplus M, \mu_{\alpha,\,\beta}^{\psi},T_{\omega}^{\chi},p_{\omega},q_{\omega})_{\alpha,\,\beta,\,\omega\in \Omega}$ is a Rota-Baxter family BiHom-$\Omega$-associative algebra if and only if $(\psi_{\alpha,\,\beta}, \chi_{\omega})_{\alpha,\,\beta,\,\omega\in \Omega}$ is a 2-cocycle in the cochain complex $\mathrm{C}_{RBFA_{\lambda}}^{\bullet}(A,M)$.
\end{lemma}
\begin{proof}
	In this case, we have the abelian extension
		\[0\longrightarrow(M,T_{\omega},p_{\omega}^M,q_{\omega}^M)_{\omega\in \Omega}\overset{(0,\,\text{id})}{\longrightarrow}(A\oplus M, \mu_{\alpha,\,\beta}^{\psi},T_{\omega}^{\chi},p_{\omega},q_{\omega})_{\alpha,\,\beta,\,\omega\in \Omega}\overset{\atopwithdelims{\text{id}}{0}}{\longrightarrow}(A,\mu_{\alpha,\,\beta},R_{\omega},p_{\omega}^A,q_{\omega}^A)_{\alpha,\,\beta,\,\omega\in \Omega}\longrightarrow 0,\]
		where section $( s_{\alpha})_{\alpha\in \Omega}=(id,\, 0): (A,\mu_{\alpha,\,\beta},R_{\omega},p_{\omega}^A,q_{\omega}^A)_{\alpha,\,\beta,\,\omega\in \Omega}\rightarrow (A\oplus M, \mu_{\alpha,\,\beta}^{\psi},T_{\omega}^{\chi},p_{\omega},q_{\omega})_{\alpha,\,\beta,\,\omega\in \Omega}$ and the bimodule structure on $M$ is the prescribed one.
		For any $ \alpha,\,\beta,\,\gamma\in\Omega$, by Definition~\ref{def-RBfBHOalg}, we first have
		\[p_{\alpha}\circ T_{\alpha}^{\chi}=T_{\alpha}^{\chi}\circ p_{\alpha},\quad q_{\alpha}\circ T_{\alpha}^{\chi}=T_{\alpha}^{\chi}\circ q_{\alpha},\]
		\[p_{\alpha\,\beta}\circ \mu_{\alpha,\,\beta}^{\chi}=\mu_{\alpha,\,\beta}^{\psi}\big(p_{\alpha}\otimes p_{\beta}\big),\quad q_{\alpha\,\beta}\circ \mu_{\alpha,\,\beta}^{\psi}=\mu_{\alpha,\,\beta}^{\psi}\big(q_{\alpha}\otimes q_{\beta}\big),\]
		which imply
		\[(\chi_{\alpha})_{\alpha\in\Omega}\in \mathrm{C}_{\Omega}^1(A,M),\; (\psi_{\alpha,\,\beta})_{\alpha,\,\beta\in\Omega}\in \mathrm{C}_{\Omega}^2(A,M).\]
		Then, from the equation $\mu_{\alpha,\,\beta\,\gamma}^{\psi}\big(p_{\alpha}\otimes \mu_{\beta,\,\gamma}^{\psi}\big)=\mu_{\alpha\,\beta,\,\gamma}^{\psi}\big(\mu_{\alpha,\,\beta}^{\psi}\otimes q_{\gamma}\big)$, we get $\delta_{\mathrm{Alg}}^2(\psi)_{\alpha,\,\beta,\,\gamma}=0$. By \[\mu_{\alpha,\,\beta}^{\psi}\big(T_{\alpha}^{\chi}\otimes T_{\beta}^{\chi}\big)=T_{\alpha\,\beta}^{\chi}\big(\mu_{\alpha,\,\beta}^{\psi}(T_{\alpha}^{\chi}\otimes id)+\mu_{\alpha,\,\beta}^{\psi}\big(id\otimes T_{\beta}^{\psi})+\lambda \mu_{\alpha,\,\beta}^{\chi}\big),\]
		we get $\partial^1(\chi)_{\alpha,\,\beta}+\Phi^2(\psi)_{\alpha,\,\beta}=0$. Thus, we obtain that $(\psi_{\alpha,\,\beta}, \chi_{\omega})_{\alpha,\,\beta,\,\omega\in \Omega}$ is a 2-cocycle.
		
		Conversely, if $(\psi_{\alpha,\,\beta}, \chi_{\omega})_{\alpha,\,\beta,\,\omega\in \Omega}$ is a 2-cocycle , one can check that $(A\oplus M, \mu_{\alpha,\,\beta}^{\psi},T_{\omega}^{\chi},p_{\omega},q_{\omega})_{\alpha,\,\beta,\,\omega\in \Omega}$ is a Rota-Baxter family BiHom-$\Omega$-associative algebra. This completes the proof.
\end{proof}

Suppose that $M$ is a given bimodule over Rota-Baxter family BiHom-$\Omega$-associative algebra $A$. We denote by $\mathrm{Ext}(A,M)$ the isomorphic classes of abelian extensions of $A$ by $M$ for which the induced bimodule structure on $M$ is the prescribed one.

Now, we show that there is a one-to-one correspondence between the isomorphic classes of abelian extensions $\mathrm{Ext}(A,M)$ and the second cohomology group $\mathrm{H}^2_{\mathrm{RBFA}_{\lambda}}(A,M)$.

\begin{theorem}
	Let $(A,\mu_{\alpha,\,\beta},R_{\omega},p_{\omega}^A,q_{\omega}^A)_{\alpha,\,\beta,\,\omega\in \Omega}$ be a Rota-Baxter family BiHom-$\Omega$-associative algebra and $(M,T_{\omega},p_{\omega}^M,q_{\omega}^M)_{\omega\in \Omega}$ be a trivial Rota-Baxter family BiHom-$\Omega$-associative algebra. Then
	\begin{enumerate}
		\item\label{it:ext-coho} two isomorphic abelian extensions of $A$ by $M$ give rise to the same cohomology class in $H^2_{\mathrm{RBFA}_{\lambda}}(A,M)$.
		\item\label{it:coho-ext} two cohomologous 2-cocycles give rise to isomorphic abelian extensions.
	\end{enumerate}
\end{theorem}
\begin{proof}
	\ref{it:ext-coho}. Let $E=(A\oplus M,\mu_{\alpha,\,\beta}^E,T_{\omega}^E,p_{\omega},q_{\omega})_{\alpha,\,\beta,\,\omega\in \Omega}$ and $\bar{E}=(A\oplus M,\bar{\mu}_{\alpha,\,\beta}^E,\bar{T}_{\omega}^E,p_{\omega},q_{\omega})_{\alpha,\,\beta,\,\omega\in \Omega}$ be two isomorphic abelian extensions of $A$ by $M$ and let $( s_{\alpha}^1)_{\alpha\in \Omega}$ be a section of $E$. For any $\alpha,\,\beta,\,\omega\in \Omega$, by Definition~\ref{def-isoext}, we have
		\[\rho_{\alpha}^2\circ(\phi_{\alpha}\circ s_{\alpha}^1)=(\rho_{\alpha}^2\circ \phi_{\alpha})\circ s_{\alpha}^1=\rho_{\alpha}^1\circ s_{\alpha}^1=id_{A}.\]
		That is, $\phi_{\alpha}\circ s_{\alpha}^1$ is a section of $\rho_{\alpha}^2$, so we denote $s_{\alpha}^2\overset{\bigtriangleup}{=}\phi_{\alpha}\circ s_{\alpha}^1$. For the bimodule structure on $M$, we have
		\begin{align*}
			\phi_{\alpha\,\beta}(a\rhd_{\alpha,\,\beta}m)=&\phi_{\alpha\,\beta}\mu_{\alpha,\,\beta}^E\big(s_{\alpha}^1(a),\, i_{\beta}^1(m)\big)\\
			=&\bar{\mu}_{\alpha,\,\beta}^E\big(\phi_{\alpha}s_{\alpha}^1(a),\, \phi_{\beta}i_{\beta}^1(m)\big)\hspace{1cm}(\text{by $\phi_{\alpha}$ satisfying Eq.~(\ref{RBfBHOalghomo})})\\
			=&\bar{\mu}_{\alpha,\,\beta}^E\big(\phi_{\alpha}s_{\alpha}^1(a),\, i_{\beta}^2(m)\big)\hspace{1cm}(\text{by }\phi_{\beta}\circ i_{\beta}^1=i_{\beta}^2)\\
			=&a\rhd_{\alpha,\,\beta}m.
		\end{align*}
	So, we get $\phi_{\alpha}|_{M}=id_{M}$. By Eqs.~(\ref{def-eq-psi})-(\ref{def-eq-chi}) and Proposition~\ref{prop-psichi-2-cocycle}, let $(\psi_{\alpha,\,\beta}^1,\chi_{\omega}^1)_{\alpha,\,\beta,\,\omega\in \Omega}$ and $(\psi_{\alpha,\,\beta}^2,\chi_{\omega}^2)_{\alpha,\,\beta,\,\omega\in \Omega}$ be two 2-cocycles corresponding to abelian extension $E$ and $\bar{E}$, respectively, then we have
	\begin{align*}
		\psi_{\alpha,\,\beta}^2(a,b)=&\bar{\mu}_{\alpha,\,\beta}^E\big(s_{\alpha}^2(a),\, s_{\beta}^2(b)\big)-s_{\alpha\,\beta}^2\mu_{\alpha,\,\beta}(a,b)\\
		=&\bar{\mu}_{\alpha,\,\beta}^E\big(\phi_{\alpha} s_{\alpha}^1(a),\, \phi_{\beta}s_{\beta}^1(b)\big)-\phi_{\alpha\,\beta}s_{\alpha\,\beta}^1\mu_{\alpha,\,\beta}(a,b)\\
		=&\phi_{\alpha\,\beta}\Big(\mu_{\alpha,\,\beta}^E\big(s_{\alpha}^1(a),\, s_{\beta}^1(b)\big)-s_{\alpha\,\beta}^1\mu_{\alpha,\,\beta}(a,b)\Big)\\
		&\hspace{1cm}(\text{by Eq.~(\ref{RBfBHOalghomo}) and }\phi_{\alpha\,\beta}\mu_{\alpha,\,\beta}^E=\bar{\mu}_{\alpha,\,\beta}^E(\phi_{\alpha}\otimes \phi_{\beta}))\\
		=&\phi_{\alpha\,\beta}\psi_{\alpha,\,\beta}^1(a,b)\\
		=&\psi_{\alpha,\,\beta}^1(a,b).\hspace{1cm}(\text{by }\phi_{\alpha}|_{M}=id_{M})
	\end{align*}
Similarly, we get $\chi_{\omega}^2(a)=\chi_{\omega}^1(a).$ So, $(\psi_{\alpha,\,\beta}^1,\chi_{\omega}^1)_{\alpha,\,\beta,\,\omega\in \Omega}$ and $(\psi_{\alpha,\,\beta}^2,\chi_{\omega}^2)_{\alpha,\,\beta,\,\omega\in \Omega}$ correspond to the same element in $\mathrm{H}_{\mathrm{RBFA}_{\lambda}}^2(A,M)$. 

\ref{it:coho-ext}. Let $(\psi_{\alpha,\,\beta}^1,\chi_{\omega}^1)_{\alpha,\,\beta,\,\omega\in \Omega}$ and $(\psi_{\alpha,\,\beta}^2,\chi_{\omega}^2)_{\alpha,\,\beta,\,\omega\in \Omega}$ be two 2-cocycles. By Lemma~\ref{lemma} and Eqs.~(\ref{def-mu-psi})-(\ref{def-q}), we know that $(A\oplus M,\mu_{\alpha,\,\beta}^{\psi^1},T_{\omega}^{\chi^1},p_{\omega},q_{\omega})_{\alpha,\,\beta,\,\omega\in \Omega}$ and $(A\oplus M, \mu_{\alpha,\,\beta}^{\psi^2},T_{\omega}^{\chi^2},p_{\omega},q_{\omega})_{\alpha,\,\beta,\omega\in \Omega}$ are their corresponding abelian extensions, respectively. If $(\psi_{\alpha,\,\beta}^1,\chi_{\omega}^1)_{\alpha,\,\beta,\,\omega\in \Omega}$ and $(\psi_{\alpha,\,\beta}^2,\chi_{\omega}^2)_{\alpha,\,\beta,\,\omega\in \Omega}$ have the same cohomology class in $\mathrm{H}_{\mathrm{RBFA}_{\lambda}}^2(A,M)$, then there exist two families of linear maps $(\eta_{\alpha}^0)_{\alpha\in\Omega}: \mathbf{k}\rightarrow M$ and $(\eta_{\alpha}^1)_{\alpha\in\Omega}: A\rightarrow M$ satisfy
\[(\psi_{\alpha,\,\beta}^1,\chi_{\omega}^1)=(\psi_{\alpha,\,\beta}^2,\chi_{\omega}^2)+\big(\delta_{\mathrm{Alg}}^1(\eta^1)_{\alpha,\,\beta},\, -\partial^0(\eta^0)_{\omega}-\Phi^1(\eta^1)_{\omega}\big),\quad \text{for all }\alpha,\,\beta,\,\omega\in\Omega.\]
Then, we define a family of linear maps $( \phi_{\alpha})_{\alpha\in \Omega}: A\oplus M\rightarrow A\oplus M$ by
\[\phi_{\alpha}(a,m):=\Big(a,\,\big( \eta_{\alpha}^1+\delta_{\mathrm{Alg}}^0(\eta^0)_{\alpha}\big)(a)+m\Big),\quad \text{for all }(a,m)\in A\oplus M,\,\alpha\in\Omega.\]
We can easily verify that $(\phi_{\alpha})_{\alpha\in \Omega}$ is a Rota-Baxter family BiHom-$\Omega$-associative algebra isomorphism from $(A\oplus M,\mu_{\alpha,\,\beta}^{\psi^1},T_{\omega}^{\chi^1},p_{\omega},q_{\omega})_{\alpha,\,\beta,\,\omega\in \Omega}$ to $(A\oplus M, \mu_{\alpha,\,\beta}^{\psi^2},T_{\omega}^{\chi^2},p_{\omega},q_{\omega})_{\alpha,\,\beta,\omega\in \Omega}$ and satisfies 
\[\phi_{\alpha}\circ i_{\alpha}^1=i_{\alpha}^2,\quad \rho_{\alpha}^1=\rho_{\alpha}^2\circ \phi_{\alpha},\quad \text{for all }\alpha\in\Omega.\]
Thus, $(A\oplus M,\mu_{\alpha,\,\beta}^{\psi^1},T_{\omega}^{\chi^1},p_{\omega},q_{\omega})_{\alpha,\,\beta,\,\omega\in \Omega}$ and $(A\oplus M, \mu_{\alpha,\,\beta}^{\psi^2},T_{\omega}^{\chi^2},p_{\omega},q_{\omega})_{\alpha,\,\beta,\omega\in \Omega}$ are isomorphic. This completes the proof.
\end{proof}

\smallskip
\noindent
{\bf Acknowledgments.} This work is supported by Natural Science Foundation of China (12101183). Y. Y. Zhang is also supported by the Postdoctoral Fellowship Program of CPSF under Grant Number (GZC20240406).

\smallskip
\noindent
{\bf Statements and Declarations:}
All datasets underlying the conclusions of the paper are available to readers. No conflict of
interest exists in the submission of this manuscript.



\end{document}